\documentclass[a4paper, reqno, 11pt]{amsart}        

\usepackage{amssymb}

\usepackage{epsfig,bbm}  		
\usepackage{epic,eepic}       

\let\oldtocsection=\tocsection
 
\let\oldtocsubsection=\tocsubsection
 
\let\oldtocsubsubsection=\tocsubsubsection
 
\renewcommand{\tocsection}[2]{\hspace{0em}\oldtocsection{#1}{#2}}
\renewcommand{\tocsubsection}[2]{\hspace{1em}\oldtocsubsection{#1}{#2}}
\renewcommand{\tocsubsubsection}[2]{\hspace{2em}\oldtocsubsubsection{#1}{#2}}

\usepackage{amsfonts}
\usepackage{amsthm}
\usepackage{amsmath}
\usepackage{amscd}
\usepackage[latin2]{inputenc}
\usepackage{t1enc}
\usepackage[mathscr]{eucal}
\usepackage{indentfirst}
\usepackage{graphicx}
\usepackage{graphics}
\numberwithin{equation}{section}
\usepackage[margin=2.9cm]{geometry}
\usepackage{hyperref}
\usepackage{dsfont}
\usepackage{csquotes}
\usepackage{stmaryrd}
\usepackage{blindtext}
\usepackage{tikz-cd}
\usepackage{mathrsfs}
\usepackage{cite}
\usepackage{mathtools}
\usepackage{extarrows}
\usepackage{epstopdf}
\usepackage[T1]{fontenc}
\usepackage{braket}
\usepackage{yfonts}

\theoremstyle{definition}
\newtheorem{definition}[equation]{Definition}
\newtheorem{example}[equation]{Example}
\newtheorem{proposition}[equation]{Proposition}
\newtheorem{theorem}[equation]{Theorem}
\newtheorem{remark}[equation]{Remark}

\newtheorem{corollary}[equation]{Corollary}
\newtheorem{lemma}[equation]{Lemma}
\newtheorem{innercustomgeneric}{\customgenericname}
\providecommand{\customgenericname}{}
\newcommand{\newcustomtheorem}[2]{%
  \newenvironment{#1}[1]
  {%
   \renewcommand\customgenericname{#2}%
   \renewcommand\theinnercustomgeneric{##1}%
   \innercustomgeneric
  }
  {\endinnercustomgeneric}
}

\newcustomtheorem{customthm}{Theorem}
\newcustomtheorem{customlemma}{Lemma}
\newcustomtheorem{customprop}{Proposition}

\theoremstyle{remark}

\numberwithin{equation}{section}

\newcommand{\midwedge}{\text{\Large$\wedge$}}
\newcommand{\midodot}{\text{\Large$\odot$}}
\newcommand{\midotimes}{\text{\Large$\otimes$}}

\newcommand{\be}{\begin{equation}}
\newcommand{\ee}{\end{equation}}
\def\beqa{\begin{eqnarray}}
\def\eeqa{\end{eqnarray}}
\def\bean{\begin{eqnarray*}}
\def\eean{\end{eqnarray*}}

\newcommand{\de}{\mathrm{d}}

\newcommand{\del}{\partial}

\newcommand{\e}{\mathrm{e}}

\newcommand{\IZ}{\mathbb{Z}}

\newcommand{\IR}{\mathbb{R}}

\newcommand{\IT}{\mathbb{T}}

\newcommand{\IIX}{\mathbbm{X}}

\newcommand{\frg}{\mathfrak{g}}

\newcommand{\frh}{\mathfrak{h}}

\newcommand{\frd}{\mathfrak{d}}

\newcommand{\frt}{\mathfrak{t}}
\newcommand{\frk}{\mathfrak{k}}

\def\e{{\,\rm e}\,}

\newcommand{\cB}{{\mathcal B}}

\newcommand{\cH}{{\mathcal H}}

\newcommand{\cQ}{{\mathcal Q}}

\newcommand{\cF}{{\mathcal F}}
\newcommand{\cD}{{\mathcal D}}

\newcommand{\cV}{{\mathcal V}}

\newcommand{\cG}{{\mathcal G}}

\newcommand{\ccK}{{\mathscr K}}

\newcommand{\ccV}{{\mathscr V}}

\newcommand{\sfD}{{\mathsf{D}}}
\newcommand{\sfS}{{\mathsf{S}}}

\newcommand{\sfT}{{\mathsf{T}}}

\newcommand{\sfH}{{\mathsf{H}}}
\newcommand{\sfG}{{\mathsf{G}}}

\newcommand{\sfL}{{\mathsf{L}}}

\newcommand{\sfGamma}{{\mathsf{\Gamma}}}

\newcommand{\unit}{\mathds{1}}   			

\usepackage{enumitem}

\newcommand{\ip}[1]{\left\langle #1 \right\rangle}
\newcommand{\gr}{\text{gr}}
\newcommand{\ann}{\mathrm{Ann}}
\newcommand{\red}[1]{\underline{#1}}
\newcommand{\rel}{\dashrightarrow}
\newcommand{\mor}{\rightarrowtail}

\newcommand{\cal}[1]{\mathcal{#1}}
\newcommand{\rk}{\mathrm{rk}}

\renewcommand{\phi}{\varphi}

\newtheoremstyle{case}{}{}{}{}{}{:}{ }{}
\theoremstyle{remark}

\makeatletter
\newcommand{\oset}[3][0ex]{%
  \mathrel{\mathop{#3}\limits^{
    \vbox to#1{\kern-3\ex@
    \hbox{$\scriptstyle#2$}\vss}}}}
\makeatother

\makeatletter
\newcommand{\osett}[3][0ex]{%
  \mathrel{\mathop{#3}\limits^{
    \vbox to#1{\kern-2\ex@
    \hbox{$\scriptstyle#2$}\vss}}}}
\makeatother

\makeatletter
\newcommand{\osettt}[3][0ex]{%
  \mathrel{\mathop{#3}\limits^{
    \vbox to#1{\kern-1\ex@
    \hbox{$\scriptstyle#2$}\vss}}}}
\makeatother

\newcommand{\pppp}{\osett{+}{+}\osett{+}{+}}
\newcommand{\pppm}{\osett{+}{+}\osett{+}{-}}
\newcommand{\pmpp}{\osett{+}{+}\osett{-}{+}}
\newcommand{\pmpm}{\osett{+}{+}\osett{-}{-}}
\newcommand{\mmmm}{\osett{-}{-}\osett{-}{-}}
\newcommand{\mmmp}{\osett{-}{-}\osett{-}{+}}
\newcommand{\mpmm}{\osett{-}{-}\osett{+}{-}}
\newcommand{\mpmp}{\osett{-}{-}\osett{+}{+}}

\makeatletter

\DeclareRobustCommand\mathflip[1]{%
    \mathpalette\@mathflip{#1}%
}

\newcommand\@mathflip[2]{%
    \mskip4mu
    \pdfsave
    \pdfsetmatrix{-1 0 .5 1}%
    \hb@xt@\z@{\hss$\m@th#1 #2$\hss}%
    \pdfrestore
    \mskip7mu
}

\DeclareRobustCommand\mathflipnu[1]{%
    \mathpalette\@mathflipnu{#1}%
}

\newcommand\@mathflipnu[2]{%
    \mskip2mu
    \pdfsave
    \pdfsetmatrix{-1 0 .5 1}%
    \hb@xt@\z@{\hss$\m@th#1 #2$\hss}%
    \pdfrestore
    \mskip8mu
}

\makeatother

\newcommand{\um}{\text{$\mathflip{\mu}$}}
\newcommand{\un}{\text{$\mathflipnu{\nu}$}}
\newcommand{\ahpla}{\text{$\mathflip{\alpha}$}}
\newcommand{\ateb}{\text{$\mathflip{\beta}$}}
\newcommand{\flipm}[1]{\text{$\mathflip{#1}$}}

\mathtoolsset{showonlyrefs}

\begin{document}


\title[T-Dualities and Courant Algebroid Relations]{T-Dualities and Courant Algebroid Relations}

\author[T.~C.~De Fraja]{Thomas C.~De Fraja}
\address[Thomas C.~De Fraja]
{Department of Mathematics and Maxwell Institute for Mathematical
  Sciences\\ Heriot-Watt
  University\\ Edinburgh EH14 4AS\\ United Kingdom}
\email{tcd2000@hw.ac.uk}

\author[V.~E.~ Marotta]{Vincenzo Emilio Marotta}
\address[Vincenzo Emilio Marotta]
{Mathematical Institute, Faculty of Mathematics and Physics\\ Charles University\\ Prague 186 75\\ Czech Republic and
Universit\`{a} di Trieste, Dipartimento di Matematica e Geoscienze,  Via A. Valerio 12/1, 34127 Trieste, Italy
}
\email{vincenzoemilio.marotta@units.it}

\author[R.~J. Szabo]{Richard J.~Szabo}
  \address[Richard J.~Szabo]
  {Department of Mathematics, Maxwell Institute for Mathematical Sciences and Higgs Centre for Theoretical Physics\\
  Heriot-Watt University\\
  Edinburgh EH14 4AS \\
  United Kingdom}
  \email{R.J.Szabo@hw.ac.uk}

\vfill

\begin{flushright}
\footnotesize
{\sf EMPG--23--12}
\normalsize
\end{flushright}

\vspace{1cm}

\begin{abstract}
We develop a new approach to T-duality based on Courant algebroid relations which subsumes the usual T-duality as well as its various generalisations.
Starting from a relational description for the reduction of exact Courant algebroids over foliated manifolds, we introduce a weakened notion of generalised isometries that captures the generalised geometry counterpart of Riemannian submersions when applied to transverse generalised metrics. This is used to construct T-dual backgrounds as generalised metrics on reduced Courant algebroids which are related by a generalised isometry. We prove an existence and uniqueness result for generalised isometric exact Courant algebroids coming from reductions. We demonstrate that our construction reproduces standard T-duality relations based on correspondence spaces. We also describe how it applies to generalised T-duality transformations of almost para-Hermitian manifolds. 
 \end{abstract} 

\maketitle

{\baselineskip=12pt
\tableofcontents
}

\section{Introduction}
\label{sec:Intro}
Courant algebroids, as introduced in \cite{Courant1990, Weinstein1997, Severa-letters, Hitchin:2003cxu}, have  proved to be an indispensable geometric tool for understanding many aspects of supergravity and string theory, see e.g.~\cite{Coimbra:2011nw,Jurco2016courant} and references therein. In this paper we are interested in their role in capturing the geometric, topological, and physical properties of spaces related by T-duality symmetries of string theory, see e.g.~\cite{Severa-letters, Grana2008, cavalcanti2011generalized, Baraglia:2013wua, Severa2015, Jurco2018, Severa:2018pag, Vysoky2020hitchiker}.
Finalising a fully geometric description of T-duality and its various generalisations, such as non-abelian T-duality, Poisson-Lie T-duality, and non-isometric T-duality, has proven to be an elusive task. In this paper we will build on the geometric picture of Poisson-Lie T-duality proposed by Vysok\'y in \cite{Vysoky2020hitchiker} to give a generalised and unified framework for T-dualities in terms of Courant algebroid relations.

Let us start by explaining why T-duality is generally expected to be subsumed into the geometry of Courant algebroids from the perspective of the worldsheet formulation of string theory, which is the point of view taken throughout this paper. Following~\cite[Section 2.2]{Garcia-Fernandez:2020ope}, we describe the intimate relationship between two-dimensional sigma-models and Courant algebroids, originally stated in \cite{Severa-letters}.
 
\medskip

\subsection{Exact Courant Algebroids from Sigma-Models}~\\[5pt] \label{subsect:sigma}
The background data for the bosonic part of the worldsheet theory of closed oriented strings consists of a closed oriented Riemann surface $(\Sigma,h)$, a Riemannian manifold $(M,g)$ and a closed three-form $H \in \mathsf{\Omega}^3_{\rm cl} (M)$, called an \emph{$H$-flux}, which represents an integer cohomology class $[H] \in \mathsf{H}^3(M, \IZ)$. The string sigma-model is a field theory of smooth maps $\IIX\in C^\infty(\Sigma,M)$, whose action functional is a sum of two terms. 

The kinetic term is given by a Dirichlet-type functional, called the \emph{Polyakov functional.} It is defined by endowing the space of maps $\de \IIX \colon T\Sigma \rightarrow \IIX^*TM$ with a metric induced by
$g$, regarded as a metric on the vector space of sections of the pullback  $\IIX^*TM$ of the tangent bundle $TM$ to $\Sigma$ by $\IIX,$ and the cometric $h^{-1}$ on $T^*\Sigma.$  This gives
a well-defined metric 
$h^{-1} \otimes \IIX^*g$ on the vector bundle $T^*\Sigma \otimes \IIX^*TM$ over $\Sigma$ which allows one to write the Polyakov functional as
\be \label{eqn:sigmanorm}
S_0[\IIX]=\frac{1}{2}\,\int_{\Sigma}\, h^{-1}\big(\IIX^*g (\de \IIX, \de \IIX)\big) \ \de\mu(h) \ ,
\ee
where $\de \IIX \in \mathsf{\Gamma}(T^*\Sigma \otimes \IIX^*TM)$  and $\mu(h)$ is the area measure on $\Sigma$ induced by $h$. 

The topological term, called the \emph{Wess-Zumino functional}, is the functional $$S_H \colon C^\infty(\Sigma, M) \longrightarrow \IR / \IZ$$ defined by 
\be \label{eqn:ActionQTop}
S_H [\IIX] \coloneqq \int_V\, \mathbbm{X}^*_V H \ ,
\ee
where $V$ is any three-manifold with boundary $\Sigma$, and $\IIX_V \colon V \rightarrow M$ is any smooth extension of $\IIX \in C^\infty(\Sigma, M)$ to $V.$ The space of Lagrangian densities of the Wess-Zumino functional $S_H$ is~$\mathsf{\Omega}_{\rm cl}^3(M).$

Consider the variational problem for the topological term. A variation of the Wess-Zumino functional \eqref{eqn:ActionQTop} is generated by a vector field $X \in \mathsf{\Gamma}(TM)$ acting via the Lie derivative. Since $H$ is closed, via Stokes' Theorem this is given by
\begin{align}
\delta_X S_H [\IIX] = \int_\Sigma\, \IIX^* \big(\iota_X H\big) \ .
\end{align}
The solutions of the variational problem  $\delta_XS_H[\IIX]=0$ (see e.g.~\cite{Severa-letters, Severa2015, Garcia-Fernandez:2020ope}) 
are given by the maps $\IIX \in C^\infty(\Sigma, M)$ such that, for all $X \in \mathsf{\Gamma}(TM), $ there exists $\bar{\alpha} \in \mathsf{\Gamma}(T^*\Sigma)$ satisfying
\begin{align} \label{eqn:criticalpoints}
\IIX^* \bigl( \iota_X H \bigr) = \de \bar{\alpha} \ .
\end{align}

Let us introduce the vector bundle $$\IT M \coloneqq TM \oplus T^*M \ ,$$ called the {double tangent bundle} of $M$. Given a section $X + \alpha \in \mathsf{\Gamma}(\IT M),$ Equation \eqref{eqn:criticalpoints} is then satisfied if 
\begin{align} \label{eqn:constantSH}
\iota_X H = \de \alpha \ ,
\end{align}
where Equation \eqref{eqn:criticalpoints} is obtained by pulling back Equation~\eqref{eqn:constantSH} by $\IIX.$ Equation \eqref{eqn:constantSH} can be interpreted as giving the tangent directions on $C^\infty(\Sigma, M)$ along which $S_H$ is constant. 
We shall now show that these flat directions are related to symmetries of the Wess-Zumino functional \eqref{eqn:ActionQTop}.

Following \cite{Garcia-Fernandez:2020ope}, we define the action of ${\sf Diff}(M)$ on $C^\infty(\Sigma, M)$ by $(\phi, \IIX) \to \phi^{-1}\circ \IIX$. This group action maps the functional \eqref{eqn:ActionQTop} to 
    $S_H(\phi^{-1} \circ \IIX) = S_{(\phi^{-1})^*H}(\IIX)$, hence the action by $\phi$ is a symmetry of the functional \eqref{eqn:ActionQTop} if and only if $\phi^* H=H$. However, this is not the full group of symmetries of $S_H(\IIX).$
The Lie group $$\sfG \coloneqq \mathsf{Diff}(M) \ltimes \mathsf{\Omega}^2 (M)$$ 
acts on the space of Lagrangian densities $\mathsf{\Omega}^3_{\rm cl}(M)$ by 
\begin{align} \label{eqn:actionSemi}
(\phi,B) \cdot H = (\phi^{-1})^* (H - \de B) =: H' \ ,
\end{align}
for all $(\phi, B) \in \sfG$ and $H \in \mathsf{\Omega}^3_{\rm cl}(M).$ 
Under the group action \eqref{eqn:actionSemi}, the Wess-Zumino functional $S_H$ transforms as
\begin{align}
S_{H^\prime}[\IIX]= S_H[\phi^{-1} \circ \IIX] - \int_\Sigma \, \IIX^* \big((\phi^{-1})^* B\big) \ ,
\end{align}
where $H^\prime = (\phi^{-1})^* (H - \de B).$ Thus, the functional \eqref{eqn:ActionQTop} is invariant if and only if and $\phi^*H = H - \de B$. Such pairs $(\varphi,B)$ define the isotropy subgroup $\sfG_H \subset \sfG$ of the $H$-flux, see \cite[Proposition 2.23]{Garcia-Fernandez:2020ope}.

Let $\mathsf{Lie}(\sfG_H) \subset \mathsf{\Gamma}\big(T M \oplus \midwedge^2\, T^*M\big)$ be the Lie algebra of the isotropy group $\sfG_H$. A representation of $\sfG$ is given by the normal Lie subalgebra of $\mathsf{Lie}(\sfG_H) $
\begin{align}
\mathsf{Lie}(\sfG_H)^\mathtt{n} = \set{ X + \de \alpha - \iota_X H \, | \, X \in \mathsf{\Gamma}(TM) \ , \ \alpha \in \mathsf{\Gamma}(T^*M) } \ ,
\end{align}
which is parametrised by elements in $\sfGamma(\IT M)$.
See \cite[Proposition 2.24]{Garcia-Fernandez:2020ope} for the proof.

The Lie group $\sfG$ acts on $\mathsf{\Gamma}(\IT M)$ by
\be \label{eqn:GITTM}
(\phi, B) \cdot (X+\alpha) = \phi_*X + (\phi^{-1})^*(\alpha + \iota_X B) \ ,
\ee 
for all $(\phi, B) \in \sfG$ and $X + \alpha \in \mathsf{\Gamma}(\IT M).$ Let $\Psi \colon \mathsf{\Gamma}(\IT M) \rightarrow \mathsf{Lie}(\sfG_H)^\mathtt{n}$ be the $\sfG_H$-equivariant map defined by
\be\label{eqn:ghequivariantmap}
\Psi(X +\alpha) = X + \de \alpha - \iota_X H \ ,
\ee
for all $X + \alpha \in \mathsf{\Gamma}(\IT M).$ The Dorfman bracket $\llbracket  \, \cdot \, , \, \cdot \, \rrbracket_H$ is obtained from the Lie algebra action 
induced by the restriction of the representation of $\sfG$ on $\mathsf{\Gamma}(\IT M)$ to $\mathrm{im}(\Psi)\subset \mathsf{Lie}(\sfG)$:
\be 
\Psi(X +\alpha) \cdot (Y +\beta) = [X,Y] + \pounds_X \beta - \iota_Y\, \de \alpha + \iota_Y\, \iota_X H \eqqcolon \llbracket X + \alpha, Y + \beta \rrbracket_H \ ,
\ee
for all $X + \alpha, Y + \beta \in \mathsf{\Gamma}(\IT M).$\footnote{Notation: $[X,Y]$ denotes the Lie bracket of vector fields $X,Y\in\mathsf{\Gamma}(TM)$, $\pounds_X$ denotes the Lie derivative along $X$, and $\iota_{(\,\cdot\,)}$ generally denotes the contraction of a section of a vector bundle with a tensor of the dual vector bundle.} The Jacobi identity for  $\llbracket  \, \cdot \, , \, \cdot \, \rrbracket_H$ is a consequence of $\sfG_H$-equivariance of $\Psi.$

The vector bundle $\IT M$ endowed with the bracket $\llbracket  \, \cdot \, , \, \cdot \, \rrbracket_H$ (and a suitable pairing) on its sections is our first example of a (exact) Courant algebroid. We will see later on that its automorphism group is the same as the isotropy group $\sfG_H \subset \sfG$ of $H$ for the action \eqref{eqn:actionSemi} on $\mathsf{\Omega}^3_{\rm cl}(M)$, i.e. the symmetries leaving $S_H$ invariant (see Corollary \ref{cor:CAauts}). Incorporating the kinetic term \eqref{eqn:sigmanorm} amounts to introducing a generalised metric on $\IT M$. We conclude that, from the worldsheet perspective, the data of a string background is encoded in an exact Courant algebroid endowed with a generalised metric.

Courant algebroids may be regarded as a generalisation of quadratic Lie algebras from vector spaces to vector bundles, i.e. as vector bundles endowed with a bracket operation on their module of sections and a non-degenerate symmetric bilinear pairing satisfying a kind of invariance property with respect to the bracket operation. The crux to understanding T-dualities in this framework resides in the complexity of the description of morphisms of these structures. 

As we have seen above, isomorphisms of Courant algebroids provide a symmetry of the corresponding sigma-model, and they prove to be vastly well-behaved. 
Meanwhile, the notion of T-duality for sigma-models goes beyond the concept of symmetry, and more powerful tools are needed to explore the relation between T-dual sigma-models.
In this paper we advocate the use of Courant algebroid relations which, while less constraining than isomorphisms, present subtle and more complex features that before the work of Vysok\'y \cite{Vysoky2020hitchiker} were not fully understood. This conundrum motivates a discussion of what a category-like notion for Courant algebroids might be. 

\medskip

\subsection{The ``Category'' of Courant Algebroids}~\\[5pt]
As argued in \cite{Vysoky2020hitchiker}, an extended notion of the category of Courant algebroids can be introduced by allowing morphisms of Courant algebroids to be involutive isotropic subbundles of a product Courant algebroid supported on a submanifold of the base; these are called \emph{Courant algebroid relations}. By introducing this notion one avoids the problem of not having enough arrows. On the other hand, extending the notion of arrow between Courant algebroids in this way means that not all arrows are composable: the support of the composition of subbundles might fail to be a submanifold itself. Nonetheless, the notion of Courant algebroid relation provides a complete characterisation of what Courant algebroid morphisms are and how their composition behaves.

In \cite{Vysoky2020hitchiker} a well-defined notion of relations between generalised metrics on Courant algebroids is further presented, i.e.~a \emph{generalised isometry} which is also well-behaved under composition of relations. Based on this, in the present paper we will extend the notion of generalised isometry to structures similar to generalised metrics: the  \emph{transverse generalised metrics} which include in their definition a controlled degeneracy that, in turn, helps in defining their invariance conditions. 

This extension also arises from a more concrete problem. As shown in \cite{Bursztyn2007reduction}, an exact $\sfG$-equivariant Courant algebroid, over a principal $\sfG$-bundle, can be reduced to a Courant algebroid on the base of the principal bundle. For isotropic $\sfG$-actions, the reduced Courant algebroid is exact as well. As shown in \cite{Vysoky2020hitchiker}, there exists a Courant algebroid relation between the original Courant algebroid and the reduced Courant algebroid. However, if both are endowed with generalised metrics, this Courant algebroid relation always fails to be a generalised isometry if the $\sfG$-action is isotropic. Our goal in this paper is to find a notion of isometry between a suitable structure on the original Courant algebroid, e.g. a transverse generalised metric, and a generalised metric on the reduced Courant algebroid.

Courant algebroid relations pave the way for a complete formulation of T-duality as a generalised isometry. This idea was first proposed by \v{S}evera\footnote{See \url{http://thphys.irb.hr/dualities2017/files/Jun06Severa.pdf}.} in order to describe Poisson-Lie T-duality as an ``almost isomorphism'' of Hamiltonian systems \cite{Severa-unpubl}. The almost isomorphism arises from a Courant algebroid relation that is required to be both Lagrangian -- i.e. a Dirac structure for the product Courant algebroid -- and a generalised isometry. This is explored in detail in \cite{Vysoky2020hitchiker}, where Poisson-Lie T-duality is shown to be a generalised isometry between reduced Courant algebroids coming from different group actions which are not isotropic.
However, there is no general geometric framework encompassing T-duality in all of its flavours. In this paper we will provide such framework for T-dualities arising from generic isotropic Courant algebroid reductions. A particularly relevant instance that we are interested in is the T-duality for torus bundles with $H$-flux discussed in \cite{Bouwknegt2003topology, bouwknegt2004tduality, Bunke2005, cavalcanti2011generalized}. This has also recently been discussed in~\cite{Waldorf:2022tib} from a more general geometric perspective which is closer in spirit to our treatment; however, our framework also covers more general cases of affine torus bundles~\cite{Baraglia2014}.

\medskip

\subsection{T-duality and the Fourier-Mukai Transform}~\\[5pt]
In the picture proposed by \cite{Bouwknegt2003topology} and later expanded in \cite{cavalcanti2011generalized}, T-duality is accompanied by a module isomorphism between invariant sections of Courant algebroids over torus bundles with a common base. It arises from a (smooth version of the) Fourier-Mukai transform through a correspondence between these bundles. One of our applications in this paper is the description of this form of T-duality in terms of Courant algebroid relations. Let us therefore discuss this approach to T-duality in more depth. 

The setting of \cite{Bouwknegt2003topology, cavalcanti2011generalized} involves two principal $\sfT^k$-bundles $\cQ_1$ and $\cQ_2$ over a common base manifold $\cB$. Let $M = \cQ_1 \times_\cB \cQ_2$ be the fibred product, with respective projections $\varpi_1$ and $\varpi_2$ to $\cQ_1$ and $\cQ_2$. This is a principal $\sfT^{2k}$-bundle over $\cB$, called a correspondence space or a doubled torus bundle, which sits in the commutative diagram 
\begin{equation}\label{eq:correspondence}
\begin{tikzcd}
 & M \arrow[dl,"\varpi_1",swap] \arrow[dr,"\varpi_2"] & \\
 \cQ_1 \arrow[dr] &  & \cQ_2 \arrow[dl] \\ 
 & \cB &
\end{tikzcd}
\end{equation}
On $M$ the fibrewise T-duality group acts geometrically via diffeomorphisms~\cite{Hull2005,Belov:2007qj}. 

Suppose that both $\cQ_1$ and $\cQ_2$ are endowed with closed $\sfT^k$-invariant three-forms $\red H{}_1$ and $\red H{}_2,$ respectively.
Then $(\cQ_1,\red H{}_1)$ and $(\cQ_2,\red H{}_2)$ are said to be T-dual if 
\begin{align}
 \varpi_1^*\, \red H{}_1 - \varpi_2^*\, \red H{}_2 = \de B 
\end{align}
on the correspondence space $M$, for some $\sfT^{2k}$-invariant two-form \smash{$B \in \mathsf{\Omega}^2_{\sfT^{2k}}(\cQ_1 \times_\cB \cQ_2)$} whose restriction to $\ker(\varpi_{1*}) \otimes \ker(\varpi_{2 *})$ is non-degenerate. Gauging the string sigma-model for the doubled torus bundle $M$ then relates the sigma-models for the quotient spaces $\cQ_1$ and $\cQ_2$ via reduction along the projections $\varpi_1$ and $\varpi_2$, respectively~\cite{Hull2007}.

In this case, since the fibres are compact, the Fourier-Mukai transform is well-defined and given by
\begin{align}
 \varrho(\alpha) \coloneqq \int_{\sfT^k}\, \e^{B}\, \wedge \varpi_1^*\,\alpha  
\end{align}
for any \smash{$\alpha \in \mathsf{\Omega}^\bullet_{\sfT^k}(\cQ_1)$}, where the fibrewise integration is the pushforward of forms by the projection $\varpi_{2}:M\to\cQ_2$. It  is shown in~\cite{Bouwknegt2003topology} that $\varrho$ defines a degree-shifting isomorphism between the twisted differential complexes \smash{$\big(\mathsf{\Omega}^\bullet_{\sfT^k}(\cQ_1) , \de_{\red H{}_1}\big)$} and 
    \smash{$\big(\mathsf{\Omega}^\bullet_{\sfT^k}(\cQ_2) , \de_{\red H{}_2}\big),$} where $\de_{\red H{}_i}:=\de+\red H{}_i\,\wedge\,$. It describes the transformation of Ramond--Ramond fields in type~II string theory under T-duality.
    
This becomes an isomorphism of irreducible Clifford modules by choosing an isomorphism~\cite{cavalcanti2011generalized} $$\mathscr{R} \colon \mathsf{\Gamma}_{\sfT^k}(T\cQ_1 \oplus T^*\cQ_1) \longrightarrow \mathsf{\Gamma}_{\sfT^k}(T\cQ_2 \oplus T^*\cQ_2)$$ of $C^\infty(\cB)$-modules such that
\begin{align}
 \varrho(e \cdot \alpha) = \mathscr{R}(e) \cdot \varrho(\alpha) \ ,   
\end{align}
for any \smash{$e \in \mathsf{\Gamma}_{\sfT^k}(T\cQ_1 \oplus T^*\cQ_1)$} and \smash{$\alpha \in \mathsf{\Omega}^\bullet_{\sfT^k}(\cQ_1),$} where $e \cdot \alpha$ is the natural representation on $\mathsf{\Omega}^\bullet(\cQ_1)$ of the Clifford algebra induced by the pairing $\ip{\,\cdot\,,\, \cdot\,}$ on the fibres of $T \cQ_1 \oplus T^*\cQ_1$, and similarly on \smash{$\mathsf{\Omega}^\bullet(\cQ_2)$}. 
Thus T-duality, as an isomorphism of $\sfT^k$-invariant sections of the Courant algebroids $\IT \cQ_1$ and $\IT \cQ_2$, is given by the unique choice of $\mathscr{R}$ in terms of horizontal lifts of vector fields on $\cQ_1$ determined by the non-degeneracy of the two-form $B.$ 

Our general framework for T-duality, seen as a relation between Courant algebroids, will encapsulate this construction, but it reproduces the isomorphism $\mathscr{R}$ for any doubled fibration $M$ without any conditions on the topology of the fibres.
Thus while the above construction is restricted to the possibility of integrating along the fibres, in this paper we propose an alternative construction without introducing the Fourier-Mukai transform $\varrho,$ based on the reduction of Courant algebroids over foliated manifolds. Our approach is thus powerful enough to cover T-dualities based on correspondence spaces which involve non-compact fibres, such as those which arise in Poisson-Lie T-duality~\cite{Arvanitakis:2021lwo,Demulder:2022nlz}.

In our framework, the metric and Kalb-Ramond field of a string background, given by specifying a generalised metric on $\cQ_1$, play a crucial role in determining a natural definition of T-duality, providing the necessary restrictions for the definition of $\mathscr{R}$ in a unique way. Remarkably, the crucial non-degeneracy condition for $B$ arises naturally in our construction. 
This extends the work of \cite{cavalcanti2011generalized}, where only isotropic reduction of a Courant algebroid by a product Lie group action is considered. At the same time it preserves all of its features, such as the description of the Buscher rules for T-dual backgrounds in terms of the module isomorphism $\mathscr{R}$ as a generalised isometry.

\medskip

\subsection{Generalised T-duality and Para-Hermitian Geometry}~\\[5pt]
The correspondence space picture of T-duality was extended to doubled twisted tori in~\cite{Hull:2007jy,DallAgata2007}, which further double the base $\cB$ and provide instances of almost para-Hermitian manifolds. This incorporates a geometric description of a space together with all of its T-duals, including those that may be `non-geometric'.
This perspective was extended by~\cite{SzMar, Marotta2021born} to give a geometric notion of generalised T-duality transformations of arbitrary foliated almost para-Hermitian manifolds, endowed with generalised para-Hermitian metrics, as diffeomorphisms which preserve the split-signature metric $\eta$ and naturally incorporate non-isometric T-duality as well as non-abelian T-duality. 

The string sigma-model for the doubled twisted torus considered in~\cite{Hull2009,ReidEdwards2009}, and its generalisation to the Born sigma-model for almost para-Hermitian manifolds in~\cite{Marotta2021born}, allow for a Lie algebroid gauging along the leaves of a foliation. This relates T-dual sigma-models via reductions to the respective leaf spaces,
though  the explicit nature of the relation between T-dual backgrounds has so far not been identified. In this paper we will make this notion of generalised T-duality more precise by viewing it in terms of our Courant algebroid relations, which provide a `map-like' device preserving all Courant algebroid structures, and in turn the dynamics of the corresponding sigma-models.

\medskip

\subsection{Summary of Results and Outline}~\\[5pt]
The outline of this paper is as follows:
Sections \ref{sec:Courantalgebroids}--\ref{sec:Generalisedisometries} give the necessary background material for our main results which are presented in Section \ref{sec:T-duality}, with each section each being relatively self-contained, while Section \ref{sec:applications} considers some concrete applications and examples.

\textbf{Section \ref{sec:Courantalgebroids}} begins with a review of Courant algebroids. After discussing some important properties of Courant algebroids and their isomorphisms, in Section \ref{ssec:CAreduction} we introduce the first of the three major building blocks required for Section \ref{sec:T-duality}, that of Courant algebroid reduction, following closely the work of Bursztyn-Cavalcanti-Gualtieri~\cite{Bursztyn2007reduction}, and in particular of Zambon~\cite{Zambon2008reduction}. 
The reduction begins with an exact Courant algebroid $E$ over a fibred manifold, whose fibres are defined by an involutive isotropic subbundle $K\subset E$, and produces an exact Courant algebroid on the leaf space (here assumed to be smooth). 
We highlight the relevance of the \emph{basic sections} of\footnote{Notation: For a subbundle $K$ of a Courant algebroid $E$, we denote by $K^\perp\subset E$ the annihilator of $K$ with respect to the Courant algebroid pairing; then $K$ is \emph{isotropic} if $K\subseteq K^\perp$. For a subbundle $V$ of the tangent bundle $TM$ of a manifold $M$, we denote by $\ann(V)\subset T^*M$ the annihilator of $V$ with respect to the dual pairing between $TM$ and $T^*M$.} $K^\perp$ (Definition \ref{def:basisections}), i.e. sections of $K^\perp$ that admit a form of $K$-invariance, since it suffices that there are enough basic sections to make
the reduction possible~\cite{Zambon2008reduction}. Equivalently, it suffices that there exists an adapted splitting (Definition~\ref{defn:adaptedsplittings}), and we lean heavily on such objects throughout Section \ref{sec:T-duality}.
We further explore their crucial role in this construction in order to prove some properties of reducible exact Courant algebroids. 

\textbf{Section \ref{sec:CArelations}} provides an introduction to the second, and perhaps most important, building block of Section \ref{sec:T-duality}, that of Courant algebroid relations. Courant algebroid relations generalise Courant algebroid isomorphisms to cases where there is no map between the base manifolds.
T-duality gives a duality between sigma-models over manifolds which are not necessarily diffeomorphic. This makes Courant algebroid relations a powerful framework for the formalisation of T-duality. Section \ref{ssec:CArelations} repeats arguments of \cite[Section~3]{Vysoky2020hitchiker}, defining Courant algebroid relations, as well as how and when they can be composed. Section~\ref{ssec:reductionrelation} then reformulates the reduction of Section \ref{ssec:CAreduction} by an isotropic bundle $K$ as a Courant algebroid relation, denoted $Q(K)$, following \cite[Section~4.2]{Vysoky2020hitchiker} closely.

\textbf{Section \ref{sec:Generalisedisometries}} presents the final building block by introducing the notion of (transverse) generalised isometries for Courant algebroid relations. Section \ref{ssec:generalisedmetrics} reviews generalised metrics while Section \ref{ssec:generalisedisometries} reviews generalised isometries, following \cite[Section~5]{Vysoky2020hitchiker}. 
We also show that the relation $Q(K)$ defined by isotropic reduction of an exact Courant algebroid cannot be a generalised isometry, as similarly discussed by Vysok\'y in \cite{Vysoky2020hitchiker} for the case of reduction of $\sfG$-equivariant Courant algebroids with an isotropic action of a Lie group $\sfG$. 
In order to make this relation a generalised isometry, Vysok\'y relinquishes the assumption that the action is isotropic and shows that the reduction relation can be a generalised isometry only if $K \cap K^\perp = \set{0}.$ 
In this paper we choose another approach and try to find a structure such that the relation may be seen as a different type of isometry whilst preserving the isotropicity of $K.$ We defer the exploration of non-isotropic foliated reduction to future work.

Thus in Section \ref{ssec:transisometries} we extend isometries of generalised metrics to isometries of transverse generalised metrics. Transverse generalised metrics were first introduced by \v{S}evera-Strobl in~\cite{Severa2019transverse} as degenerate generalised metrics which admit an invariance condition with respect to an involutive isotropic subbundle $K$ of an exact Courant algebroid. 
The Courant algebroid relation $Q(K)$ discussed in Section \ref{ssec:reductionrelation} is our main example of what we call a transverse generalised isometry, as seen in 
\begin{customthm}{\ref{thm:qktgi}}
 If the reduced exact Courant algebroid admits a generalised metric, then there exists a unique transverse generalised metric on the original exact Courant algebroid, invariant with respect to the subbundle inducing the fibration for the reduction, which makes the reduction relation $Q(K)$ a transverse generalised isometry. The converse is also true.   
\end{customthm}
We then discuss the composition of transverse generalised isometries in Section \ref{ssec:transisomcomposition}; this is a technical undertaking, and most of the technicalities are not necessary for our purposes, so this discussion is kept brief with some details delegated to Appendix~\ref{app:changesplitting}.

\textbf{Section \ref{sec:T-duality}} develops in detail the main idea of this paper: rephrasing T-duality as a Courant algebroid relation. Such a Courant algebroid relation $R$ fits into the diagram
\begin{equation}
\begin{tikzcd}
 & E \arrow[tail]{dr}{Q(K_2)} &  \\
\red E{}_1 \arrow[dashed]{rr}{R}\arrow[dashed]{ur}{Q(K_1)^\top } & & \red E{}_2 
\end{tikzcd}
\qquad \qquad \text{over} \qquad \qquad
\begin{tikzcd}
& M \arrow[swap]{dl}{\varpi_1} \arrow{dr}{\varpi_2} & \\
\cQ_1  & & \cQ_2 
\end{tikzcd}
\end{equation}
where $Q(K_i)$ is the Courant algebroid relation characterising the reduction of $E$ over $M$ by $K_i$ to $\red E{}_i$ over $\cQ_i$, as in Section \ref{ssec:reductionrelation}. 

Section \ref{ssec:topoligcalTduality} deals with the construction of the Courant algebroid relation $R$  
and gives a first notion of T-duality: we say that $\red E{}_1$ and $\red E{}_2$ are \emph{T-duality related} (Definition \ref{def:topologicalTdual}). This may also be stated as
\begin{customthm}{\ref{thm:topoligicaltd}}
    $\red E{}_1$ and $\red E{}_2$ are T-duality related if and only if $K_1 \cap K_2$ has constant rank.
\end{customthm}
We also see that the T-duality relation $R$ is maximally isotropic, i.e. it is a Dirac structure.
This definition alone provides a notion of T-duality only for the topological (Wess-Zumino) term of sigma-models, disregarding the dynamics, i.e. the Polyakov functional~\eqref{eqn:sigmanorm}.

Section \ref{ssec:topoligcalTduality} can be seen as a prelude to Section \ref{ssec:geometricTduality}, passing from purely topological data to include dynamical data, by introducing geometric structure in the form of generalised metrics into the picture. In this vein, we give our second and principal definition of T-duality: $\red E{}_1$ and $\red E{}_2$ are \emph{geometrically T-dual} when the Courant algebroid relation $R$ is additionally a generalised isometry (Definition \ref{defn:Tdual}).
The remainder of Section~\ref{ssec:geometricTduality}~is devoted to stating an equivalence condition akin to Theorem~\ref{thm:topoligicaltd} in the geometric case. We introduce the notion of T-duality directions (Definition \ref{def:tdualitydirections}), and discuss some necessary invariance-like conditions for generalised metrics and adapted splittings in the T-duality directions. Our main result is then

\begin{customthm}{\ref{thm:maingeneral}}
    Starting with a generalised metric on $\red E{}_1$, invariant in the T-duality directions, the following are equivalent:
    \begin{enumerate}[label = (\roman{enumi})]
        \item $K_2^\perp \cap K_1\subseteq K_2 \ .$
        \item $K_2 \cap K_1^\perp \subseteq K_1 \ .$
        \item $\red E{}_1$ and $\red E{}_2$ are geometrically T-dual.
    \end{enumerate}
\end{customthm}

Indeed, starting with a generalised metric on $\red E{}_1$, we are able to construct a generalised metric on $\red E{}_2$, making $R$ into a generalised isometry. The generalised metric on $\red E{}_2$ is the correct one, for instance Remark \ref{rmk:buscherrules} shows that it satisfies generalised Buscher rules.

In both Sections \ref{ssec:topoligcalTduality} and \ref{ssec:geometricTduality}, we also describe the relation $R$ in the special case of twisted standard Courant algebroids, see Proposition \ref{prop:splitcaseconditions} and Theorem \ref{thm:mainsplitcase}.

\textbf{Section \ref{sec:applications}} puts the Courant algebroid relation $R$ to the test. Section \ref{ssec:correspondencespace} is concerned with T-duality in the correspondence space picture for principal $k$-torus bundles, as described by Cavalcanti-Gualtieri \cite{cavalcanti2011generalized}, and we show that this is a subclass of our definition of T-duality through

\begin{customprop}{\ref{prop:equivalencewithcorresp}}
    For two principal $\sfT^k$-bundles over the same base manifold, Definition~\ref{defn:Tdual} and the definition of T-duality in \cite{cavalcanti2011generalized} are equivalent.
\end{customprop}

Furthermore, we show that, without using the Fourier-Mukai transform, we can reproduce the isomorphism between $\sfT^k$-invariant sections of the T-dual Courant algebroids as arising from the generalised isometry obtained from our reduction, see Proposition \ref{prop:Fourier-Mukai}. Lastly, we show how global Buscher rules naturally arise from our definition of geometric T-duality, see Remark \ref{rmk:buscherrules}, and illustrate our approach on the explicit example of T-dualities between three-dimensional lens spaces in Example~\ref{ex:lensspaces}.

Another class of examples falling under our definition is presented in Section \ref{ssec:parahermitian}: the case of para-Hermitian manifolds and generalised T-duality as described in~\cite{Marotta2021born}. 
Here almost para-Hermitian structures are taken on a pair of $\eta$-isometric manifolds $\phi: M_1 \to M_2$, and where (after pulling back by $\phi$) they differ tells us in which direction T-duality should be taken. We give conditions on the Lie algebra of a local diagonalising frame for the almost para-complex structure so that firstly the T-duality relation exists (Proposition~\ref{prop:paraTdualrelation}), and secondly when it is a generalised isometry (Lemma \ref{lemma:paraBinvariance}). We also give an alternative description of the T-dual backgrounds in terms of generalised para-Hermitian metrics $\cH_1$ and $\cH_2$, respectively, resulting in an extension of the Buscher rules given in concise form by
\begin{customprop}{\ref{lemma:tdualgenparametric}}
    The T-dual backgrounds give rise to generalised para-Hermitian metrics satisfying $\phi^* \cH_2 = \cH_1$.
\end{customprop}
An explicit illustration of generalised T-duality is presented in Section \ref{ssec:doubledtwistedtorus} wherein we consider the doubled Heisenberg nilmanifold in six dimensions, and demonstrate the well-known T-duality between the three-dimensional Heisenberg nilmanifold and the three-torus with $H$-flux in our approach.

\medskip

\subsection{Acknowledgements}~\\[5pt]
We thank Alex Arvanitakis and Daniel Thompson for helpful discussions, and Pavol \v{S}evera for insightful comments on our manuscript. We thank the anonymous referees for their comments that helped us to greatly improve the presentation. This article is based upon work from COST Action CaLISTA CA21109 supported by COST (European Cooperation in Science and Technology). The work of T.C.D. was supported by an EPSRC Doctoral Training Partnership Award.
The work of {\sc V.E.M.} was supported in part by a
Maxwell Institute Research Fellowship, by the GACR Grant EXPRO 19-28268X and by PNRR MUR projects PE0000023-NQSTI. The work of R.J.S. was supported in part by the STFC Consolidated Grant ST/P000363/1.

\medskip

\section{Exact Courant Algebroids: Isomorphism and Reduction}\label{sec:Courantalgebroids}
We briefly recall the main notions and results concerning Courant algebroids. We refer to \cite{gualtieri:tesi, Kotov:2010wr, Jurco2016courant, Garcia-Fernandez:2020ope} and references therein for more extensive treatments.

\medskip

\subsection{Courant Algebroids}~\\[5pt]
The structure of the vector bundle $\IT M = TM\oplus T^*M$ endowed with the Dorfman bracket $\llbracket  \, \cdot \, , \, \cdot \, \rrbracket_H$ introduced in Section \ref{subsect:sigma} suggests that the fundamental notion to introduce is that of a Courant algebroid. In the following we provide a brief account of Courant algebroids.
\begin{definition}\label{def:CourantAlg}
A \emph{Courant algebroid} is a quadruple $(E,\llbracket\,\cdot\,,\,\cdot\,\rrbracket,\langle\,\cdot\,,\,\cdot\,\rangle ,\rho)$, where $E$ is a vector bundle over a manifold $M$
with a fibrewise non-degenerate pairing $\langle\,\cdot\,,\,\cdot\,\rangle \in\mathsf{\Gamma}(\midodot^2E^*),$ a vector bundle morphism 
$\rho \colon E \rightarrow TM $
called the \emph{anchor}, and a bracket operation   
\begin{equation} 
\llbracket\,\cdot\,,\,\cdot\,\rrbracket \colon \mathsf{\Gamma}(E) \times \mathsf{\Gamma}(E) \longrightarrow \mathsf{\Gamma}(E)
\end{equation}
called the \emph{Dorfman bracket}, which together satisfy
\begin{enumerate}[label=(\roman{enumi})]
    \item $\rho(e)\cdot\langle e_1, e_2\rangle  = \langle\llbracket e,e_1\rrbracket, e_2\rangle  + \langle e_1, \llbracket e, e_2\rrbracket\rangle  \ ,$  \label{eqn:metric1} 
    \item $\langle\llbracket e, e\rrbracket, e_1\rangle  = \tfrac{1}{2}\, \rho(e_1)\cdot\langle e, e\rangle  \ ,$  \label{eqn:metric2}
    \item $\llbracket e, \llbracket e_1,e_2 \rrbracket  \rrbracket  = \llbracket \llbracket e,e_1\rrbracket  ,e_2 \rrbracket  +  \llbracket e_1, \llbracket e,e_2 \rrbracket  \rrbracket  \ , $ \label{eqn:Jacobiid}
\end{enumerate}
for all $e,e_1, e_2 \in \mathsf{\Gamma}(E)$. 
\end{definition}

\begin{remark}
The anchored Leibniz rule for the Dorfman bracket
\begin{align} \label{eqn:anchorLeibniz}
\llbracket e_1 ,f\,e_2 \rrbracket  = f\, \llbracket e_1, e_2 \rrbracket  + \big( \rho(e_1)\cdot f\big)\,e_2 \ ,
\end{align}
for all $e_1,e_2\in\mathsf{\Gamma}(E)$ and $f\in C^\infty(M),$ follows from item~\ref{eqn:metric1}.
The Jacobi identity~\ref{eqn:Jacobiid} and the anchored Leibniz rule \eqref{eqn:anchorLeibniz} imply that the anchor $\rho$ is a bracket homomorphism:
\be
\rho(\llbracket e_1, e_2\rrbracket)= [\rho(e_1), \rho(e_2)] \ ,
\ee
for all $e_1 , e_2 \in \mathsf{\Gamma}(E)$. We refer to \cite[Section 2]{Kotov:2010wr} for a complete account of all the main properties of Courant algebroids.
\end{remark}

Throughout this paper we will denote the Courant algebroid $(E,\llbracket\,\cdot\,,\,\cdot\,\rrbracket,\langle\,\cdot\,,\,\cdot\,\rangle ,\rho)$ simply by $E$ when there is no ambiguity in the structure maps. If there is more than one Courant algebroid $E$ involved in the discussion, we label its operations with a subscript $_E\,$.

A Courant algebroid $E$ is called \emph{regular} if its anchor map $\rho\colon E \to TM$ has constant rank, and \emph{transitive} if $\rho$ is surjective. 

\begin{remark}\label{rem:Courantcond}
For any Courant algebroid $E$ over $M$ there is a map $\rho^*\colon T^*M\to E$ given by
\begin{align}\label{eqn:rhostar}
\langle\rho^*(\alpha),e\rangle  \coloneqq \iota_e\, \rho^{\rm t}(\alpha) \ ,
\end{align}
for all $\alpha\in \mathsf{\Omega}^1(M)$ and $e\in\mathsf{\Gamma}(E)$, where $\rho^{\rm t}\colon T^*M\to E^*$ is the transpose of $\rho.$ 
The map $\rho^*$ induces a map $\cD\colon C^\infty(M)\to\mathsf{\Gamma}(E)$ defined by $$\cD f=\rho^*\,\de f\ ,$$ for all $f\in C^\infty(M)$, which obeys a derivation-like rule and is the natural generalisation of the exterior derivative in the algebroid $E$. Then item~\ref{eqn:metric2} of Definition~\ref{def:CourantAlg} is equivalent to
\begin{align}\label{eqn:brakee}
    \llbracket e,e\rrbracket = \tfrac12\,\cD\langle e,e\rangle \ ,
\end{align}
which together with Equation~\eqref{eqn:anchorLeibniz} imply the additional Leibniz rule
\begin{align} \label{eqn:Leibnizfirst}
    \llbracket f\,e_1,e_2\rrbracket = f\,\llbracket e_1,e_2\rrbracket -\big(\rho(e_2)\cdot f\big)\,e_1 + \langle e_1,e_2\rangle \, \cD f \ .
\end{align}
\end{remark}

Applying $\rho$ to \eqref{eqn:brakee}, one has
\begin{align}
 \rho \circ \rho^* = 0 \ ,   
\end{align}
which motivates

\begin{definition}
A Courant algebroid $E$ over $M$ is \emph{exact} if the chain complex
\begin{align}\label{eqn:exCou}
0 \longrightarrow T^*M \xlongrightarrow{\rho^*} E \xlongrightarrow{\rho} TM \longrightarrow 0
\end{align}
is a short exact sequence of vector bundles.
\end{definition}  

\medskip

\subsection{Isomorphisms of Courant Algebroids}~\\[5pt] \label{sub:expreiso}
Throughout this paper we will make extensive use of isomorphisms of exact Courant algebroids. Here we recall their main properties, together with some of the standard examples, closely following \cite{gualtieri:tesi, Garcia-Fernandez:2020ope}.

\begin{definition} \label{def:courantmorph}
Let $E_1$ and $E_2$ be Courant algebroids over manifolds $M_1$ and $M_2$, respectively. A vector bundle morphism $\Phi\colon E_1 \to E_2$ covering a diffeomorphism $\phi \colon  M_1 \to M_2$ is a \emph{Courant algebroid morphism} if it satisfies
\begin{enumerate}[label = (\roman{enumi})]
    \item \label{item:orthogonality} $\langle \Phi (\,\cdot\,)\, ,\, \Phi(\,\cdot \,)\rangle_{E_2} \circ \phi =\ip{\,\cdot\,,\,\cdot\,}_{E_1} \ ,$ \ (isometry) 
    \item\label{item:bracketpreserving} $\llbracket\Phi (\,\cdot \,)\, , \,\Phi (\,\cdot\,)\rrbracket_{E_2} = \Phi (\,\llbracket\,\cdot \,, \,\cdot\,\rrbracket_{E_1}) \ .$ \ (bracket homomorphism)
\end{enumerate}
\end{definition}

Here and in the following the word \emph{isomorphism} in the context of Courant algebroids refers to vector bundle isomorphisms covering diffeomorphisms between different Courant algebroid structures, even if defined on the same vector bundle.

\begin{remark} \label{rmk:anchorcomp}
Definition \ref{def:courantmorph} together with the anchored Leibniz rule \eqref{eqn:anchorLeibniz} implies a  compatibility condition with the anchor maps:
\begin{align}\label{eqn:isoanchorcompatibility}
\rho_{E_2} \circ \Phi = \phi_* \circ \rho_{E_1} \ .  
\end{align}
\end{remark}

\begin{example} \label{ex:diffeostandard}
Let $(\IT M , \llbracket\,\cdot\,,\,\cdot\,\rrbracket_{H} ,\langle\,\cdot\,,\,\cdot\,\rangle_{\IT M}, {\rm pr}_1  ),$ be the (\emph{$H$-twisted}) \emph{standard Courant algebroid} over $M$ with $H \in \mathsf{\Omega}^3_{\rm cl} (M),$ the Dorfman bracket 
\begin{align}\label{eqn:Hbracket}
    \llbracket X + \alpha, Y + \beta \rrbracket_H \coloneqq[X,Y] + \pounds_X \beta - \iota_Y\, \de \alpha + \iota_Y\, \iota_X H,
\end{align}
the pairing given by
\begin{align}
    \ip{X+\alpha, Y+\beta} = \iota_Y \alpha + \iota_X\beta,
\end{align} 
and anchor ${\rm pr}_1$ the projection to the first summand of $\IT M = TM \oplus T^*M$. From now on, we will denote the $H$-twisted standard Courant algebroid over $M$ by $(\IT M, H).$

For any $\phi \in {\sf Diff}(M_1, M_2)$, consider the induced vector bundle isomorphism $$\overline{\phi}:=\phi_* + (\phi^{-1})^*:\IT M_1\longrightarrow\IT M_2$$
which covers $\phi$. It is straightforward to see that the Dorfman bracket on $\IT M_1$ transforms to
\be
\overline{\phi}\big(\llbracket \overline{\phi}^{-1}(X+\alpha), \overline{\phi}^{-1}(Y +\beta) \rrbracket_{H_1} \big) = \llbracket X+\alpha, Y+\beta \rrbracket_{(\phi^{-1})^*H_1} \ ,
\ee
for all $X+\alpha, Y+\beta \in \mathsf{\Gamma}(\IT M_2).$ Thus $\overline{\phi}$ is an isomorphism between the  $H_1$-twisted standard Courant algebroid $\IT M_1$ and the $(\phi^{-1})^* H_1$-twisted standard Courant algebroid $\IT M_2.$
\end{example}

Among the main examples of isomorphisms of exact Courant algebroids are the $B$-field transformations for a twisted standard algebroid. 

\begin{definition}
Let $(\IT M , H)$ be the $H$-twisted standard Courant algebroid over $M,$ 
where $H$ is the three-form on $M$ characterising its Dorfman bracket. Let $B \in \mathsf{\Omega}^2(M)$ be any two-form on $M$; it induces the vector bundle morphism 
$B^\flat \colon T M \rightarrow T^* M$ given by $B^\flat(X) = \iota_X B,$ for all $X \in \mathsf{\Gamma}(T M).$
The vector bundle isomorphism $\e^{B}\, \colon \IT M \rightarrow \IT M$ given by
\be
\e^{B}\, (X + \alpha) = X + B^\flat(X) + \alpha \ ,
\ee
for all $X \in \mathsf{\Gamma}(T M)$ and $\alpha \in \mathsf{\Gamma}(T^* M),$ is a \emph{$B$-field transformation}. 
\end{definition}

The graph ${\rm gr}(B) \subset \IT M$ is a maximally isotropic subbundle, since
\be
\langle X + \iota_X B, Y + \iota_Y B \rangle = \iota_X \iota_Y B + \iota_Y \iota_X B=0 \ ,
\ee
for all $X, Y \in \mathsf{\Gamma}(T M)$. Any $B$-field transformation is compatible with the anchor ${\rm pr}_1,$ since
\be
{\rm pr}_1 \big(\e^{B}\,(X +\alpha)\big) = X \ ,
\ee
for all $X\in \mathsf{\Gamma}(T M)$ and $\alpha \in \mathsf{\Gamma}(T^* M).$
It further transforms the Dorfman bracket to 
\be
\llbracket \e^{B}\,(X+\alpha), \e^{B}\,(Y+\beta) \rrbracket_H = \e^{B}\,(\llbracket X+\alpha, Y+\beta \rrbracket_{H}) + \iota_Y \,\iota_X\, \de B  \ ,
\ee
for all $X+\alpha, Y+ \beta \in \mathsf{\Gamma}(\IT M),$ see e.g. \cite{gualtieri:tesi}. 

We adapt the arguments used in \cite[Section 3.2]{gualtieri:tesi} to characterise isomorphisms of exact Courant algebroids as

\begin{proposition} \label{decpreiso}
Let $\Phi \colon \IT M_1 \rightarrow \IT M_2$ be an isomorphism of the twisted standard Courant algebroids $\IT M_1$ and $\IT M_2$, covering $\phi \in {\sf Diff}(M_1, M_2)$, with twisting three-forms $H_1 \in \mathsf{\Omega}_{\rm cl}^3(M_1)$ and $H_2 \in \mathsf{\Omega}_{\rm cl}^3(M_2)$. Then $\Phi$ can be expressed as the composition $\Phi =\overline{\phi}\circ \e^{B}\,$ of a $B$-field transformation with two-form $B \in \mathsf{\Omega}^2 (M_1)$, and $\phi^*H_2=H_1-\de B$. 
\end{proposition}
\begin{proof}
Set $F \coloneqq \overline{\phi}^{-1} \circ \Phi$, which covers the identity $\unit_{M_1}.$ By Remark \ref{rmk:anchorcomp} it follows that 
\be \label{eqn:anchpr1}
\rho_{E_1} \circ F = \rho_{E_1} \ .
\ee
Equation \eqref{eqn:anchpr1} gives
\be
\langle e, F^{-1}(\de f) \rangle_{E_1}=\langle F(e), \de f \rangle_{E_1} =  \iota_{\rho_{E_1}(F(e))}\, \de f = \iota_{\rho_{E_1}(e)}\, \de f = \langle e, \de f\rangle_{E_1}
\ee
for all $e \in \mathsf{\Gamma}(\IT M_1)$ and $f \in C^\infty(M_1).$ This shows that $F^{-1}$ acts as the identity on $T^* M_1.$ Therefore $F= \e^{B}\,$ for some $B \in \mathsf{\Omega}^2 (M_1),$ because $F$ is an isomorphism. Thus we obtain, from the calculation of the Dorfman bracket that $\phi^*H_2=H_1-\de B$.   
\end{proof}

When $M_1=M_2=M$, the group action defined by Proposition \ref{decpreiso} is precisely the group action \eqref{eqn:actionSemi} transforming the Wess-Zumino functional $S_H$ of the sigma-model discussed in Section \ref{subsect:sigma}.

\begin{corollary}\label{cor:CAauts}
    The automorphism group of the $H$-twisted standard Courant algebroid $(\IT M, H)$ consists of those $\Phi = \overline{\phi} \circ \e^{B}\,$ such that $\phi^* H = H$ and $B$ is closed.
\end{corollary}

\begin{remark}[\textbf{\v{S}evera's Classification}]
The classification of exact Courant algebroids $E$, due to \v{S}evera~\cite{Severa-letters}, arises from the choice of any isotropic splitting $\sigma$ of the short exact sequence \eqref{eqn:exCou}. This yields the Courant algebroid isomorphism $\sigma + \tfrac{1}{2}\rho^*$ which identifies $E$ with $\IT M$ and defines a closed three-form $H_\sigma \in \mathsf{\Omega}_{\rm cl}^3(M)$ given by
\begin{align}\label{eqn:severaclass}
    H_\sigma(X,Y,Z) =  \ip{\llbracket\sigma(X),\sigma(Y)\rrbracket, \sigma(Z)}
\end{align}
for $X,Y,Z \in \mathsf{\Gamma}(TM)$.

The equivalence classes of Courant algebroids (up to small isomorphisms\footnote{A small isomorphism is a Courant algebroid isomorphism covering a diffeomorphism lying in the identity component of $\mathsf{Diff}(M)$, see \cite[Definition 2.18]{Garcia-Fernandez:2020ope}}) are in one-to-one correspondence with cohomology classes in ${\sf H}^3( M,\IR)$. Since the difference between two isotropic splittings $\sigma-\sigma'$ defines a two-form $B\in \mathsf{\Omega}^2(M)$ via $(\sigma-\sigma')(X) = \rho_E^*\left(\iota_X B\right)$, $H_\sigma$ is shifted by $\de B$ under a change of splitting. Hence there is a well-defined cohomology class $[H_\sigma]\in {\sf H}^3(M,\IR)$ associated to the exact sequence \eqref{eqn:exCou} which completely determines the Courant algebroid structure. This is called the \v{S}evera class of the exact Courant algebroid. A general discussion can be found in~\cite[Section 2.2]{Garcia-Fernandez:2020ope}.

As a consequence of the \v{S}evera classification and Proposition \ref{decpreiso}, any isomorphism between exact Courant algebroids $E_1$ and $E_2$ with representatives of their \v{S}evera class $H_1$ and $H_2$, respectively, covering a diffeomorphism $\phi \colon M_1 \to M_2$ is such that $\phi^* H_2= H_1 -\de B$, for some $B \in \mathsf{\Omega}^2(M_1).$
\end{remark}

\medskip

\subsection{Reduction of Exact Courant Algebroids}\label{ssec:CAreduction}~\\[5pt]
We now summarise the reduction of exact Courant algebroids over foliated manifolds, as shown by Zambon~\cite{Zambon2008reduction} generalising work of Bursztyn-Cavalcanti-Gualtieri~\cite{Bursztyn2007reduction} to foliations that are not necessarily generated by a group action. We specialise Zambon's discussion to subbundles of a Courant algebroid over $M$ which are supported on all of $M$, rather than on some submanifold of $M$. As discussed in \cite{Severa-letters, Severa2015}, this case naturally arises in the description of sigma-models related to equivariant exact Courant algebroids.
In particular, the infinitesimal symmetries of the Wess-Zumino functional \eqref{eqn:ActionQTop} generate directions in which $S_H$ is constant. 

This motivates a ``gauging-like'' reduction of the string background defining the sigma-model: the condition \eqref{eqn:constantSH} naturally arises from the definition of the map $\Psi$ in Equation~\eqref{eqn:ghequivariantmap} by looking at the sections $X + \alpha \in \mathsf{\Gamma}(\IT M)$ for which $\Psi(X + \alpha) = X,$ i.e. $\iota_X H =\de \alpha.$ Therefore the infinitesimal action associated with the variation generated by
$X \in \mathsf{\Gamma}(T M)$ is given by the differential operator $\llbracket X +\alpha , \, \cdot \, \rrbracket_H,$ for any $X + \alpha \in \mathsf{\Gamma}(\IT M)$ for which $\Psi(X + \alpha) = X,$ which are the elements associated with the flat directions for $S_H$. This shows how reduction of exact Courant algebroids naturally arises from the symmetries of the topological term of a string sigma-model. The construction which follows can be motivated by the observation~\cite{Plauschinn:2014nha} that the constraints for gauging the full sigma-model by isometries of $M$ are equivalent to requiring that the sections $X+\alpha$ span an involutive isotropic subbundle $K\subset\IT M$.

Let $E$ be an exact Courant algebroid over $M$, and $K$ an isotropic subbundle of $E$ supported on $M$ such that $\rho(K^\perp) = TM$.
\begin{definition} \label{def:basisections}
The space of sections of $K^\perp$ which are \emph{basic with respect to $K$} is
\begin{align*}
\mathsf{\Gamma}_{\text{bas}}(K^\perp) \coloneqq \set{e \in \mathsf{\Gamma}(K^\perp) \, | \, \llbracket \mathsf{\Gamma}(K) ,e \rrbracket \subset \mathsf{\Gamma}(K)} \ .
\end{align*}
\end{definition}
When there are enough basic sections, i.e. $\mathsf{\Gamma}_{\mathrm{bas}}(K^\perp)$ spans $K^\perp$ pointwise, then $K$ and $K^\perp$ satisfy the main properties~\cite[Lemma 3.5]{Zambon2008reduction}
\begin{align}
  \llbracket \mathsf{\Gamma}(K), \mathsf{\Gamma}(K^\perp) \rrbracket & \subset \mathsf{\Gamma}(K^\perp) \ ,   \label{eqn:basicprop1}\\[4pt]
 \llbracket \mathsf{\Gamma}(K), \mathsf{\Gamma}(K) \rrbracket & \subset \mathsf{\Gamma}(K) \label{eqn:basicprop2} \ .
\end{align}
In particular, the distribution $\rho(K)\subset TM$ induces a smooth integrating foliation $\cF$ of $M$. 

These properties imply the following central result for reduction of exact Courant algebroids over foliated manifolds, as stated by Zambon in~\cite[Theorem 3.7]{Zambon2008reduction}, generalising the analogous statement of Bursztyn-Cavalcanti-Gualtieri in~\cite[Theorem 3.3]{Bursztyn2007reduction} for the case of reduction by group actions.
\begin{theorem}\label{thm:foliationreduction}
Let $E$ be an exact Courant algebroid over $M$, and $K$ an isotropic subbundle of $E$ supported on $M$ such that $\rho(K^\perp) = TM$. Assume that the space of basic sections $\mathsf{\Gamma}_{\mathrm{bas}}(K^\perp)$ spans $K^\perp$ pointwise, and that the quotient $\cQ$ of $M$ by the foliation $\cF$ integrating $\rho(K)$ is a smooth manifold, so that the quotient map $\varpi \colon M \rightarrow \cQ$ is a surjective submersion. Then there is an exact Courant algebroid $\underline{E} $ over $\cQ$ which fits in the pullback diagram
\[
\begin{tikzcd}
K^\perp/K \arrow{rr}{} \arrow[swap]{dd}{} & & \red E \arrow[]{dd}{} \\ & & \\
M \arrow[]{rr}{\varpi} & & \cQ 
\end{tikzcd}
\]
of vector bundles.
\end{theorem}
\begin{proof}[Sketch of Proof]
We sketch the first part of the proof given in \cite{Zambon2008reduction} in order to discuss how elements in $K^\perp/K$ at two points in a leaf of $\cF$ are identified and to introduce our notation. Recall that since $K^\perp$ has enough basic sections, Equation \eqref{eqn:basicprop2} holds, i.e., $K$ is involutive. 
Let $\mathscr{P} \colon K^\perp \to K^\perp /K$ be the quotient map and denote by $N_q = \varpi^{-1}(q),$ for any $q \in \cQ,$ a leaf of $\cF$.  Then for any $m, m' \in N_q,$ we say that $e \in (K^\perp/K)_m$ and $e' \in (K^\perp/K)_{m'}$ are identified if and only if there exists a basic section $\hat{e} \in \mathsf{\Gamma}_{\rm{bas}}(K^\perp)$ such that
\begin{align} \label{eqn:identificationKperp}
 \mathscr{P}(\hat{e}(m))= e \qquad \text{and} \qquad \mathscr{P}(\hat{e}(m'))=e' \ .   
\end{align}
This means that there exists a trivialisation of $K^\perp/K \rvert_{N_q}$ obtained by projecting onto it basic sections giving a frame for $K^\perp.$ By assumption there are enough basic sections, so that they induce a frame for $K^\perp/K \rvert_{N_q}.$

It is shown in \cite{Zambon2008reduction} that this identification is well-defined, i.e. for any $\hat{e}, \check{e} \in \mathsf{\Gamma}_{\rm{bas}}(K^\perp)$ such that
$\mathscr{P}(\hat{e}(m)) =\mathscr{P}(\check{e}(m)),$ then
$\mathscr{P}(\hat{e}(m'))=\mathscr{P}(\check{e}(m')).$
This is achieved by considering a finite sequence $\set{ k_i }_{i =1,\dots,n}$ of (local) sections $k_i \in \mathsf{\Gamma}(K)$ such that the integral curves of their corresponding vector fields $\rho(k_i)$ can be linked so as to join $m$ and $m'.$ Let us denote by $t_i$ the parameter for each integral curve. 
Denote by 
\begin{align}
\exp(\mathsf{ad}_{k_i}) \ \in \ \mathsf{Aut}(E)
\end{align}
the automorphism of $E$ obtained by integrating the $K$-action
$\mathsf{ad}_{k_i} = \llbracket k_i , \, \cdot \, \rrbracket.$ 

 Then we can define the automorphism
 \begin{align} \label{eqn:Kautomorp}
  \mathsf{Ad}_k \coloneqq \exp(t_1 \, \mathsf{ad}_{k_1}) \circ \cdots \circ \exp(t_n \, \mathsf{ad}_{k_n}) \ .
 \end{align}
Since $\hat{e}$ and $\check{e}$ are basic and the property \eqref{eqn:basicprop1} holds, it follows that $\mathsf{Ad}_k (\hat{e}(m)) - \hat{e}(m') \in K_{m'}$ and 
$\mathsf{Ad}_k(\check{e}(m)) - \check{e}(m') \in K_{m'}.$
We assumed that $\hat{e}(m)$ and $\check{e}(m)$ reduce to the same element in $(K^\perp/K)_m,$ hence $\hat{e}(m) - \check{e}(m) \in K_m.$ Thus $\mathsf{Ad}_k(\hat{e}(m) - \check{e}(m) ) \in K_{m'}$ by Equation \eqref{eqn:basicprop2}.
Therefore $\hat{e}(m') - \check{e}(m') \in K_{m'},$ which means that $\hat{e}(m')$ and $\check{e}(m')$ project to the same element in $(K^\perp / K)_{m'}.$ 

The identification \eqref{eqn:identificationKperp} yields a pointwise isomorphism
\begin{align}\label{eqn:bethisom}
    \mathscr{J}_{q, m} \colon \red E{}_q \longrightarrow (K^\perp/K)_m
\end{align}
for any $m \in N_q,$ which together with the surjective submersion $\varpi$ show that $\red E$ is a vector bundle. The map $\mathscr{P},$ together with the pointwise isomorphisms $\mathscr{J}_{q,m},$ imply that
$$\mathsf{\Gamma}(\red E) \simeq \mathsf{\Gamma}_{\rm{bas}}(K^\perp)\,\big/\, \mathsf{\Gamma}(K)$$
as $C^\infty(\cQ)$-modules, where the $C^\infty(\cQ)$-module structure on $\mathsf{\Gamma}_{\rm{bas}}(K^\perp)/\mathsf{\Gamma}(K)$ is given by pulling back functions to $M$ with $\varpi.$\footnote{This is the same construction given in \cite{Mackenzie} for the quotient of $\sfG$-equivariant vector bundles by the action of a group $\sfG$. In that case the invariance condition is dictated by the involutive vector subbundle $K$ inducing the foliation of the base manifold, which in turn corresponds to the orbits of the $\sfG$-action.} 

The completion of the proof can be found in \cite{Zambon2008reduction, Bursztyn2007reduction}, wherein $\red E$ is shown to be an exact Courant algebroid with the bracket, anchor and pairing inherited from $E$.
\end{proof}

In the notation of \cite{Bursztyn2007reduction,Zambon2008reduction}, we write $$\red E = \frac{K^\perp}{K}\Big/ \cF$$ for the above construction of $\red E$, and we denote by $$\natural \colon K^\perp \longrightarrow \red E$$ the vector bundle morphism covering $\varpi$ given by the two quotients by $K$ and $\cF.$

When Theorem \ref{thm:foliationreduction} holds, we can also establish that basic sections have the property of

\begin{lemma}\label{lemma:liftbasic}
 Let $E$ be an exact Courant algebroid endowed with an involutive isotropic subbundle $K$ satisfying the assumptions of Theorem \ref{thm:foliationreduction}. Consider the short exact sequence of vector bundles
 \begin{align}
  0  \longrightarrow K \longrightarrow K^\perp \longrightarrow K^\perp/K \longrightarrow 0 \ .
 \end{align}
 Then any splitting $s \colon K^\perp/K \to K^\perp$ of this sequence satisfies $s([e]) \in \mathsf{\Gamma}_{\rm{bas}}(K^\perp),$ for any $[e] \in \mathsf{\Gamma}_{\rm{bas}}(K^\perp)/\mathsf{\Gamma}(K)$ corresponding to ${\red e} \in \mathsf{\Gamma}(\red E).$
\end{lemma}
\begin{proof}
By the isomorphism $\mathsf{\Gamma}_{\rm{bas}}(K^\perp)/\mathsf{\Gamma}(K) \simeq \mathsf{\Gamma}(\red E)$ of $C^\infty(\cQ)$-modules, to any section $\red e$ of $\red E$ there corresponds an element $[e] \in \mathsf{\Gamma}_{\rm{bas}}(K^\perp)/\mathsf{\Gamma}(K).$ Thus for any splitting $s$ of the sequence we get
 \begin{align}
 s([e]) =e + k_s \ ,    
 \end{align}
 for some $k_s \in \mathsf{\Gamma}(K),$ where $e \in \mathsf{\Gamma}_{\rm{bas}}(K^\perp).$
 Hence
 \begin{align}
   \llbracket k , s([e]) \rrbracket = \llbracket k, e \rrbracket + \llbracket k, k_s \rrbracket \ \in \ \mathsf{\Gamma}(K)  \ ,
 \end{align}
 for any section $k$ of $K$,
by Definition~\ref{def:basisections} and involutivity of $K.$
\end{proof}

\begin{example}[\textbf{Reduction of Standard Courant Algebroids}]\label{eg:standardreduction}
Of importance to us is the case of a foliated manifold $(M,\cF)$ such that $\cQ = M/\cF$ is smooth. Take $$K = T\cF \oplus \set{0}$$ as a subbundle of the standard Courant algebroid $(\IT M, 0).$ Then $K^\perp = TM \oplus \mathrm{Ann}(T\cF) $, and basic sections are given by
projectable vector fields and one-forms of the type $\de (\varpi^*f)$, where $f \in C^\infty(\cQ),$ 
which respectively span $TM$ and $\mathrm{Ann}(T\cF)$ pointwise. Hence the conditions of Theorem \ref{thm:foliationreduction} are satisfied, and our reduced Courant algebroid becomes
\begin{align*}
    \red E = \big(TM/T\cF \oplus \mathrm{Ann}(T\cF)\big)\,\big/\,\cF = T\cQ \oplus T^*\cQ = \IT\cQ \ .
\end{align*}
Remark \ref{eg:twistedCAreduction} below gives the reduction of the twisted standard Courant algebroid $(\IT M, H)$, after introducing adapted splittings.
\end{example}

\medskip

\subsubsection{Reduction of Subbundles of Exact Courant Algebroids}~\\[5pt]
Following \cite[Proposition 4.1]{Zambon2008reduction} which deals with the reduction of Dirac structures, a useful result concerning the reduction of subbundles of $E$ is

\begin{proposition} \label{prop:reducedKsubbundle}
 Let $E$ be an exact Courant algebroid over $M$ endowed with an isotropic subbundle $K$ satisfying the assumptions of Theorem \ref{thm:foliationreduction}, and let $W$ be a subbundle of $E$ such that $W \cap K^\perp$ has constant rank. If
 \begin{align} \label{eqn:reductionW}
 \llbracket \mathsf{\Gamma}(K) , \mathsf{\Gamma}(W \cap K^\perp) \rrbracket   \subset \mathsf{\Gamma}(W+K) \ ,  
 \end{align}
then $W$ descends to a subbundle $\red W$ of the reduced Courant algebroid $\red E$ over $\cQ.$
\end{proposition}

\begin{proof}
To show that elements of $W \cap K^\perp$ have a well-defined notion of identification under the map $\mathscr{P} \colon K^\perp \rightarrow K^\perp/ K$  when restricted to a leaf $N_q \subset \cF,$ for some $q \in \cQ,$ recall the definition of the automorphism $\mathsf{Ad}_k$ from Equation \eqref{eqn:Kautomorp}, together with the identification of elements in $K^\perp$ given in the sketch of the proof of Theorem \ref{thm:foliationreduction}. If the condition \eqref{eqn:reductionW} holds, then
\begin{align}
 \mathsf{Ad}_k(W \cap K^\perp) \subseteq (W+K) \cap K^\perp = (W \cap K^\perp) + K \ ,   
\end{align}
where we also use Equation \eqref{eqn:basicprop1}. Therefore ${\red W} \coloneqq \natural(W \cap K^\perp)$ is a subbundle of $\red E,$ since $W \cap K^\perp$ has constant rank.
\end{proof}

\begin{remark}\label{rmk:reductedKsubbundle}
If $K \subset W \subset K^\perp,$ then Equation \eqref{eqn:reductionW} becomes
\begin{align}
 \llbracket \mathsf{\Gamma}(K) , \mathsf{\Gamma}(W) \rrbracket   \subset \mathsf{\Gamma}(W) \ .    
\end{align}
We will be particularly interested in this case in the following sections.
Any subbundle $W$ satisfying these properties is pointwise the span of $\mathsf{\Gamma}_{\mathrm{bas}}(W)$: there are $w_i\in W$ such that $\llbracket\mathsf{\Gamma}(K),w_i\rrbracket \subset \mathsf{\Gamma}(K)$, and $W=\mathrm{Span}\set{ w_i }$. The proof can be found in \cite[Theorem 4.1]{calvo2010deformation}.\footnote{In their notation, $S=K$, $D=D^S=W$, and ``canonical'' means existence of basic sections. Note that we do not assume that $W$ is a Dirac structure, and hence $W$ is not necessarily involutive, thus we only have the implication c) $\implies$ a) of  \cite[Theorem 4.1]{calvo2010deformation}, rather than the full equivalence.}
\end{remark}

\medskip

\subsubsection{Adapted Splittings of Exact Courant Algebroids}~\\[5pt]
We give a brief summary of Zambon's notion~\cite[Section 5]{Zambon2008reduction} of adapted splittings and the reduced \v{S}evera class.

Let $E$ be an exact Courant algebroid over $M$, and let $K$ be an isotropic subbundle of $E$ such that $\rho(K^\perp) = TM$ and $\rho(K)$ is an integrable distribution with smooth leaf space $\cQ$.

\begin{definition}\label{defn:adaptedsplittings}
A splitting $\sigma \colon  TM \to E$ is \emph{adapted to $K$} if
\begin{enumerate}[label= (\alph{enumi}), labelwidth=0pt]
\item the image of $\sigma$ is isotropic,
\item $\sigma(TM)\subset K^\perp \ ,$
\item $\sigma(X)\in \mathsf{\Gamma}_{\mathrm{bas}}(K^\perp)$, for any vector field $X$ on $M$ which is projectable to $\cQ$.
\end{enumerate} 
\end{definition}

\begin{remark}[\textbf{Properties of Adapted Splittings}] \label{rmk:propertiesadapted}
Adapted splittings $\sigma$ are in one-to-one correspondence with maximally isotropic subbundles $L_\sigma\subset E$ satisfying
\begin{enumerate}[label= (\roman{enumi})]
    \item $\rho(L_\sigma \cap K^\perp) = TM \ ,$
    \item\label{item:adaptedsplitting2} $\llbracket \mathsf{\Gamma}(K) , \mathsf{\Gamma}(L_\sigma \cap K^\perp)\rrbracket \subset \mathsf{\Gamma}(L_\sigma + K) \ ,$
\end{enumerate}
via the prescription $L_\sigma = \sigma(TM)$. 
\end{remark}

If $\sigma$ is an adapted splitting, then
\begin{align} \label{eqn:whatwewanted}
\sigma(\rho(K)) \subseteq K \ ,    
\end{align}
see \cite[Remark 5.2]{Zambon2008reduction}.
Using this construction, we can establish properties of the anchor through

\begin{lemma} \label{lemma:injectiverho}
 Let $E$ be an exact Courant algebroid with an involutive isotropic subbundle $K$ such that $\rho(K^\perp)=TM$, endowed with a splitting $\sigma$ adapted to $K.$ Then $\rho \rvert_K$ is injective and $\sigma \circ \rho |_K = \unit_K.$   
\end{lemma}

\begin{proof}
Since $K^\perp \cap \rho^*(T^*M)=\rho^*\big(\ann (\rho(K))\big),$ it follows that 
\begin{align} \label{eqn:splitKperp}
    K^\perp = \sigma(TM) \oplus \rho^*\big(\ann (\rho(K))\big) \ ,
\end{align}
where we used $\sigma(TM) \cap \rho^*(T^*M) = \set{0}.$ Equation \eqref{eqn:splitKperp} and the Rank-Nullity Theorem imply
\begin{align}
 \rk(K^\perp)= 2\,\rk(TM) - \rk(K) +\rk(\ker(\rho \rvert_K)) \ .   
\end{align}
Since $\rk(K^\perp)= \rk(E) - \rk(K),$ we get $\rk(\ker(\rho \rvert_K))=0,$ because $E$ is exact (note that $\rho^*$ is injective). Thus $\rho \rvert_K$ is injective.
 
Since $\rho |_K$ is injective, it follows that $\sigma \circ \rho |_K = \unit_K,$ because 
\begin{align}
    0 = \rho(k) - \rho(k) = \rho\big(\sigma(\rho(k))\big) - \rho(k) \implies \sigma(\rho(k)) - k = 0 \ ,
\end{align}
for all $k \in K.$
\end{proof}

\begin{remark}[\textbf{Existence of Adapted Splittings}] \label{rmk:existenceadapted}
The condition of having enough basic sections is closely related to the existence of adapted splittings: by \cite[Proposition 5.5]{Zambon2008reduction}, splittings adapted to $K$ exist if and only if $\mathsf{\Gamma}_{\rm{bas}}(K^\perp)$ spans $K^\perp$ pointwise. This result will be very useful in the rest of the paper.
\end{remark}

\begin{remark}[\textbf{Reduction of Adapted Splittings}] \label{rmk:reducedadapted}
By \cite[Proposition 5.7]{Zambon2008reduction}, a splitting $\sigma$  of an exact Courant algebroid $E$ adapted to $K$, where $K$ satisfies the assumptions of Theorem \ref{thm:foliationreduction}, induces a splitting $\red \sigma$ of the reduced Courant algebroid $\red E$ over $\cQ.$ This follows from property \ref{item:adaptedsplitting2} of Remark \ref{rmk:propertiesadapted} applied to the maximally isotropic subbundle $L_\sigma = \sigma(TM)$, then applying Proposition \ref{prop:reducedKsubbundle}. Thus $L_\sigma$ reduces to a maximally isotropic subbundle $\red L{}_{\red \sigma}$ of $\red E$ which is the image of the splitting $\red \sigma$ of $\red E.$

This shows how the three-form $H_\sigma$ from Equation~\eqref{eqn:severaclass} descends to a closed three-form $\red H{}_{\red\sigma}$ representing the {\v S}evera class of $\red E$: by applying Lemma \ref{lemma:liftbasic} to any $\red X{}_i \in \mathsf{\Gamma}(T\cQ)$ for $i=1,2,3$ that are the images of projectable vector fields $X_i \in \mathsf{\Gamma}(TM),$ we can lift $\red \sigma(\red X{}_i) \in \mathsf{\Gamma}(\red E)$ to $\sigma(X_i) \in \mathsf{\Gamma}_{\rm{bas}}(K^\perp).$ Therefore
\begin{align}
 H_\sigma(X_1, X_2, X_3) &= \ip{\llbracket \sigma(X_1), \sigma(X_2)\rrbracket_E , \sigma(X_3)}_E \\[4pt] &= \ip{\llbracket \red \sigma(\red X{}_1), \red \sigma(\red X{}_2)\rrbracket_{\red E} , \red \sigma(\red X{}_3)}_{\red E} =: \red H{}_{\red\sigma}(\red X{}_1, \red X{}_2, \red X{}_3) \ ,   
\end{align}
showing that $H_\sigma$ descends to the three-form $\red H{}_{\red\sigma}$ associated with the reduced splitting $\red \sigma.$
\end{remark}

\begin{example}[\textbf{Reduction of Twisted Standard Courant Algebroids}]\label{eg:twistedCAreduction}
Let $(M,\cF)$ be a foliated manifold with smooth leaf space $\cQ = M/\cF$. Consider the twisted standard Courant algebroid  $(\IT M, H)$ for some $H \in \mathsf{\Omega}^3_{\rm cl}(M).$ 
We take as our maximally isotropic subbundle $L_\sigma$ in Remark \ref{rmk:reducedadapted} to be the tangent bundle $$L_\sigma = TM \oplus \set{0} \ .$$
Then $K = T\cF \oplus \set{0}$, and $\llbracket \mathsf{\Gamma}(K), \mathsf{\Gamma}(L_\sigma)\rrbracket \subset \mathsf{\Gamma}(L_\sigma)$ if and only if $\iota_XH = 0$ for every $X \in \mathsf{\Gamma}(T\cF)$. This suffices to show how $H$ reduces, since $K \subset L_\sigma \subset K^\perp$ in this case, hence  Remark \ref{rmk:reductedKsubbundle} applies.
Thus $\pounds_X H = 0,$ for all $X \in \mathsf{\Gamma}(T\cF),$ hence $H$ is the pullback of a three-form $\red H \in \mathsf{\Omega}_{\rm cl}^3(\cQ)$ by the quotient map $\varpi:M\to\cQ$, and the reduced Courant algebroid is $\red{E} = (\IT \cQ, {\red H}).$ Note that by Remark \ref{rmk:existenceadapted}, $K^\perp = TM \oplus \ann(T\cF)$ has enough basic sections, which, as in Example \ref{eg:standardreduction}, are given by projectable vector fields and one-forms of the form $\de \varpi^*f$ for $f \in C^\infty(\cQ)$.
\end{example}

\medskip

\section{Courant Algebroid Relations: Reduction and Composition}\label{sec:CArelations}
In this section we introduce the notion of Courant algebroid relations, following the work of Vysok\'y \cite{Vysoky2020hitchiker}. They will be used to give a new look at the reduction processes from Section~\ref{ssec:CAreduction}, and used extensively throughout Section \ref{sec:T-duality}.

\subsection{Courant Algebroid Relations}\label{ssec:CArelations}~\\[5pt]
To define Courant algebroid relations, we first need the notion of involutive structures on a Courant algebroid \cite{Li-Bland:2011iaz}.

\begin{definition}\label{def:involutivesubbundle}
Let $E$ be a Courant algebroid over $M$. A subbundle  $L$ of $E$ supported on a submanifold $C\subset M$ is an \emph{almost involutive structure supported on $C$} if
\begin{enumerate}[label = (\roman{enumi})]
    \item \label{item:invosub1} $L$ is isotropic,
    \item \label{item:invosub2} $L^\perp$ is compatible with the anchor: $\rho(L^\perp)\subset TC$. 
\end{enumerate}
Denoting by $\sfGamma(E; L)$ the $C^\infty(M)$-submodule of sections of $E$ which take values in $L$ when restricted to $C$, then $L$ is \emph{involutive} if
\begin{align}
    \llbracket \sfGamma(E; L) , \sfGamma(E; L)  \rrbracket \subseteq \sfGamma(E ; L) \ .
\end{align}
In this case, we call $L$ an \emph{involutive structure supported on $C$}. If in addition $L=L^\perp$, then $L$ is a \emph{Dirac structure supported on $C$}.
\end{definition}
 
For two Courant algebroids $E_1$ and $E_2$ over $M_1$ and $M_2$, respectively, the product $(E_1\times E_2,\llbracket\,\cdot\,,\,\cdot\,\rrbracket ,\langle\,\cdot\,,\,\cdot\,\rangle,\rho)$ is a Courant algebroid over $M_1\times M_2$, with the Courant algebroid structures defined by
\begin{align}
    \rho(e_1, e_2) & \coloneqq (\rho_{E_1}(e_1), \rho_{E_2}(e_2)) \ , \\[4pt] \ip{(e_1, e_2), (e_1', e_2')} & \coloneqq \ip{e_1, e_1'}_{E_1} \circ {\rm pr}_1 + \ip{e_2, e_2'}_{E_2} \circ {\rm pr}_2 \ , \\[4pt]
    \llbracket(e_1,e_2),(e_1',e_2')\rrbracket &\coloneqq (\llbracket e_1,e_1'\rrbracket_{E_1},\llbracket e_2, e_2'\rrbracket_{E_2}) \ ,
\end{align}
where ${\rm pr}_i\colon M_1\times M_2 \to M_i$ are the projection maps for $i=1,2$; for the sake of brevity, we will usually not write the projections explicitly in the following.
Denote by $\overline{E}$ the Courant algebroid $(E,\llbracket\,\cdot\,,\,\cdot\,\rrbracket_E ,-\langle\,\cdot\,,\,\cdot\,\rangle_E,\rho_E)$. From \cite{Li-Bland:2011iaz}, we state the following

\begin{definition}
Let $E_1$ and $E_2$ be Courant algebroids over $M_1$ and $M_2$ respectively.

A \emph{Courant algebroid relation from $E_1$ to $E_2$} is a Dirac structure $R \subseteq E_1 \times \overline{E}_2$ supported on a submanifold $C\subseteq M_1 \times M_2$, denoted $R\colon  E_1 \dashrightarrow E_2$.

If $C$ is the graph of a smooth map $\varphi \colon M_1 \to M_2$, then $R$ is a \emph{Courant algebroid morphism from $E_1$ to $E_2$ over $\varphi$}, denoted $R\colon E_1 \rightarrowtail E_2$.

If $R$ is further the graph of some vector bundle map $\Phi \colon E_1 \to E_2$ covering $\varphi\colon M_1 \to M_2$, then $\Phi$ is a \emph{classical Courant algebroid morphism over $\varphi$}.

By viewing a Courant algebroid relation $R$ as a subset of $E_2 \times \overline{E}_1$, we obtain the \emph{transpose relation} $R^\top\colon E_2 \dashrightarrow E_1$, whose support is denoted $C^\top\subseteq M_2\times M_1$, and similarly the \emph{transpose morphism} $R^\top\colon E_2 \rightarrowtail E_1$ when $\varphi$ is a diffeomorphism. 
\end{definition}

Note that in \cite{Vysoky2020hitchiker}, the above definition is given in terms of involutive structures, i.e. isotropic subbundles. We do not need this degree of generality in this work and defer to \cite[Example 4.33]{Vysoky2020hitchiker} for motivations justifying this generalisation.

\begin{example}\label{eg:classicalCAmorph}
Let $\Phi \colon E_1\to E_2$ be a vector bundle map which covers a diffeomorphism $\phi\colon M_1\to M_2$. To see why the notion of Courant algebroid relation is well-defined, suppose that its graph $\gr(\Phi) \subset E_1 \times \overline{E}_2 $ defines a classical Courant algebroid morphism. Then for any $e,e'\in \mathsf{\Gamma}(E_1)$, isotropy of $\gr(\Phi)$ gives
\begin{align*}
    0=\ip{(e,\Phi(e)),(e',\Phi(e'))} = \ip{e,e'}_{E_1} - \ip{\Phi(e),\Phi(e')}_{E_2} \ ,
\end{align*}
so $\Phi$ is an isometry. Note the minus sign, which is why we consider subbundles of $E_1\times \overline{E}_2$, rather than of $E_1\times E_2$.
Since $\gr(\Phi)$ is involutive, $\Phi$ is a bracket morphism, and hence a Courant algebroid morphism. 
That $\gr(\Phi)^\perp$ is compatible with the anchor (item~\ref{item:invosub2} of Definition~\ref{def:involutivesubbundle}) gives Equation \eqref{eqn:isoanchorcompatibility}.

Conversely, if $\Phi$ is a Courant algebroid isomorphism covering a diffeomorphism $\phi$, then $\gr(\Phi)$ is an isotropic and involutive subbundle of $E_1\times \overline{E}_2$. Since the diagram 
\[
\begin{tikzcd}
T^*M_1 \arrow{rr}{(\phi^{-1})^*} \arrow[swap]{dd}{\rho_{E_1}^*} & & T^*M_2 \arrow[]{dd}{\rho_{E_2}^*} \\
 & & \\
E_1 \arrow[]{rr}{\Phi} & & E_2 
\end{tikzcd}
\]
commutes, where $\rho_{E_i}^*$ are defined by Equation \eqref{eqn:rhostar}, $\gr(\Phi)^\perp$ is compatible with the anchor $\rho$, and hence $\Phi$ defines a classical Courant algebroid morphism. This last condition is satisfied when $E_i = \IT M_i$ and $\rho_{E_i}$ is the projection to $TM_i.$
\end{example}

We will also need a notion of sections related by a Courant algebroid relation. 

\begin{definition}
 Let $R \colon E_1 \dashrightarrow E_2$ be a Courant algebroid relation. Two sections $e_1 \in \mathsf{\Gamma}(E_1)$ and $e_2 \in \mathsf{\Gamma}(E_2)$ are $R$-\emph{related}, denoted $e_1 \sim_R e_2,$ if $(e_1, e_2) \in \mathsf{\Gamma}(E_1 \times \overline{E}_2 ; R).$ 
\end{definition}

\medskip

\subsection{Relational Approach to Reduction}~\\[5pt]
\label{ssec:reductionrelation}
In the reduction procedure of Theorem \ref{thm:foliationreduction}, there can be no vector bundle morphism between the exact Courant algebroids $E$ and $\red E$. Using the language of Courant algebroid relations, we can describe the reduction of exact Courant algebroids over foliated manifolds, analogously to the approach given in \cite[Section 4.3]{Vysoky2020hitchiker} for equivariant exact Courant algebroids.

Let $E$ be an exact Courant algebroid as in Theorem \ref{thm:foliationreduction}, i.e. $E$ is endowed with an isotropic subbundle $K$ such that $\rho(K^\perp)=TM$ and $\mathsf{\Gamma}_{\mathrm{bas}}(K^\perp)$ spans $K^\perp$ pointwise, inducing a foliation $\cF$ of the base manifold $M$ with smooth leaf space $\cQ$ such that the quotient map $\varpi \colon M \to \cQ$ is a surjective submersion. Thus $\natural \colon K^\perp \to \red E$ is a vector bundle map covering $\varpi$. Consider the vector subbundle, supported on $\gr(\varpi)$, defined fibrewise as
\begin{align*}
    Q(K)_{(m,\varpi(m))}=\set{(e,\natural(e))\, | \,  e\in K^\perp_m} \ \subset \ E_m\times \red{\overline{E}}{}_{\varpi(m)} \ ,
\end{align*}
for any $m \in M.$
We show that $Q(K)$ defines an involutive structure on $E\times \red{\overline{E}}$, in the sense of Definition~\ref{def:involutivesubbundle}, following the approach used in \cite{Vysoky2020hitchiker}.

\begin{lemma}\label{prop:Qcompatiblewithanchor}
$Q(K)$ is an almost involutive structure.
\end{lemma}
\begin{proof}
Since the pairing on $\red E$ is induced by the pairing on $E$, it follows that $Q(K)$ is isotropic, proving item~\ref{item:invosub1} of Definition~\ref{def:involutivesubbundle}.

To show item~\ref{item:invosub2} of Definition \ref{def:involutivesubbundle}, since $Q(K)^\perp$ is supported on $\gr(\varpi)$, we are required to show that \smash{$(\rho_E(e),\red\rho{}_{\red E}(\red e)) \in T_{(m,\varpi(m))} \gr(\varpi) = \gr(\varpi_*)$} for each \smash{$(e,\red e)\in Q(K)_{(m,\varpi(m))}^\perp$}, where $e \in K^\perp_m$ with $m \in M$. 

We first notice that one may write 
\begin{equation}\label{eqn:decompQperp}
Q(K)^\perp_{(m, \varpi(m))} = K_m \times \set{0^{\red E}_{\varpi(m)}}+ Q(K)_{(m, \varpi(m))} \ ,  
\end{equation}
where $0^{\red E}:\cQ\to\red E$ denotes the zero section.
To see this, let \smash{$\tilde{k} \in Q(K)^\perp_{(m, \varpi(m))}$}. Then $\tilde k = (k , \natural (k'))$ for some $k \in E_m$ and $k' \in K^\perp_m$. Then, for each $\tilde{e} = (e ,  \natural(e)) \in Q(K)_{(m, \varpi(m))}$, it follows that
\begin{align*}
0=\langle\tilde k ,\tilde e\rangle = \ip{k,e}_E-\ip{\natural(k'),\natural(e)}_{\red E}= \ip{k-k',e}_E \ .
\end{align*}
Hence $k-k' \in (K_m^\perp)^\perp = K_m$, i.e. $k = \red k + k'$ for some $\red k\in K_m$. We can therefore write $\tilde k = (\red k, 0) + ( k',\natural(k'))$. The opposite inclusion also holds, giving the decomposition \eqref{eqn:decompQperp}.

Since the anchor on $\red E$ descends from the anchor on $E$, it follows that $Q(K)$ is compatible with the anchor $\rho$. 

Thus we are left to show 
\begin{align}
 \rho_E(K_m) \times \red\rho{}_{\red E}\big(\set{0^{\red E}_{\varpi(m)}}\big ) \subseteq \text{gr}(\varpi_*) \ .
\end{align}
If $k\in K_m$, then $\varpi_*(\rho_E(k))=0^{TM}_{\varpi(m)}$. But $\red\rho{}_{\red E} \big(0^{\red E}_{\varpi(m)}\big) = 0^{TM}_{\varpi(m)}$, hence
\begin{align*}
    \big(\rho_E\times \red\rho{}_{\red E}\big)\big (k, 0^{\red E}_{\varpi(m)}\big) = \big(\rho_E(k), \varpi_*(\rho_E(k))\big) \ \in \ \gr(\varpi_*)
\end{align*}
as required.
\end{proof}

\begin{proposition}
$Q(K)$ is a Dirac structure supported on $\gr(\varpi)$, hence it defines a Courant algebroid morphism $Q(K) \colon E \rightarrowtail \red E$.
\end{proposition}
\begin{proof}
Suppose $e$ and $e'$ are basic sections for $K^\perp$, and consider the corresponding elements $(e,\natural(e))$ and $ (e',\natural(e'))$ of $ \mathsf{\Gamma}(Q(K))$. Since such sections span $Q(K)$ pointwise, by \cite[Proposition 2.23]{Vysoky2020hitchiker} it is enough to show that $\llbracket (e,\natural(e)), (e',\natural(e')) \rrbracket\in \mathsf{\Gamma}(Q(K))$. For every $k\in \mathsf{\Gamma}(K)$, using property~(i) from Definition~\ref{def:CourantAlg} we compute
    \begin{align}
        \ip{k, \llbracket e, e' \rrbracket_E}_E = \rho_E(e)\cdot \ip{k,e'}_E - \ip{\llbracket e, k \rrbracket_E , e'}_E = 0 \ ,
    \end{align}
    hence $\llbracket e,e' \rrbracket_E \in \mathsf{\Gamma}(K^\perp)$. In particular, because of the Jacobi identity \ref{eqn:Jacobiid} and property \ref{eqn:metric2} of the Dorfman bracket from Definition~\ref{def:CourantAlg}, it follows that $\llbracket e,e' \rrbracket_E \in \mathsf{\Gamma}_{\rm{bas}}(K^\perp).$

    Thus, by the way the bracket on $\red E$ is constructed, 
    \begin{align}
        \llbracket (e,\natural(e)), (e',\natural(e')) \rrbracket = \big(\llbracket e,e' \rrbracket_E , \natural(\llbracket e,e' \rrbracket_E )\big) \ \in \ \mathsf{\Gamma}(Q(K)) \ ,
    \end{align}
    hence involutivity follows. The maximality of the rank of $Q(K)$ follows from a standard dimensional calculation.
\end{proof}

Following on from Example \ref{eg:twistedCAreduction}, we then have

\begin{corollary}\label{cor:qkbasicH}
    If $E=(\IT M, H)$ is a twisted standard Courant algebroid, let $K = T \cF\oplus \set{0}$ where $\cF$ is a foliation of $M$ with smooth leaf space $\cQ = M/\cF$. Then $Q(K)$ is a Courant algebroid relation if and only if $\iota_X H = 0$ for every $X\in \mathsf{\Gamma}(T\cF)$.
\end{corollary}

In the setting of Corollary~\ref{cor:qkbasicH}, we denote $Q(K)$ by $Q(\cF)$.

\medskip

\subsection{Composition of Courant Algebroid Relations}~\\[5pt]
It will be particularly relevant in the rest of this paper to discuss the circumstances under which the composition of Courant algebroid relations gives another Courant algebroid relation. In this brief summary we will follow \cite{Vysoky2020hitchiker,Li-Bland:2011iaz}.
\begin{definition}
Let $R\colon E_1 \rel E_2$ and $R' \colon  E_2 \rel E_3$ be Courant algebroid relations. The \emph{composition} $R' \circ R$ is the subset of $E_1 \times \overline{E}_3$ given by
\begin{align}\label{eqn:reldefinition}
\hspace{-5mm} R'\circ R = \set{(e_1,e_3) \in E_1 \times \overline{E}_3 \, | \, (e_1, e_2) \in R \ , \ (e_2,e_3) \in R' \ \text{ for some } e_2\in E_2} \ .
\end{align}
\end{definition}
The requirements for the set \eqref{eqn:reldefinition} to be a Courant algebroid relation are two-fold, given in Propositions \ref{prop:cleanintersect} and \ref{prop:cleancomposition} below, and can be equivalently formulated at the level of either the Courant algebroid or the base manifold, but both boil down to checking if the set defined by Equation~\eqref{eqn:reldefinition} is a smooth subbundle of $E_1\times \overline E_3$.

We recall the important definitions and results from \cite{Vysoky2020hitchiker}. There are two manifolds that are important: $R \times R'$ and $E_1 \times \Delta(E_2) \times \overline{E}_3$, where $\Delta(E)$ is the diagonal embedding of $E$ into $\overline{E} \times E$.
Their intersection, which we denote by 
\begin{align*}
R' \diamond R = (R \times R') \cap \big(E_1 \times \Delta(E_2) \times \overline{E}_3\big) \ ,
\end{align*}
projects to $R' \circ R$. They are the key ingredients to apply the Grabowski-Rotkiewicz Theorem~\cite[Theorem 2.3]{Grabowski2009} to $R' \diamond R$.   
The set $R' \diamond R$ is closed under scalar multiplication, but in general it fails to be a submanifold of $E_1 \times \overline{E}_2 \times E_2 \times \overline{E}_3.$ If the latter holds as well, it follows immediately that $R' \diamond R$ is a subbundle of $E_1 \times \overline{E}_2 \times E_2 \times \overline{E}_3$ over 
\begin{align*}
 C' \diamond C \coloneqq (C \times C') \cap \big(M_1 \times \Delta(M_2) \times M_3\big)    
\end{align*}
by the Grabowski-Rotkiewicz Theorem. Thus let us briefly discuss the conditions under which $R' \diamond R$ is a submanifold.

\begin{definition}
Two submanifolds $C$ and $C'$ of a manifold $M$ \emph{intersect cleanly in $M$} if $C \cap C'$ is a submanifold of $M$, and 
\begin{align*}
T_c (C\cap C') = T_c C \cap T_c C' \ ,
\end{align*}
for each $c \in C \cap C'$.
\end{definition}
The inclusion $\subseteq$ is always true. The importance of this condition is that if $C$ and $C'$ intersect cleanly, then $C$ and $C'$ look locally like a pair of intersecting vector subspaces; see e.g.~\cite[Proposition~C.3.1]{hormander2009analysis}.

Following Vysok\'{y}~\cite{Vysoky2020hitchiker}, we state two propositions that are essential for the characterisation of relations and will be needed in the rest of this paper.

\begin{proposition}\label{prop:cleanintersect}
Let $R\colon E_1 \rel E_2$ and $R'\colon E_2 \rel E_3$ be Courant algebroid relations over $C\subseteq M_1\times M_2$ and $C'\subseteq M_2\times M_3$ respectively. The following conditions are equivalent:
\begin{enumerate}[label=(\roman{enumi})]
\item $R\times R'$ and $E_1 \times \Delta(E_2) \times \overline{E}_3$ intersect cleanly.
\item $C \times C' $ and $M_1 \times \Delta(M_2) \times M_3$ intersect cleanly, and the dimension of $(R' \diamond R)_c$ is independent of $c\in C' \diamond C$. \label{item:cleanint2}
\end{enumerate}
Both these conditions ensure that $R' \diamond R$ is a subbundle of $E_1 \times \overline{E}_2 \times E_2 \times \overline{E}_3$ over $C' \diamond C.$
\end{proposition}

Similar conditions can be stated for $R' \circ R$ in order to make it a subbundle of $E_1 \times \overline{E}_3$ over
\begin{align*}
 C' \circ C \coloneqq \set{(m_1, m_3) \in M_1 \times M_3 \, | \, (m_1, m_2) \in C \ , \ (m_2, m_3) \in C' \ \text{ for some} \ m_2 \in M_2 } \ ,  
\end{align*}
since $R' \circ R$ is an $\IR$-module.

\begin{proposition}\label{prop:cleancomposition}
Let $R\colon E_1 \rel E_2$ and $R'\colon E_2 \rel E_3$ be Courant algebroid relations over $C\subseteq M_1\times M_2$ and $C'\subseteq M_2\times M_3$ respectively. Suppose at least one of the conditions of Proposition~\ref{prop:cleanintersect} are satisfied. The following conditions are equivalent:
\begin{enumerate}[label=(\roman{enumi})]
\item $R'\circ R$ is a submanifold of $E_1 \times \overline{E}_3$ such that the induced map\footnote{Here $p \colon E_1 \times \overline{E}_2 \times E_2 \times \overline{E}_3 \to E_1 \times \overline{E}_3$ is the projection to the first and the fourth factor of the Cartesian product, and similarly for $\pi \colon M_1 \times M_2 \times M_2 \times M_3 \to M_1 \times M_3$ below.} $p\colon R' \diamond R \to R' \circ R$ is a smooth surjective submersion. \label{item:clean1}
\item $C'\circ C$ is a submanifold of $M_1 \times M_3$ such that the induced map $\pi\colon C' \diamond C \to C' \circ C$ is a smooth surjective submersion and the rank of the linear map $p\colon (R' \diamond R)_c \to (R' \circ R)_{\pi(c)}$ is independent of $c \in C' \diamond C$. \label{item:clean2}
\end{enumerate}
Both these conditions ensure that $R'\circ R$ is a subbundle of $E'$ supported on $C'\circ C$ and $p \colon R' \diamond R \to R' \circ R$ is a fibrewise surjective vector bundle map over $\pi \colon C'\diamond C \to C' \circ C$. 
\end{proposition}

If any of the two equivalent conditions \ref{item:clean1} or \ref{item:clean2} occurs, we say that $R$ and $R'$ \emph{compose cleanly}. Vysok\'{y} shows in~\cite{Vysoky2020hitchiker} that Propositions \ref{prop:cleanintersect} and \ref{prop:cleancomposition} together give
\begin{theorem}\label{thm:relationcomposition}
Let $R\colon E_1 \rel E_2$ and $R'\colon E_2 \rel E_3$ be Courant algebroid relations over $C$ and $C'$, respectively, which compose cleanly.
Then $R'\circ R$ is an involutive structure supported on $C' \circ C$, hence it defines a Courant algebroid relation $R' \circ R \colon E_1 \rel E_3$.
\end{theorem}

\begin{example}[\textbf{Relations between Reduced Courant Algebroids}] \label{ex:relationsreduced}
Let $K$ be an isotropic subbundle of an exact Courant algebroid $E$ over $M$ such that $\rho(K^\perp)=TM$. Suppose there exists a subbundle
$K_1 \subset K$ such that $K_1^\perp$ has enough basic sections which are also basic with respect to $K.$ Then since $K^\perp \subset K_1^\perp,$ both $K$ and $K_1$ are involutive subbundles, hence they induce foliations $\cF$ and $\cF_1$ of $M$, respectively, such that $\cF_1 \subset \cF.$ We also assume that they have smooth leaf spaces $\cQ_1 = M/\cF_1$ and $\cQ=M/\cF,$ so that we can complete the diagram
\begin{equation}\label{cd:doublequotient}
    \begin{tikzcd}
         & M \arrow{dr}{\varpi}\arrow[swap]{dl}{\varpi_1}  & \\
        \cQ_1 \arrow[dashed]{rr}{\pi} & & \cQ
    \end{tikzcd}
\end{equation}
 of surjective submersions.
 
Since $K^\perp$ and $K_1^\perp$ have enough basic sections, we may form the Courant algebroid relations $Q(K)$ and $Q(K_1)$, which fit into the diagram
\begin{equation}
    \begin{tikzcd}
        & E \arrow[swap,dashed]{dl}{Q(K_1)}\arrow[dashed]{dr}{Q(K)} &  \\
        \red E{}_1 & & \red E
    \end{tikzcd}
\end{equation}
We consider the composition of $Q(K)$ with $Q(K_1)^\top$, and check the conditions in Propositions~\ref{prop:cleanintersect} and~\ref{prop:cleancomposition}. 

It can be shown that the submanifolds $\gr(\varpi_1)^\top \times \gr(\varpi)$ and $\cQ_1 \times \Delta(M) \times \cQ$ intersect cleanly, and the projection $\gr(\varpi) \diamond \gr(\varpi_1)^\top  \to \gr(\pi)$ is a smooth surjective submersion. For $(q_1, \pi(q_1)) \in \gr(\pi)$, the space
\begin{align}
    \big(Q(K) \diamond Q(K_1)^\top \big)_{(q_1, \pi(q_1))} = \set{\big(\natural_{E_1}(e), e, e, \natural(e)\big) \ \big| \ e \in K_m^\perp }
\end{align}
 is isomorphic to $K^\perp_m$, for any $m \in M$ such that $q_1 = \varpi_1(m)$, and hence has constant dimension. Finally, the projection $Q(K) \diamond Q(K_1)^\top  \to Q(K) \circ Q(K_1)^\top $ has kernel given by $K \cap K_1 = K_1$, and hence has constant rank. 
 
 Thus the conditions \ref{item:clean1} of both Propositions \ref{prop:cleanintersect} and \ref{prop:cleancomposition} are satisfied, and the composition is clean, giving the relation
\begin{align}
R(\cF_1) \coloneqq Q(K) \circ Q(K_1)^\top  = \set{(\natural_{E_1} (e), \natural (e)) \ \big| \ e \in K^\perp} \ \subset \ \red E{}_1 \times \overline{\red{E}} \ ,
\end{align}
supported on the graph $\gr(\pi)$.
\end{example}

Example~\ref{ex:relationsreduced} may be seen as the foliated reduction counterpart of the well-known case of reduction of a $\sfG$-equivariant exact Courant algebroid by a closed Lie subgroup\footnote{This is usually chosen to be Lagrangian with respect to the split-signature pairing.} $\sfH \subset \sfG$ that represents one of the building blocks of Poisson-Lie T-duality (see e.g.~\cite[Section 5.1]{Severa:2018pag}). We will explore this case in detail along the lines of \cite{Severa:2018pag, Vysoky2020hitchiker} in future work, including how to properly deal with reduction of generalised metrics. 

\medskip

\section{Courant Algebroid Relations as Isometries}\label{sec:Generalisedisometries}
In this section we review important ideas of \cite[Section 5]{Vysoky2020hitchiker}. We then generalise these to cases of interest to us, proceeding to define generalised isometries for transverse generalised metrics. 

\medskip

\subsection{Generalised Metrics}\label{ssec:generalisedmetrics}~\\[5pt]
We wish to know when a given Courant algebroid relation is able to carry geometric structure between Courant algebroids. 
In particular, we are interested in relations between generalised metrics as a building block for our approach to T-duality.

\begin{definition}\label{defn:generalisedmetric}
A \emph{generalised metric} on a Courant algebroid $E $ is an automorphism $\tau\colon E \to E$ covering the identity with $\tau^2 = \unit_E$ which, together with the pairing $\ip{\,\cdot\,,\,\cdot\,}$, defines a positive-definite fibrewise metric 
\begin{align}\label{eqn:generalisedmetrictau}
    \cG (e, e') = \ip{e, \tau(e')}
\end{align}
on $E$, for $e, e'\in E$. We denote the Courant algebroid $E$ endowed with a generalised metric $\tau$ by $(E, \tau)$.
\end{definition}

Generalised metrics on exact Courant algebroids are characterised as follows, see e.g.~\cite[Proposition~3.5]{Jurco2016courant} for details.

\begin{proposition}\label{prop:generalisedmetric}
Let $E$ be an exact Courant algebroid over $M$. A generalised metric $\cG$, as in Equation~\eqref{eqn:generalisedmetrictau}, uniquely determines a pair $(g, b)\in \mathsf{\Gamma}(\midodot ^2 T^*M) \times \mathsf{\Omega}^2(M)$ where $g$ is non-degenerate. Conversely, a Riemannian metric $g$ and a two-form $b$ on $M$ define a generalised metric $\cG$ given by
\begin{align}\label{eqn:generalisedmetric}
    \cG =
    \begin{pmatrix}
    g - b\,g^{-1}\,b & b\, g^{-1} \\
    -g^{-1}\, b & g^{-1}
    \end{pmatrix}
\end{align}
on the isomorphic twisted standard Courant algebroid $\IT M=TM\oplus T^*M.$
\end{proposition}

An alternative characterisation of generalised metrics is

\begin{definition}\label{defn:generalisedmetric2}
    A \emph{generalised metric} is a subbundle $V^+ \subseteq E$ which is a maximal positive-definite subbundle of $E$ with respect to $\ip{\,\cdot\,,\,\cdot\,}$. We denote the Courant algebroid $E$ endowed with a generalised metric $V^+$ by $(E, V^+)$.
\end{definition}

\begin{remark}[\textbf{Generalised Metrics as Subbundles}] \label{rmk:Vplusgeneralised}
Definitions \ref{defn:generalisedmetric} and \ref{defn:generalisedmetric2} are equivalent, see e.g. \cite[Proposition 3.3]{Jurco2016courant}: the maximal positive-definite subbundle $V^+$ is the $+1$-eigenbundle for the generalised metric $\tau$. Indeed, if $\tau$ is given by $g$ and $b$ as in Proposition~\ref{prop:generalisedmetric}, then one can show
\begin{align}\label{eqn:+1eigen}
V^+= \gr(g+b) = \set{ v+ \iota_v(g+b) \in \IT M \ \vert \  v \in TM } \ .
\end{align}
The $-1$-eigenbundle of $\tau$ is given by $V^-=\gr(-g+b) = (V^+)^\perp$.   
\end{remark}
Since Definitions \ref{defn:generalisedmetric} and \ref{defn:generalisedmetric2} are equivalent, we can switch between the notations $(E,\tau)$ and $(E,V^+)$.

\medskip

\subsection{Generalised Isometries}\label{ssec:generalisedisometries}~\\[5pt]
We now have enough machinery to talk about a notion of isometries of generalised metrics for Courant algebroid relations, following~\cite{Vysoky2020hitchiker}.
\begin{definition}\label{defn:generalisedisometry}
Suppose $R\colon E_1 \rel E_2$ is a Courant algebroid relation supported on a submanifold $C\subseteq M_1 \times M_2$. Let $\tau_1$ and $\tau_2$ be generalised metrics for $E_1$ and $E_2$ respectively, and set $\tau \coloneqq \tau_1 \times \tau_2$. Then $R$ is a \emph{generalised isometry between $\tau_1$ and $\tau_2$} if $\tau (R) = R$.
\end{definition}

It follows from the definitions that $\tau=\tau_1\times \tau_2$ is a generalised metric for $E_1\times E_2$. It is not, however, a generalised metric for $E_1\times \overline{E}_2$: the $\pm 1$-eigenbundles are not necessarily maximally positive-definite with respect to the pairing on $E_1\times \overline{E}_2$. See \cite[Remark 5.2]{Vysoky2020hitchiker} for details.

\begin{remark}[\textbf{Classical Generalised Isometries}] \label{rmk:classicalIsometries}
If $R$ is a classical Courant algebroid morphism, then $R = \gr(\Phi)$ for some vector bundle map $\Phi:E_1\to E_2$ covering a smooth map $\phi\colon M_1 \to M_2$. For $\gr(\Phi)$ to be a generalised isometry, the equation
\begin{align}\label{eqn:classical generalised isometry}
    \tau_2 \circ \Phi = \Phi \circ \tau_1
\end{align}
must hold, or equivalently
\begin{align}
    \cG_2\big(\Phi(e_1), \Phi(e_1')\big)_{\phi(m_1)} = \cG_1 (e_1,e_1')_{m_1} \ ,
\end{align}
for each $e_1,e_1' \in (E_1)_{m_1}$ with $m_1\in M_1$, for the metrics $\cG_1$ and  $\cG_2$ induced by $\tau_1$ and $\tau_2$ respectively. This justifies the terminology `generalised isometry'~\cite{Vysoky2020hitchiker}.

For later purposes, let us stress that a further equivalent condition for a classical Courant algebroid morphism $\Phi$ to be a generalised isometry is 
\begin{align} \label{eqn:classgeniso3}
 \Phi(V_1 ^+) = V_2 ^+ \ ,   
\end{align}
where $V_i^+ = \ker(\unit_{E_i} - \tau_i)$ is the $+1$-eigenbundle of $\tau_i$ for $i=1, 2.$
\end{remark}

When $\Phi$ is a Courant algebroid isomorphism, generalised isometries are characterised by

\begin{proposition} \label{prop:genisochar}
 Let $E_1$ and $E_2$ be exact Courant algebroids over $M_1$ and $M_2,$ respectively, and suppose they are endowed with generalised metrics $\tau_1$ and $\tau_2$, respectively. A Courant algebroid isomorphism $\Phi \colon E_1\to E_2,$ covering $\phi \in \mathsf{Diff}(M_1, M_2),$ is a generalised isometry if and only if 
 \begin{align} 
 \phi^* g_2 =  g_1   \qquad
 \mathrm{and} \qquad
\phi^* b_2 = b_1 + B \ ,   
 \end{align}
 where, by using the canonical isomorphism $E_i \simeq TM_i \oplus TM^*_i$ 
 given by any isotropic splitting, $\tau_i$ is determined by the pair $(g_i, b_i)$ for $i=1,2,$ and (with a slight abuse of notation) $\Phi= \overline{\phi} \circ \e^{B}\,$ for some $B \in \mathsf{\Omega}^2(M_1).$
\end{proposition}
\begin{proof}
By choosing a splitting for $E_1,$ its generalised metric is $V_1^+ \simeq \mathrm{gr}(g_1+b_1).$ It then follows that
\begin{align}
\Phi\big(v_1 + \iota_{v_1}(g_1+b_1)\big) = \phi_*(v_1) + (\phi^{-1})^*\big(\iota_{v_1}(g_1+b_1 + B) \big) \ ,   
\end{align}
for all $v_1 \in TM_1,$ is an element in $V_2^+$ if and only if $\phi^*g_2 = g_1$ and $\phi^* b_2 = b_1 + B$, since $V_2^+\simeq \mathrm{gr}(g_2+b_2)$ in the chosen splitting for $E_2$.
\end{proof}

\begin{remark}
 Proposition \ref{prop:genisochar} reduces to \cite[Proposition 2.41]{Garcia-Fernandez:2020ope} when $M_1=M_2=M$ and $\tau_1=\tau_2=\tau,$ i.e. in this case $\Phi$ is a generalised isometry if and only if $\phi^* g = g$ and $B=0.$  
\end{remark}

Following the equivalence between Definitions \ref{defn:generalisedmetric} and \ref{defn:generalisedmetric2}, we can formulate an alternative definition of generalised isometry in terms of the eigenbundles of the automorphisms $\tau_1$ and~$\tau_2.$

\begin{definition}\label{defn:generalisedisometry2}
Suppose $R \colon E_1 \rel E_2$ is a Courant algebroid relation supported on a submanifold $C\subseteq M_1 \times M_2$. Let $V_1^+$ and $V_2^+$ be generalised metrics for $E_1$ and $E_2$ respectively, and set $V_i^-\coloneqq (V_i^+)^\perp$. Let $\cal V^\pm=V_1^\pm \times V_2^\pm$. Then $R$ is a \emph{generalised isometry between $V_1^+$ and $V_2^+$} if
\begin{align}\label{eqn:Rdecomp}
    R_c = (\cal V_c^+ \cap R_c) \oplus (\cal V_c^- \cap R_c) \ ,
\end{align}
for each $c \in C$.
\end{definition}
We briefly comment on the pointwise nature of this definition, which will be a theme throughout this section. Due to the nature of relations, we only consider their pointwise description, and do not demand that our definition holds globally over the submanifold $C$ that $R$ is supported on. In particular, in Equation~\eqref{eqn:Rdecomp} we do not demand that $\cal V^\pm\big|_C \cap R$ are both subbundles supported on $C,$
since the intersection might not have constant rank. 
We could distinguish from cases where global descriptions (on $C$) are possible, however this is not necessary for the description of generalised isometries, as made clear through
\begin{proposition}
    Definition \ref{defn:generalisedisometry} and Definition \ref{defn:generalisedisometry2} are equivalent.
\end{proposition}
\begin{proof}
    Suppose $\tau(R)=R$. The bundles $\cal V^\pm$ are the $\pm 1$-eigenbundles for $\tau=\tau_1\times \tau_2$. Take $c\in C$ and let $r\in R_c$. Since $E_1\times \overline{E}_2 = \cal V^+\oplus \cal V^-$, we can write $r=r^+ + r^-$, for $r^\pm \in \cal V^\pm_c$. Thus
    \begin{align}
        \tau(r^+ + r^-) = r^+ - r^- \ \in \ R_c \ .
    \end{align}
    Hence both $r^++r^-$ and $r^+-r^-$ are in $R_c$, so that $r^+, r^- \in R_c$, and it follows that $R_c\subseteq (\cal V_c^+ \cap R_c) \oplus (\cal V_c^- \cap R_c)$, hence Equation \eqref{eqn:Rdecomp} holds.

    Conversely, if $R_c$ decomposes as in Equation~\eqref{eqn:Rdecomp} for every $c\in C$, then each $r\in R_c$ can be written as $r=r^+ + r^-$ where $r^\pm \in \cal V^\pm_c \cap R_c$. Hence $\tau(r) = r^+ - r^- \in R_c$, so $\tau(R_c) = R_c$. Since this holds for every $c \in C$, it follows that $\tau(R) = R$.
\end{proof}

Composition of generalised isometries works well: the composition $R' \circ R$ of two Courant algebroid relations that are generalised isometries and that compose cleanly is also a generalised isometry, see \cite[Proposition 5.7]{Vysoky2020hitchiker}.

\begin{example}[\textbf{$\boldsymbol{Q(K)}$ is not a Generalised Isometry}] \label{eg:redrelationnotGI}
Let us look at the relation $Q(K)$ of Section \ref{ssec:reductionrelation}, where $K$ is an isotropic subbundle\footnote{We always assume that $K$ has non-zero rank.} inducing a regular foliation $\cF$ of $M$ given by the integral manifolds of $\rho(K)$ with smooth leaf space $\cQ=M/\cF$ and projection map $\varpi:M\to\cQ$. Take generalised metrics $\mathscr{V}^+ \subset E$ and $V^+ \subset \red E$: a generalised metric $V^+$ may be constructed from $\mathscr{V}^+$ whenever $W=\mathscr{V}^+$ satisfies Equation \eqref{eqn:reductionW} of Proposition~\ref{prop:reducedKsubbundle}. This is similar to \cite[Example 5.6]{Vysoky2020hitchiker}, where Equation~\eqref{eqn:reductionW} is tantamount to $\sfG$-equivariance; therein Vysok\'{y} concludes that $Q(K)$ is not a generalised isometry. 

In our case we argue as follows. 
If $Q(K)$ is a generalised isometry between $\mathscr{V}^+$ and $V^+$, at a point $(m, \varpi(m))\in \gr(\varpi)$ we could write
\begin{align}
    Q(K)_{(m, \varpi(m))} = \big(\cV_{(m, \varpi(m))}^+ \cap Q(K)_{(m, \varpi(m))}\big) \oplus \big( \cV_{(m, \varpi(m))}^- \cap Q(K)_{(m, \varpi(m))}\big) \ ,
\end{align}
where $\cV^\pm = \mathscr{V}^\pm \times V^\pm.$
Since $\big(k, 0^{\red E}_{\varpi(m)} \big) \in Q(K)_{(m, \varpi(m))}$, for $k\in K_m$, this implies
\begin{align}
    \big(k, 0^{\red E}_{\varpi(m)}\big) = \big(k^+,  0^{\red E}_{\varpi(m)} \big) + \big(k^-, 0^{\red E}_{\varpi(m)} \big) \ ,
\end{align}
where $\big(k^\pm, 0^{\red E}_{\varpi(m)} \big)\in \cV_{(m, \varpi(m))}^\pm \cap Q(K)_{(m, \varpi(m))}$. Since these live inside $Q(K)$, and the kernel of $\natural$ is $K$, it follows that $k^\pm \in K_m$. 
Thus $k^\pm \in K_m \cap \mathscr{V}_m^\pm$. In particular
\begin{align}
    0 < \ip{k^+,k^+}_{E} = 0 \ ,
\end{align}
which is a contradiction. Hence $Q(K)$ cannot be a generalised isometry.
 
Notice that the isotropy of $K$ here plays a crucial role. It is shown in \cite[Example 5.6]{Vysoky2020hitchiker} that $Q(K)$ may be a generalised isometry for non-isotropic $\sfG$-actions only if $K \cap K^\perp = \set{0}.$ We will explore non-isotropic reductions by foliations in future work.
\end{example}

Example~\ref{eg:redrelationnotGI} motivates a characterising property of generalised isometries given by
\begin{proposition}\label{prop:GIhavetrivialK}
Let $R \colon E_1 \rel E_2$ be a generalised isometry between $\tau_1\in \mathsf{Aut}(E_1)$ and $\tau_2\in \mathsf{Aut}(E_2)$ supported on $C\subseteq M_1\times M_2$. For each $c=(m_1,m_2)\in C$, define subsets $K_1\subset E_1$ and  $K_2 \subset E_2$ by
\begin{align}\label{eqn:subbundleK1}
    (K_1)_{m_1} \times \set{0^{E_2}_{m_2}} = R_c \cap \big((E_1)_{m_1} \times \set{0_{m_2}^{E_2}}\big) \ , \\[4pt] \label{eqn:subbundleK2}
    \set{0^{E_1}_{m_1}} \times (K_2)_{m_2} = R_c \cap \big( \set{0^{E_1}_{m_1}} \times (E_2)_{m_2}\big) \ .
\end{align}
Then $(K_i)_{m_i} = \set{0^{E_i}_{m_i}}$ for $i=1,2$.
\end{proposition}
\begin{proof}
Let $c=(m_1,m_2)\in C$ and $k_1\in (K_1)_{m_1}$. Then
\begin{align*}
    (\tau_1\times \tau_2)(k_1,0) = \big(\tau_1(k_1),0\big) \ \in \ (K_1)_{m_1} \times \set{0^{E_2}_{m_2}} \ ,
\end{align*}
hence $(K_1)_{m_1}$ is invariant under $\tau_1$. Since $R$ is isotropic, it follows that $(K_1)_{m_1}$ is also isotropic, as $0 = \ip{(k_1,0),(k'_1,0)} = \ip{k_1,k'_1}_{E_1}-\ip{0,0}_{E_2}$ for all $k_1,k_1'\in(K_1)_{m_1}$, hence
\begin{align*}
    \cG_1(k_1,k_1) = \ip{k_1,\tau_1(k_1)}_{E_1} = 0 \ .
\end{align*}
Therefore $(K_1)_{m_1} = \set{0^{E_1}_{m_1}}$ for every $m_1\in {\rm pr}_1(C)$.

Similarly $(K_2)_{m_2} = \set{0^{E_2}_{m_2}}$ for every $m_2\in {\rm pr}_2(C)$. 
\end{proof}

A useful uniqueness result for generalised metrics now follows from
\begin{corollary}\label{cor:GMisunique}
    Let $R\colon E_1 \rel E_2$ be a generalised isometry between $\tau_1\in \mathsf{Aut} (E_1)$ and $\tau_2\in \mathsf{Aut}(E_2)$ which is also a generalised isometry between $\tau_1\in \mathsf{Aut}(E_1)$ and $\tau_2'\in \mathsf{Aut}(E_2)$. Then $\tau_2|_{{\rm pr}_2(R)} = \tau_2'|_{{\rm pr}_2(R)}$, where ${\rm pr}_2:E_1 \times \overline{E}_2 \to \overline{E}_2$ is the projection.
\end{corollary}
\begin{proof}
    Let $(e_1,e_2) \in R$. Then $(\tau_1(e_1), \tau_2(e_2))\in R$ and $(\tau_1(e_1), \tau_2'(e_2)) \in R$, therefore $(0, \tau_2(e_2) - \tau_2'(e_2)) \in R$. Hence $\tau_2(e_2) = \tau_2'(e_2)$ for all $e_2 \in {\rm pr}_2(R)$ by Proposition~\ref{prop:GIhavetrivialK}.
\end{proof}

\medskip

\subsection{Transverse Generalised Isometries}\label{ssec:transisometries}~\\[5pt]
Example \ref{eg:redrelationnotGI} and Proposition \ref{prop:GIhavetrivialK} highlight the need to go beyond the notion of generalised isometry.
In particular, we wish to extend the definition of generalised isometry to the case when the subsets $K_i\subset E_i$ defined by Equations \eqref{eqn:subbundleK1} and \eqref{eqn:subbundleK2} are non-trivial. For this, we assume from now on that the subsets $K_i$ are subbundles supported on $M_i.$ Hence ${\rm pr}_i(C)=M_i$, where ${\rm pr}_i\colon M_1\times M_2 \to M_i$ are the Cartesian projections for $i=1,2$. 

\medskip

\subsubsection{Transverse Generalised Metrics}~\\[5pt]    
We recall the notion of transverse generalised metrics, as first introduced in \cite{Severa2019transverse}, which will aid in understanding the properties of our extension of generalised isometry.

\begin{definition}
Let $E$ be an exact Courant algebroid over $M$ and $K$ an involutive isotropic subbundle of $E$. A subbundle $W \subset E$ of rank $\mathrm{rk}(W) = \dim(M)$ is a \emph{pre-}$K$-\emph{transverse generalised metric} if
$K \subset W \subset K^\perp$ and
\begin{align}
\langle w, w\rangle_E > 0 \ ,
\end{align}
for all $w \in W$ with $w \notin K$. A pre-$K$-transverse generalised metric $W$ is a $K$-\emph{transverse generalised metric} if it is invariant with respect to $K$, in the sense that
\begin{align} \label{eqn:invcond}
\llbracket \mathsf{\Gamma}(K), \mathsf{\Gamma}(W) \rrbracket_E \subseteq \mathsf{\Gamma}(W) \ .
\end{align}
An exact Courant algebroid $E$ endowed with a transverse generalised metric $W$ is denoted by~$(E,W)$.
\end{definition}

Note that this is the same condition appearing in the context of reducible subbundles of exact Courant algebroids, see Remark \ref{rmk:reductedKsubbundle}.

\begin{remark}[\textbf{Transverse Generalised Metrics as Graphs}] \label{rmk:splittransverse}
Similarly to Remark \ref{rmk:Vplusgeneralised} (see Equation~\eqref{eqn:+1eigen}), suppose that the restriction of the anchor $\rho_E|_K\colon K \to TM$ to $K$ is injective, and that $W$ is a pre-$K$-transverse generalised metric. Then in a given splitting, there is a (degenerate) symmetric bilinear pairing $g$ and a (degenerate) two-form $b$, with $\ker (b) \supset \ker (g)$, such that
\begin{align}
W \simeq \mathrm{gr}(g+b) = \set{ v+ \iota_v (g+b)  \in \IT M \ \vert \  v \in TM } \ .
\end{align}
In this splitting, we have the isomorphism $K \simeq \ker(g)$.

If $W$ is a $K$-transverse generalised metric, then 
\begin{align}
    \pounds_X g = \pounds_X b = 0 \qquad \text{and} \qquad \iota_X H =0 \ ,
\end{align}
for every $X\in \mathsf{\Gamma}(\rho_E(K))$, where $H\in \mathsf{\Omega}^3_{\rm cl}(M)$ represents the \v Severa class of $E$. The converse is also true.
\end{remark}

The result of Remark~\ref{rmk:splittransverse} allows an interpretation in terms of Riemannian submersions and Riemannian foliations, as shown in \cite{Severa2019transverse}. 

\begin{example}[\textbf{Riemannian Foliations}] \label{rmk:riemfol}
This construction extends to any manifold $(M, g)$ endowed with a foliation $\cF$ such that the degenerate symmetric bilinear tensor  $g$ satisfies $\ker(g)=T\cF$ and is leaf-invariant:
\be
\pounds_X g = 0 \ ,
\ee
for all $X \in \mathsf{\Gamma}(T\cF).$ In other words, $(M, g, \cF )$ is  a Riemannian foliation. It is easy to see that 
\be
W= \mathrm{gr}(g)  
\ee 
is a $K$-transverse generalised metric, where $K \simeq T \cF$ is an involutive isotropic subbundle of $\IT M$ and
$K^\perp = TM \oplus \mathrm{Ann}(T\cF).$
\end{example} 

\medskip

\subsubsection{Relations for Transverse Generalised Metrics}~\\[5pt]
The main goal of this section may be rephrased as giving a  definition of isometry of transverse generalised metrics. In order to provide more insight into what our definition should look like, let us first explore the properties of the most natural way to extend generalised isometries: by adapting Equation \eqref{eqn:classgeniso3} to (pre-)transverse generalised metrics.

\begin{remark}[\textbf{Isomorphisms and Transverse Generalised Metrics}]
Suppose $E_1$ and $E_2$ are exact Courant algebroids over $M_1$ and $M_2$, respectively, which are isomorphic via a map $\Phi\colon E_1 \to E_2$ covering a diffeomorphism $\phi\colon M_1 \to M_2$.
Suppose $K_1$ and $K_2$ are involutive isotropic subbundles of $E_1$ and $E_2$, respectively, and that $W_i$ are pre-$K_i$-transverse generalised metrics for $i=1,2$. 
The natural definition would be to call $\Phi$ an isometry of these metrics if $\Phi(W_1) = W_2$. This turns out to be too restrictive. 

Let us discuss the conditions arising from this choice. Upon choosing splittings for $E_i,$ we have $\Phi = \overline{\phi} \circ \e^{B}\,$ for some $B \in \mathsf{\Omega}^2(M_1)$, and $W_i \simeq \mathrm{gr}(g_i + b_i)$ for some $g_i \in \mathsf{\Gamma}(\midodot^2 T^*M_i)$ and $b_i \in \mathsf{\Omega}^2(M_i).$ As in Proposition \ref{prop:genisochar}, $\Phi(W_1) = W_2$ if and only if $\phi^*g_2 = g_1$ and $\phi^* b_2 = b_1 + B.$
However, the condition that $\phi$ should map the degenerate symmetric tensors $g_i$ into each other implies that
\begin{align}
 \phi_*\big(\ker(g_1)\big) = \ker(g_2) \ ,   
\end{align}
which is equivalent to requiring $\Phi(K_1)= K_2.$ In particular, if $M_i$ are endowed with Riemannian foliations, by Example \ref{rmk:riemfol} this is tantamount to requiring that $\phi$ is a foliation-preserving diffeomorphism.
\end{remark}

We would like to introduce a notion of generalised isometry for transverse generalised metrics that is not constrained to being foliation-preserving. This choice is motivated by the possibility to formalise the description of T-duality where T-dual manifolds may not be diffeomorphic, and hence the foliations defining the reductions to these backgrounds cannot be bijectively mapped into each other (see Section \ref{sec:T-duality} for more details). 
Instead we consider $W^+_i \coloneqq W_i/K_i$. We then require \smash{$\Phi(\widetilde W^+_1) = \widetilde W^+_2$} for some lifts \smash{$\widetilde W^+_i$ of $W^+_i$} to $E_i$. In order to extend this condition to a general Courant algebroid relation, we mimic Definition \ref{defn:generalisedisometry2} using this idea.

\begin{definition}\label{defn:transgeniso}
    Let $E_1$ and $E_2$ be Courant algebroids over $M_1$ and $M_2$, respectively, and let $R \colon E_1 \rel E_2$ be a Courant algebroid relation supported on $C \subset M_1 \times M_2$. 
    Suppose $K_1$ and $K_2$ are involutive isotropic subbundles of $E_1$ and $E_2$, respectively, and that $W_i$ are pre-$K_i$-transverse generalised metrics for $i=1,2$. Set $W^+_i \coloneqq W_i/ K_i$ and $W^-_i \coloneqq W_i^\perp / K_i$.
    
    Then $R$ is a \emph{transverse generalised isometry}\footnote{We are aware that our terminology `transverse generalised isometry' might be a misnomer, since these relations do not necessarily preserve the subbundles $K_i$ defining the transversality conditions. We chose this name for the sake of simplicity, since the $K_i$-preserving case is included in our definition.} \emph{between $W_1$ and $W_2$} if there are pointwise lifts \smash{$\widetilde W^\pm_i \subseteq W_i$} of $W^\pm_i$ to $E_i$ such that
    \begin{align}\label{eqn:transdecomp}
        R_c\cap (\cal W^+_c \oplus \cal W^-_c) = (R_c \cap \cal W^+_c) \oplus (R_c\cap \cal W^-_c) \ ,
    \end{align}
    for all $c\in C$,   where \smash{$\cal W^\pm = \widetilde W^\pm_1 \times \widetilde W^\pm_2$}. A transverse generalised isometry $R$ is \emph{regular} if the lifts \smash{$\widetilde W_i^\pm$} are subbundles of $E_i$ for $i=1,2$.
\end{definition}

Note that the inclusion $\supseteq$ in Equation~\eqref{eqn:transdecomp} always holds.

Definition~\ref{defn:transgeniso} may be better understood by looking at the case of a classical Courant algebroid morphism.

\begin{remark}[\textbf{Lifts for Transverse Generalised Isometries}]
     The lift $\widetilde W^+_i$ is a choice of subbundle of $W_i$ such that
     \begin{align}\label{eqn:+1eigenofW}
        \ip{w_i,w_i}_{E_i}>0 \ ,
     \end{align}
    for every $w_i\in \widetilde W^+_i$. In other words, it is a choice of splitting $s^+_{m_i}\colon (W^+_i)_{m_i} \to (W_i)_{m_i}$, at each point $m_i\in M_i$,  of the short exact sequence
    \begin{align}\label{eqn:eigenbundlesequence}
        0 \longrightarrow (K_i)_{m_i} \longrightarrow (W_i)_{m_i} \longrightarrow (W^+_i)_{m_i} \longrightarrow 0 \ ,
    \end{align}
    with $(\widetilde W^+_i)_{m_i} = s^+_{m_i}\big((W_i^+)_{m_i}\big)$. Similarly, the lift $\widetilde W^-_i$ is a choice of negative definite subbundle with respect to the pairing $\ip{\, \cdot \, , \, \cdot \,}_{E_i}$.

    We may think of $\widetilde{W}_i^+$ as the counterpart of $V_i^+$ for transverse generalised metrics, whereas $K$ can be interpreted as the degenerate part of $W_i$. Definition \ref{defn:transgeniso} is thus the analogue of Definition \ref{defn:generalisedisometry2} such that only the positive definite components of $W_i$ are related.
    Note that we recover Definition \ref{defn:generalisedisometry2} for generalised metrics when $K_i$ is the zero subbundle. 
    
    If the splitting $s^+$ varies smoothly, then we may consider the short exact sequence \eqref{eqn:eigenbundlesequence} to be a short exact sequence of vector bundles and the splitting is a bundle map, thus $R$ is a regular transverse generalised isometry. This will be the case in Section \ref{sec:T-duality}.
\end{remark}

The main case in which regular transverse generalised isometries exist is given by

\begin{theorem}\label{thm:qktgi}
Consider the reduction relation $Q(K)\colon E \mor \red E$ over $\varpi \colon M \to \cal Q$  of Section~\ref{ssec:reductionrelation} and let $V^+ \subset {\red E}$ be a generalised metric on $\red E.$ There exists a unique $K$-transverse generalised metric $W$ on $E$ such that $Q(K)$ is a regular transverse generalised isometry between $W$ and $V^+$. 

Conversely, if there is a $K$-transverse generalised metric $W$ on $E$, then $V^+\coloneqq\natural(W)$ is a generalised metric on $\red E$ and $Q(K)$ is a regular transverse generalised isometry between $W$ and~$V^+$.
\end{theorem}

\begin{proof}
Take a generalised metric $V^+$ on $\red E$. Recall that for every $q\in \cQ$ there is an isomorphism $\mathscr{J}_{q,m} \colon \red E{}_q \to K^\perp_m / K_m$ given by Equation~\eqref{eqn:bethisom}. Define $W_m^+ = \mathscr{J}_{q,m}(V_q^+)$. For two points $m,m' \in \cF_q$, take $w^+\in W^+_m$ and $w'^+\in W^+_{m'}$ such that $\mathscr{J}^{-1}_{q,m}(w^+) = \mathscr{J}^{-1}_{q,m'}(w'^+)$. By the construction in the sketched proof of Theorem \ref{thm:foliationreduction}, there is a basic section $w$ such that, under the projection $K^\perp \to K^\perp / K$, $w$ maps to $w^+$ at $m$ and $w'^+$ at $m'$. Take $W$ to be the span of such basic sections $w$. 
Then $K \subset W \subset K^\perp$ and $\ip{w,w}_E>0$ for every $w\in W$ such that $w\notin K$. Hence $W$ is a pre-$K$-transverse generalised metric on $E$. 

By Remark \ref{rmk:reductedKsubbundle}, for $w\in W$ we can write $w = \sum_i\, f_i\, w_i$, for $f_i \in C^\infty (M)$ (so that the sum is locally finite) and $ w_i \in \mathsf{\Gamma}_{\rm{bas}}(W)$. Then
\begin{align}
    \llbracket k, w \rrbracket_E = \sum_i\, f_i\, \llbracket k , w_i \rrbracket_E + \sum_i\,\big(\rho_E(k)\cdot f_i\big)\, w_i \ ,
\end{align}
for each $k\in K$. Since $w_i$ are basic, it follows that 
$\llbracket \mathsf{\Gamma}(K), \mathsf{\Gamma}(W) \rrbracket_E \subseteq \mathsf{\Gamma}(W),$
hence $W$ is a $K$-transverse generalised metric and $\natural(W) = V^+$.

Uniqueness is seen as follows: Suppose there exists another $K$-transverse generalised metric $W'$ such that $\natural(W') = V^+$. By Remark \ref{rmk:reductedKsubbundle}, it follows that $W'$ is spanned pointwise by $\set{w'_i} \subset \mathsf{\Gamma}_{\mathrm{bas}}(W')$. Thus $v_i = \natural(w'_i)$ give a set of sections spanning $V^+$ pointwise. Now we apply the previous construction to $v_i$ to construct $W$. 
Hence we get $W$ spanned by $w_i$ with $\natural(w_i) = \natural(w'_i)$. 
Therefore $w_i$ and $w'_i$ differ by an element in $K$, so $W=W'$. 

Conversely, starting with a $K$-transverse generalised metric $W$ on $E$, then $V^+ \coloneqq\natural(W)$ defines a generalised metric on $\red E$, since the $K$-transversality condition makes $V_+$ a well-defined subbundle of $\red E$ over $\cQ$ by Proposition \ref{prop:reducedKsubbundle} and
\begin{align}
    \ip{v,v}_{\red E} = \ip{w,w}_E \geq 0 \ ,
\end{align}
for every $v\in V^+$, with equality if and only if $w\in K$, i.e. $v=0$. It also follows that $\rk(V^+) = \dim(\cQ)$.

We now show, in both cases, that $Q(K)$ is a regular transverse generalised isometry between $W$ and $V^+$. Fix $p = (m, \varpi(m))\in \gr(\varpi) \subset M \times \cQ$. Take arbitrary smooth splittings $\widetilde W{}^\pm$ of the sequence \eqref{eqn:eigenbundlesequence}, i.e. a splitting of the short exact sequence of vector bundles
\begin{align}
 0 \longrightarrow K \longrightarrow W \longrightarrow W^+ \longrightarrow 0 
\end{align}
over $M$, and set $\cal W^\pm =\widetilde W{}^\pm \times V^\pm$. Then
\begin{align}
    Q(K)_p \cap (\cal W^+ \oplus \cal W^-)_p = \set{(e, \natural(e))\, | \, e \in K_m^\perp \cap (\widetilde W^+_m \oplus \widetilde W^-_m )} \ .
\end{align}

Thus if $(e, \natural(e))\in Q(K)_p \cap (\cal W^+ \oplus \cal W^-)_p $, then 
\begin{align*}
    \big(e,\natural(e)\big) = \big(e_+ + e_-, \natural(e_+ + e_-)\big) \ ,
\end{align*}
with $e_\pm \in \widetilde W_m^\pm$. By construction $\natural(\widetilde W^\pm) = V^\pm$. Hence
\begin{align*}
    \big(e,\natural(e)\big) &= (e_+ + e_-, \red e{}_+ + \red e{}_-) \\[4pt]
    &= (e_+, \red e{}_+) + (e_-, \red e{}_-) \ \in \ \big(Q(K)_p \cap \cal W^+_p\big) \oplus \big(Q(K)_p\cap \cal W^-_p\big) \ ,
\end{align*}
where $\red e{}_\pm = \natural(e_\pm) = \natural(e)_\pm \in V ^\pm$.
Thus Equation~\eqref{eqn:transdecomp} follows, and $Q(K)$ is a regular transverse generalised isometry between $W$ and $V^+$.
\end{proof}

\begin{example} \label{eg:pullbacktransverse}
For a twisted standard Courant algebroid $E=(\IT M,H)$, a generalised metric on $\red E = (\IT \cQ,\red H)$ is given by a metric ${\red g}\in \Gamma(\midodot^2 T^*\cal Q) $ and a two-form ${\red b}\in \mathsf{\Omega}^2(\cal Q).$ Then the pullbacks $\varpi^*{\red g}$ and $\varpi^*{\red  b}$ define a $K$-transverse generalised metric $W$ such that $\natural(W) = V^+$, and $Q(\cF)$ is a regular transverse generalised isometry between $W$ and $V^+$.
\end{example}

\medskip

\subsection{Composition of Transverse Generalised Isometries}\label{ssec:transisomcomposition}~\\[5pt]
Composition of transverse generalised isometries presents a multi-faceted problem by nature of its less restrictive construction compared to generalised isometries. To elucidate why this is so, suppose $R \colon E_1\rel E_2 $ and $ R' \colon E_2 \rel E_3$ are Courant algebroid relations composing cleanly, and let $W_i$ be transverse generalised metrics on $E_i$ for $i=1,2,3$. Suppose further that $R$ is a transverse generalised isometry between $W_1$ and $W_2$, and $R'$ is a transverse generalised isometry between $W_2$ and $W_3$. The problems encountered are two-fold.

Firstly, since the lifts $\widetilde W^\pm_i$ of $W^\pm_i$ are not unique (as would be the case for a generalised metric) it may not be the case that the lift \smash{$\widetilde W_2^\pm$} making $R$ a transverse generalised metric is the same as the lift \smash{$\widetilde W_2'^\pm$} making $R'$ a transverse generalised metric. 

Secondly, even if we had $\widetilde W_2^\pm = \widetilde W_2'^\pm$, this may still not be enough. For instance, suppose \smash{$(e_1,e_3) \in (R' \circ R) \cap (\widetilde W_1^+ \times \widetilde W_3^+)$}. By definition, there exists $e_2$ such that $(e_1,e_2)\in R$ and $(e_2,e_3)\in R'$. But it is in no way guaranteed that \smash{$e_2\in \widetilde W_2^+$}.

The first problem is discussed in Appendix \ref{app:changesplitting}. For now we assume that $\widetilde W_2^\pm = \widetilde W_2'^\pm$ and address the second problem through

\begin{proposition}
    Suppose that $R_1\colon E_1\rel E_2$ and $R_2\colon E_2 \rel E_3$ are Courant algebroid relations supported on $C_1$ and $C_2$, respectively, that compose cleanly. 
    For $i=1,2,3$ let $K_i$ be isotropic involutive subbundles of $E_i$, and let $W_i$ be pre-$K_i$-transverse generalised metrics such that $R_i$ is a transverse generalised isometry between $W_i$ and $W_{i+1}$. Suppose there are splittings \smash{$\widetilde W_1^\pm$}, \smash{$\widetilde W_2^\pm$} and \smash{$\widetilde W_3^\pm$} giving the decomposition \eqref{eqn:transdecomp} for $R_1$ and $R_2$. Let \smash{$\cal A^\pm = \widetilde W_1^\pm \times \widetilde W_2^\pm \times \widetilde W_2^\pm \times \widetilde W_3^\pm$}. If
    \begin{align}\label{eqn:rankcondition}
        \rk\big((R_2\diamond R_1)_{(c_1,c_2)} \cap (\cal A^+ \oplus \cal A^-)_{(c_1,c_2)}\big) > 0 \ ,
    \end{align}
    for every $(c_1,c_2)\in C_2 \diamond C_1$, then $R_2\circ R_1$ is a transverse generalised isometry between $W_1$ and~$W_3$.
\end{proposition}

\begin{proof}
    Take $(e_1,e_3)\in (R_2 \circ R_1)\cap (\cal W^+ \oplus \cal W^-)$, where $\cal W^\pm = \widetilde W_1^\pm \times \widetilde W_3^\pm$. Then there exists $e_2\in E_2$ such that $(e_i,e_{i+1})\in R_i$ for $i=1,2$. That is, $(e_1,e_2,e_2,e_3)\in R_2 \diamond R_1$. By Equation~\eqref{eqn:rankcondition} we may assume $e_2=e_2^++e_2^-$ where \smash{$e^\pm_2\in \widetilde W_2^\pm$}. Then
    \begin{align}
        (e_i,e_{i+1}) &= (e_i^+ + e_i^-, e_{i+1}^+ + e_{i+1}^-)\\[4pt]
        &= (e_i^+,e_{i+1}^+) + (e_i^-,e_{i+1}^-) \ \in \ R_i \cap \big((\widetilde W_i^+ \times \widetilde W_{i+1}^+ )\oplus (\widetilde W_i^- \times \widetilde W_{i+1}^- )\big ) \ ,
    \end{align}
    for $i=1,2$. Since $R_i$ is a transverse generalised isometry between $W_i$ and $W_{i+1}$, the decomposition \eqref{eqn:transdecomp} holds for $R_1$ and $R_2$, hence $(e_i^\pm,e_{i+1}^\pm)\in R_i$. It then follows by definition that $(e_1^\pm, e_3^\pm) \in R_2 \circ R_1$, and so the decomposition \eqref{eqn:transdecomp} also holds for $R_2 \circ R_1$.
\end{proof}

\begin{corollary}\label{cor:graphscompose}
    If either $R_1$ or $R_2$ is the graph of a classical Courant algebroid isomorphism, then $R_2\circ R_1$ is always a transverse generalised isometry between $W_1$ and $W_3$.
\end{corollary}

\begin{proof}
    Suppose $R_1 = \gr(\Phi)$, where $\Phi\colon E_1\to E_2$ is a Courant algebroid isomorphism over a diffeomorphism $\phi\colon M_1 \to M_2$ (hence $C_1 = \gr(\phi))$. Since $\Phi$ is an isomorphism, we have already seen that \smash{$\widetilde W_2^\pm = \Phi(\widetilde W_1^\pm)$}. To show that Equation~\eqref{eqn:rankcondition} holds, fix $(c_1,c_2)\in C_2 \diamond C_1$.
    Then at this point
    \begin{align}
        R_2\diamond \gr(\Phi) = \set{(e_1, \Phi(e_1), \Phi(e_1), e_3) \, | \, e_1\in E_1 \ , \  (\Phi(e_1),e_3) \in R_2} \ .
    \end{align}
    But if $e_1\in \widetilde W_1^+ \oplus \widetilde W_1^-$ then $\Phi(e_1) \in \widetilde W_2^+ \oplus \widetilde W_2^-$, and hence Equation~\eqref{eqn:rankcondition} holds.

    If $R_2$ is the graph of a classical Courant algebroid isomorphism $\Phi$, then this argument can be applied to \smash{$(R_1 \circ \gr(\Phi))^\top  = \gr(\Phi^{-1}) \circ R_1^\top $}, as the transpose of a transverse generalised isometry is a transverse generalised isometry.
\end{proof}

\section{T-duality as a Courant Algebroid Relation}\label{sec:T-duality}

Recall that infinitesimal symmetries of the Wess-Zumino functional $S_{H_1}$ of a string sigma-model are characterised by the Dorfman bracket of the standard Courant algebroid $(\IT M_1,H_1)$. A Courant algebroid isomorphism thus maps symmetries of $S_{H_1}$ to symmetries of $S_{H_2}$, with $H_1$ and $H_2$ related by Equation~\eqref{eqn:actionSemi}, hence giving the same equations of motion. This however requires the string backgrounds to be diffeomorphic. T-duality is a correspondence between dynamics arising from sigma-models whose backgrounds are not necessarily diffeomorphic, but whose symmetries are related in a certain sense. 
Courant algebroid relations preserve the Dorfman bracket even when the backgrounds are not diffeomorphic. This motivates an attempt to reformulate T-duality in terms of Courant algebroid relations.

In this section we formalise this idea, where the Courant algebroid relation may be viewed as the relation induced  by two reductions of an exact Courant algebroid. 
That is, we consider an exact Courant algebroid $E$ over a manifold $M$ and take isotropic subbundles $K_1 \subset E$ and $K_2 \subset E$, inducing regular foliations by $\rho_{E}(K_1)$ and $\rho_{E}(K_2)$ of $M$, and apply the reduction procedure of Theorem \ref{thm:foliationreduction}.  Under the right conditions, the reduced Courant algebroids inherit a Courant algebroid relation from the reductions. In other words, we will explore circumstances under which the Wess-Zumino functionals of our sigma-models are related.

\medskip

\subsection{Topological T-duality}\label{ssec:topoligcalTduality}
\label{sec:reductionofCAisomorphisms}~\\[5pt]
Let $E$ be an exact Courant algebroids over $M$ endowed with isotropic subbundles $K_1$ and $K_2$. Here and throughout the rest of the paper we will assume that 
\begin{align} \label{eqn:rankK}
\rk (K_1) = \rk(K_2) \ .    
\end{align}
Assume as well that $K_1^\perp$ and $K_2^\perp$ have enough basic sections and denote by $\cF_i$ the foliation of $M$ given by the integral manifolds of the distribution $\rho_{E}(K_i).$ Suppose that the leaf spaces $\cQ_i=M / \cF_i$ have a smooth structure, hence there are unique surjective submersions $\varpi_i \colon M \to \cQ_i$ for $i=1,2.$ 

From Section \ref{ssec:reductionrelation} we can then form the Courant algebroid morphisms $Q(K_i)\colon E \mor \red E{}_i$ supported on  $\gr(\varpi_i)$. This gives the diagram
\begin{equation}
\begin{tikzcd}\label{cd:T-dualitycd}
 &\arrow[tail,swap]{dl}{Q(K_1)} E \arrow[tail]{dr}{Q(K_2)}&   \\
\red E{}_1 & & \red E{}_2 
\end{tikzcd}
\end{equation}
which prepares for our first definition of T-duality.  

\begin{definition}\label{def:topologicalTdual}
    Let $E_1$ and $E_2$ be exact Courant algebroids fitting into the diagram \eqref{cd:T-dualitycd}. Then $\red E{}_1$ and $\red E{}_2$  are \emph{T-duality related} over $\cQ_1$ and $\cQ_2$ if the composition
    \begin{equation}\label{eqn:T-dualcomposition}
    \begin{tikzcd}
        R = Q(K_2) \circ Q(K_1)^\top \colon 
        \red E{}_1 \arrow[dashed]{r} & \red E{}_2
    \end{tikzcd}
    \end{equation}
    is a Courant algebroid relation supported on 
    \begin{align}
     C=\set{(\varpi_1(m),\varpi_2(m)) \, | \, m \in M} \ \subset \ \cQ_1 \times \cQ_2 \ .
    \end{align}
    The Courant algebroid relation $R$ is the \emph{T-duality relation} between $\red E{}_1$ and $\red E{}_2$.
\end{definition}

\begin{remark}
    We first discuss the smoothness of the supporting submanifold $C$, as in general it may not be smooth. The most interesting case is when $\cQ_1$ and $\cQ_2$ are fibred manifolds over a common base manifold $\cB$. In this case there is the commutative diagram
    \begin{equation}\label{cd:fibredproduct}
        \begin{tikzcd}[row sep = 10pt]
             & M \arrow[swap]{dl}{\varpi_1} \arrow{dr}{\varpi_2} & \\
            \cQ_1 \arrow[swap]{dr}{\pi_1} &  & \cQ_2 \arrow{dl}{\pi_2} \\
             & \cB & 
        \end{tikzcd}
    \end{equation}
    Since the fibred product $\cQ_1 \times_{\cB} \cQ_2$ also fits into the diagram \eqref{cd:fibredproduct}, there is a smooth map $M \to \cQ_1 \times_{\cB} \cQ_2$ given by the quotient of $M$ by the foliation induced by the distribution $T\cF_1 \cap  T\cF_2$. Thus in this case the rank of $T\cF_1 \cap T\cF_2$ must be constant.\footnote{We will see below that this  must also be true in the general case without a common base $\cB$.}
    \end{remark}
    
\begin{lemma} \label{lemma:cdsmoothcirclecomp}
    If $M$ fits into the commutative diagram \eqref{cd:fibredproduct}, then $C = \cQ_1 \times_{\cB} \cQ_2$ is smooth.
\end{lemma}

\begin{proof}
    Note that $C \subseteq \cQ_1 \times_{\cB} \cQ_2$. For the opposite inclusion, consider a point $(q_1,q_2) $ in $\cQ_1 \times_{\cB} \cQ_2$. Then there exist $m,m' \in M$ such that $q_1 = \varpi_1(m)$ and $ q_2 = \varpi_2(m')$, and $b\coloneqq \pi_1(\varpi_1(m)) = \pi_2(\varpi_2(m'))\in\cB$. Since also $\pi_2(\varpi_2(m))=b$, with $q_0 \coloneqq \varpi_2(m)$ the points $q_2$ and $q_0$ belong to the same fibre of $\pi_2$. Hence there is a path $\gamma:[0,1]\to\cQ_2$ such that $\gamma(0) = q_0$ and $\gamma(1) = q_2$. By the homotopy lifting property we get a path $\tilde{\gamma}$ in $ (\cF_1)_{q_1}$ such that $\tilde{\gamma}(0) = m$, where $(\cF_1)_{q_1}$ is the leaf of the foliation $\cF_1$ given by $\varpi_1^{-1}(q_1)$. We then define \smash{$m''\coloneqq \tilde{\gamma}(1).$} It follows that $\varpi_1(m'') = q_1$ and $\varpi_2(m'') = q_2$, hence $\cQ_1 \times_{\cB} \cQ_2 \subseteq C$.
\end{proof}

Assuming $C$ is smooth, we can now state the conditions under which the T-duality relation $R$ can be formed as

\begin{theorem}\label{thm:topoligicaltd}
Let $E$ be an exact Courant algebroid over $M$ endowed with isotropic subbundles $K_1$ and $K_2$ satisfying the conditions for the diagram \eqref{cd:T-dualitycd} discussed above. Let us also assume that $T\cF_1 \cap T\cF_2$ has constant rank and that $C$ is a smooth manifold.
Then $\red E{}_1$ and $\red E{}_2$ are T-duality related if and only if $K_1\cap K_2$ has constant rank.
\end{theorem}

\begin{proof}
To compose $Q(K_2)$ with $Q(K_1)^\top $, we first check the conditions on the base manifolds. 
Setting $C_1 = \gr(\varpi_1)$ and $C_2 = \gr(\varpi_2)$, the intersection of $C_1^\top  \times C_2$ with $\cQ_1 \times \Delta(M) \times \cQ_2 $ is clean: one can show that they intersect transversally\footnote{We say that submanifolds $S$ and $S'$ of $M$ \emph{intersect transversally} if $T_sS + T_s S' =T_s M$ for each $s \in S\cap S'$. Submanifolds that intersect transversally also intersect cleanly~\cite[Appendix C.3]{hormander2009analysis}.} in $\cQ_1 \times M \times M \times \cQ_2$. For if $v \in T_{\hat m}( \cQ_1 \times M \times M \times \cQ_2)$, for any $\hat{m} \in (C_1^\top  \times C_2) \cap  (\cQ_1 \times \Delta(M) \times \cQ_2)$, then
\begin{align}
    v = ( {\red v}_1, v_1, v_1', {\red v}_2) &= \big( ( \varpi_1)_* ( v_1) , v_1, v_1', ( \varpi_2)_*( v_1')\big) \\ & \quad \, + \big( {\red v}_1 -  ( \varpi_1)_* ( v_1) , 0, 0, {\red v}_2 - ( \varpi_2)_*( v_1')\big) \ ,
\end{align}
where the first term is an element of $T_{\hat m} (C_1^\top  \times C_2)$ and the second term is an element of $T_{\hat m}( \cQ_1 \times \Delta(M) \times \cQ_2)$.
This implies the condition \ref{item:cleanint2} on the base manifolds of Proposition~\ref{prop:cleanintersect}. Next since $C = C_2 \circ C^\top_1 $ is a smooth manifold and $T\cF_1 \cap T\cF_2$ has constant rank, the map $C_2 \diamond C_1^\top  \to C_2 \circ C_1^\top $ is a smooth surjective submersion, thereby satisfying the condition \ref{item:clean2} on the base manifolds of Proposition~\ref{prop:cleancomposition}. 

To check the constant rank criteria of conditions \ref{item:clean2} of Propositions \ref{prop:cleanintersect} and \ref{prop:cleancomposition}, we note that since $Q(K_1) = \set{(e,\natural_{1}(e))\,|\, e\in K_1^\perp}$,
\begin{align}\label{eqn:seconddiamond}
Q(K_2) \diamond Q(K_1)^\top  = \set{ (\natural_{1}(e),e,e,\natural_{2}(e))\,|\, e\in K_1^\perp \cap K_2^\perp} \ .
\end{align}
This is pointwise isomorphic to $K_1^\perp \cap K_2^\perp$, and for $m \in M$ we find
\begin{align}
    \dim\big((K_2^\perp)_m \cap (K_1^\perp)_m\big) &= \rk(K_1^\perp) + \rk(K_2^\perp) - \dim\big((K_1^\perp)_{m}+ (K_2^\perp)_{m}\big) \\[4pt]
    &= \rk(K_1^\perp) + \rk(K_2^\perp) - \rk(E_2) +\dim\big((K_1^\perp)_{m}^\perp \cap (K_2^\perp)_{m}^\perp\big) \\[4pt]
    &= \rk\big(K_1^\perp\big) + \rk(K_2^\perp) - \rk(E_2)  +\dim\big((K_1)_{m} \cap (K_2)_{m}\big) \ .
\end{align}
It follows that condition \ref{item:clean2} of Proposition \ref{prop:cleanintersect} is satisfied if and only if the dimension of $(K_1)_m \cap (K_2)_m$ is independent of $m\in M$.

Finally, the kernel of the projection of Equation~\eqref{eqn:seconddiamond} to $Q(K_2) \circ Q(K_1)^\top $ is $K_1 \cap K_2$. Hence by item \ref{item:clean2} of Proposition \ref{prop:cleancomposition}, the composition is clean if and only if $K_1 \cap K_2$ has constant rank. Thus by Theorem \ref{thm:relationcomposition}, we get the Courant algebroid relation \eqref{eqn:T-dualcomposition}.
\end{proof}

The diagram \eqref{cd:T-dualitycd} now becomes
\begin{equation}\label{cd:Tdualsquare}
\begin{tikzcd}
 & E \arrow[tail]{dr}{Q(K_2)} &  \\
\red E{}_1 \arrow[dashed]{rr}{R}\arrow[dashed]{ur}{Q(K_1)^\top } & & \red E{}_2 
\end{tikzcd}
\end{equation}
Explicitly, if $\natural_i \colon  K_i^\perp \to \bigr( K_i^\perp /K_i \bigl) / \cF_i = \red E{}_i$ denotes the quotient map  for $i=1,\, 2$, then
\begin{align}\label{eqn:T-dualRelation}
    R = \set{(\natural_{1}(e), \natural_{2}(e))\,|\, e \in K_1^\perp \cap K_2^\perp} \ .
\end{align}
Note that the vertical arrow on the left of the diagram points upwards, as the composition of the relations does not commute. 

As mentioned in Section~\ref{sec:Intro}, there is a notion that the T-duality relation should be Lagrangian. Our relation is consistent with this, through
\begin{proposition}
    $R$ is maximally isotropic, hence a Dirac structure in $\red E_1 \times \overline{\red E}_2$.
\end{proposition}
\begin{proof}
  Since $R$ is involutive by construction, we only need to show that it is maximally isotropic in $\red E_1 \times \overline{\red E}_2.$ Note that, from the standpoint of graded symplectic geometry, this is a consequence of the fact that the composition of two Lagrangian relations is again Lagrangian.
    
  Using Equation \eqref{eqn:rankK}, we have
    \begin{align}
        \rk(\red E_1 \times \overline{\red E}_2) = 2\, \big(\rk(E) - 2\,\rk(K_1)\big) \ .
    \end{align} 
    We have already seen that $Q(K_2) \diamond Q(K_1)^\top$ is isomorphic to $K_1^\perp \cap K_2^\perp$ pointwise, and that its projection to $R$ has kernel $K_1 \cap K_2$. Thus
    \begin{align}
        \rk(R) &= \rk\big(K_1^\perp \cap K_2^\perp\big) - \rk\big(K_1 \cap K_2\big ) = \rk(K_1^\perp) + \rk(K_2^\perp) - \rk(E) = \rk(E) - 2\,\rk(K_1)
    \end{align}
    as required.
\end{proof}

\begin{remark}[\textbf{Bisubmersions}] \label{rmk:bisubmersions}
In the picture of T-duality related Courant algebroids of Definition \ref{def:topologicalTdual}, at the level of base manifolds there are surjective submersions
\[
\begin{tikzcd}
 & M \arrow[dr,"\varpi_2"] \arrow[dl,"\varpi_1",swap]  \\
\cQ_1   && \cQ_2
\end{tikzcd}
\]
with the fibres of $\varpi_1$ given by the leaves of the foliation $\cF_1$ and the fibres of $\varpi_2$ given by the leaves of the foliation $\cF_2.$ 

If both surjective submersions have connected fibres and the Lie bracket of vector fields satisfies
\begin{align} \label{eqn:bisubmersioncd}
 \big[\mathsf{\Gamma}\bigl(\ker(\varpi_{1*})\bigr) , \mathsf{\Gamma}\big(\ker(\varpi_{2*})\big)\big] \subset  \mathsf{\Gamma}\bigl(\ker(\varpi_{1*})\bigr) + \mathsf{\Gamma}\bigl(\ker(\varpi_{2*})\bigr) \ ,
\end{align}
then there exist unique (possibly singular) foliations $\red\cF{}_1$ on $\cQ_1$ and $\red\cF{}_2$ on $\cQ_2$ such that~\cite[Corollary~2.16]{Garmendia2019}
\begin{align}
\varpi_{1*}^{-1}(T{\red\cF}{}_1) = (\varpi_2)_*^{-1}(T{\red\cF}{}_2)= \ker(\varpi_{1*}) +\ker(\varpi_{2*}) \ .
\end{align}
This means that there is a \emph{bisubmersion} between $(\cQ_1, \red\cF{}_1)$ and $(\cQ_2, \red\cF{}_2)$; see e.g.~\cite{Garmendia2019} and references therein for the general definition of bisubmersion. 

This construction is an example of \emph{Hausdorff Morita equivalent foliations} in the sense of \cite[Definition~2.1]{Garmendia2019}. Equation \eqref{eqn:bisubmersioncd} is an involutivity condition for the subbundle 
\begin{align}
\ker(\varpi_{1*}) +\ker(\varpi_{2*}) =T\cF_1 + T\cF_2 \ ,
\end{align}
hence $M$ is endowed with a (possibly singular) foliation $\cF$ that induces the Hausdorff Morita equivalent foliations. This is a regular foliation if  $\ker(\varpi_{1*}) \cap \ker(\varpi_{2*})$ has constant rank. This notion will be crucial in the description of the symmetries of the background fields defined on $\cQ_1$ and $\cQ_2$, and we will see how it naturally arises in the construction of T-duality relations.
\end{remark}

\begin{remark}
Courant algebroid reduction is reformulated and extended in the language of graded symplectic reduction by~\cite{Bursztyn:2023onv}. It would be interesting to fit our perspective on T-duality, and more generally Courant algebroid relations, into this framework, particularly in light of the differential graded symplectic geometry approach to generalised T-duality taken in~\cite{Arvanitakis:2021lwo}. 
\end{remark}

\medskip

\begin{example}\label{sec:topogicalTdualitystandardCA}
Our main source of examples will be the following set up: let $E_1=(\IT M_1 ,H_1)$ and $E_2=(\IT M_2, H_2)$ be isomorphic twisted standard Courant algebroids, isomorphic via $\overline{\phi} \circ \e^{B}\,$, where $\overline{\phi} = \phi_* + (\phi^{-1})^*$ and $B\in \mathsf{\Omega}^2(M_1)$ with $\de B = H_1 - \phi^* H_2$, as described in Proposition \ref{decpreiso}.
Suppose that $M_1$ and $M_2$ are foliated by $\cF_1$ and $\cF'_2$, respectively, such that $\cQ_1 = M_1/\cF_1$ with quotient map $\varpi_1$.\footnote{In general we do not require that $\phi_*T\cF_1 = T\cF'_2$. 
}  
Suppose further that  $\iota_{X_1}H_1 = 0$, for all $X_1 \in \mathsf{\Gamma}(T \cF_1)$ and $\iota_{X_2} H_2 = 0$, for all $X_2 \in \mathsf{\Gamma}(T\cF'_2)$. Finally, let $\cQ_2 = M_1 / \cF_2$ with $\cF_2 = \phi^{-1}(\cF'_2)$ and quotient map $\varpi_2$, and let $\red H_1$ and $\red H_2$ be such that $\varpi_1^*\red H_1 = H_1$, $\varpi_2^* \red H_2 = \phi^* H_2$. Then

\begin{proposition}\label{prop:splitcaseconditions}
    If $T\cF_1 \cap T\cF_2$ and $\ker \left(B|_{T\cF_1 \cap T\cF_2}\right) $ have constant rank, $(\IT \cQ_1 , \, {\red H}{}_1)$ and $(\IT \cQ_2 , \, {\red H}{}_2)$ are T-duality related.
\end{proposition}

\begin{proof}
Take $K_1 = T\cF_1$, and $K_2 = \e^{-B}\left( T\cF_2 \right)$. 
Then
\begin{align*}
    K_1 \cap K_2 = \set{v_1 \in T\cF_1 \cap T\cF_2 \,| \, \iota_{v_1} B = 0}=\ker \big(B|_{T\cF_1 \cap T \cF_2}\big) \ .
\end{align*}
By Corollary \ref{cor:qkbasicH}, there are Courant algebroid relations $Q(K_1) \colon (\IT M_1,H_1) \rel (\IT \cQ_1,\red H{}_1)$ and $Q(K_2) \colon (\IT M_1,H_1) \rel (\IT \cQ_2,\red H{}_2).$ 

By Theorem \ref{thm:topoligicaltd}, $R=Q(K_2)\circ Q(K_1)^\top$ is a Courant algebroid relation.
\end{proof} 
\end{example}

\medskip

\subsection{Geometric T-duality}\label{ssec:geometricTduality}~\\[5pt]
We now introduce geometric data into our picture of T-duality, which amounts to incorporating the Polyakov functionals \eqref{eqn:sigmanorm} into our string sigma-models; the symmetries of $S_0$ are generated by Killing vectors of the metric $g$. We will make the following assumptions.
Let $E$ be an exact Courant algebroid over $M$ endowed with involutive isotropic subbundles $K_1$ and $K_2$ that have the same rank. Suppose they satisfy the conditions to fit the diagram \eqref{cd:T-dualitycd} and suppose that the reduced Courant algebroids $\red E{}_1$ and $\red E{}_2$ are T-duality related, i.e. the constant rank assumption of Theorem \ref{thm:topoligicaltd} holds. Denote the resulting T-duality relation by $R$. We will work under these assumptions for the remainder of this section. 

We begin by endowing the reduced Courant algebroids with generalised metrics.

\begin{definition}\label{defn:Tdual}
    The reduced Courant algebroids $(\red E{}_1, V_1^+)$ and $(\red E{}_2, V_2^+)$ are \emph{geometrically T-dual} if $R$ is a generalised isometry between the generalised metrics $V_1^+$ and $V_2^+$.
\end{definition}

The starting point of any T-duality transformation is the identification of the symmetries of the original background. 

\begin{definition} \label{def:tdualitydirections}
 Let $\red E{}_1$ and $\red E{}_2$ be T-duality related Courant algebroids in the sense of Definition \ref{def:topologicalTdual} and suppose that the unique surjective submersions $\varpi_i \colon M \to \cQ_i$ satisfy the involutivity condition \eqref{eqn:bisubmersioncd} discussed in Remark \ref{rmk:bisubmersions}, for $i=1, \, 2$. The subbundle $D_1 \coloneqq T {\red\cF}{}_1 = \varpi_{1*}(T \cF_2)$ of $T\cQ_1$ is the \emph{distribution of T-duality directions}.\footnote{\label{ftn:singularfol} This definition can be adapted to the case where ${\red\cF}{}_1$ is a singular foliation. Then we consider the T-duality directions to be defined by the locally finitely generated $C^\infty_{\rm c}(\cQ_1)$-module $\mathsf{\Gamma}_{\rm c}(D_1)$ of compactly supported vector fields integrating to ${\red\cF}{}_1.$ We do not explicitly discuss this in the following in order to avoid further technicalities. Invariance with respect to a singular foliation relates to the work of Kotov-Strobl~\cite{Kotov2014}.}
\end{definition} 

Since we may form the relations $Q(K_i)$ for $i=1,2$, by Remark \ref{rmk:existenceadapted} there exists splittings $\sigma_i\colon TM \to E$ of the short exact sequences
\[
\begin{tikzcd}
    0 \arrow{r} & T^*M \arrow{r}{\rho^*} & E \arrow{r}{\rho}\arrow[bend left=40]{l}{\sigma^*_i} & TM \arrow{r} \arrow[bend left=40]{l}{\sigma_i} & 0
\end{tikzcd}
\]
which are adapted to $K_i$, where $\sigma^*_i \coloneqq \sigma_i^{\rm t} \circ \flat_{E}$ is the induced left splitting and $\flat_{E} \colon E \to E^*$ is the isomorphism induced by the pairing. By Remark \ref{rmk:propertiesadapted}, $\sigma_i$ induces a splitting $\red \sigma_i$ of $\red E_i$, giving also a left splitting $\red \sigma^*_i$.

We are now ready to relate the distribution $D_1$ to the symmetries of our initial background and begin to justify our terminology `T-duality directions'.

\begin{remark} \label{rmk:D1inv}
     Let us discuss how we can state an invariance condition of a generalised metric solely involving the tensors $(\red g, \red B)$. Let  ${\red \sigma}_1 \colon T\cQ_1 \to \red E_1$ be a splitting of $\red E_1$ descending from an adapted splitting $\sigma_1$ of $E$, then  $V_1^+$ corresponds to a pair $({\red g}{}_1, {\red b}{}_1)$ under the isomorphism $(\red \sigma_1 + \tfrac{1}{2}\rho^*_{\red E_1})^{-1} \colon \red E_1 \to \IT \cQ_1$, i.e. any element ${\red w}{}_1 \in \mathsf{\Gamma}(V_1^+)$ can be written as
    \begin{align}
        {\red w}{}_1 = {\red Y} + \iota_{\red Y} (\red g{}_1 + \red b{}_1) \ ,
    \end{align}
    for some $\red Y \in \mathsf{\Gamma}(T\cQ_1).$
    Let $\red X \in \sfGamma(D_1)$. Then the generalised metric $(\red g_1, \red b_1)$ is invariant with respect to $\red X$ if 
    \begin{align}\label{eqn:genmetricliederivative}
          \pounds_{\red X}(\red g{}_1 + \red b{}_1) = 0 \ .
    \end{align}
    Note that we can write
    \begin{align}
    \begin{split}
     \llbracket \red X , \red w_1 \rrbracket_{\red H{}_1} 
     &=  [\red X, \red Y]  
      +\iota_{\red Y} (\pounds_{\red X}(\red g{}_1 + \red b{}_1))  + \iota_{[\red X, \red Y]}(\red g{}_1 + \red b{}_1) + {\rm pr}_2(\llbracket \red X , \rho(\red w_1) \rrbracket_{\red H{}_1}) \ , 
      \end{split}
    \end{align}
    where ${\rm pr}_2(\llbracket \red X , \red Y \rrbracket_{\red H{}_1}) = \iota_{\red X} \iota_{\red Y} \red H$. 
    Hence by imposing the condition \eqref{eqn:genmetricliederivative}, we obtain
    \begin{align}\label{eqn:D1invariancecond}
       \llbracket \red X , \red w_1 \rrbracket_{\red H{}_1} - {\rm pr}_2(\llbracket \red X , \rho(\red w_1) \rrbracket_{\red H{}_1}) \in \sfGamma(\gr(\red g_1 + \red b_1)) \, ,
    \end{align}
    and vice versa.  
\end{remark}
Condition \eqref{eqn:D1invariancecond} can be stated for a generalised metric $V_1^+$ on $\red E_1$ without making reference to the isomorphism $(\red \sigma_1 + \tfrac{1}{2}\rho^*_{\red E_1})^{-1}$ as follows. 

\begin{definition} \label{def:D1invariance}
    Let $\red \sigma_1 \colon T\cQ_1 \to \red E{}_1 $ be the splitting of the exact Courant algebroid $\red E{}_1$ over $\cQ_1$ induced by an adapted splitting $\sigma_1 \colon TM \to E$ of $E$.   
    The generalised metric $V_1^+$ on $\red E{}_1$ is \emph{invariant with respect to $D_1$}, or \emph{$D_1$-invariant}, if there exists a Lie subalgebra $\mathsf{Iso}(V^+_1)$ of $\sfGamma(D_1)$ which spans $D_1$ pointwise, and is such that
    \begin{align}\label{eqn:genmetricinvariance}
        \llbracket \red{\sigma}_1  (\red X) , \red w{}_1 \rrbracket_{\red E{}_1} - {\rho}{}_{\red E{}_1}^*\bigl({\red \sigma}^*_1(\llbracket \red{\sigma}_1  (\red X) , \red{\sigma}_1(\rho{}_{\red E{}_1}(\red w{}_1)) \rrbracket_{\red E{}_1}) \bigr) \ \in \ \mathsf{\Gamma}(V_1^+) 
    \end{align}
    for every ${\red X} \in \mathsf{Iso}(V^+_1)$ and $\red w{}_1\in \mathsf{\Gamma}(V_1^+).$ The subalgebra $\mathsf{Iso}(V^+_1)$ is called the \emph{isometry subalgebra} of $D_1$.
\end{definition}

\begin{remark}
    On an exact Courant algebroid $E$ over $M$ with \v{S}evera class $[H]$ and an isotropic splitting $\sigma \colon TM\to E$, for $X,Y,Z\in \mathsf{\Gamma}(TM)$ one has
    \begin{align}
        H(X,Y,Z) = \ip{\llbracket \sigma (X), \sigma(Y)\rrbracket_E, \sigma(Z)}_E &= \iota_{Z}\,\sigma^*(\llbracket \sigma (X), \sigma(Y)\rrbracket_E) \\[4pt] &= \ip{\rho_E^*(\sigma^*(\llbracket \sigma (X), \sigma(Y)\rrbracket_E)), \sigma(Z)}_E \ .
    \end{align}
    Hence the condition \eqref{eqn:genmetricinvariance} has the interpretation of an invariance condition which itself imposes no restrictions on the representative $H$ of the \v{S}evera class; Remark \ref{rmk:compatibilitysplit} below makes this point clearer. 
    Recall that a subbundle $W$ is invariant with respect to an involutive subbundle $K$ if satisfies $\llbracket \sfGamma(K), \sfGamma(W) \rrbracket_E \subset \sfGamma (W)$, which implies $\iota_X H=0,$ for all $X \in \sfGamma(\rho(K))$, see \cite{Severa2019transverse}. 
    Definition \ref{def:D1invariance} is formulated in a way that avoids imposing any conditions on the representative $H$ of the \v{S}evera class of $E$. 
\end{remark}

Next we introduce a notion of when the adapted splittings are compatible.

\begin{definition} \label{def:compatiblesplit}
    Let us assume we have a $D_1$-invariant generalised metric $V_1^+$ on $\red E_1$, with isometry subalgebra $\mathsf{Iso}(V^+_1)$. The adapted splittings $\sigma_1$ and $\sigma_2$ to $K_1$ and $K_2$, respectively, are \emph{compatible} if 
    \begin{align}
        \llbracket (\sigma_1 -\sigma_2) (X) , e \rrbracket_{E} = \rho_E^* \bigl( \sigma^*_1 (\llbracket \sigma_1  (X) , \sigma_1(\rho_E(e)) \rrbracket_{E} )\bigr)
    \end{align}
    for every $X \in (\varpi_1)_*^{-1}(\mathsf{Iso}(V_1^+)) \cap \mathsf{\Gamma}(\rho_E(K_2)) \eqqcolon \mathsf{Iso}(W_1) $ and $e\in \mathsf{\Gamma}(E)$.
\end{definition}

\begin{remark} \label{rmk:compatibilitysplit}
Let $\sigma_1 \colon TM \to E$ be an adapted splitting to $K_1$. Then, under the isomorphism $(\sigma_1 + \tfrac{1}{2} \rho_E^*)^{-1} \colon E \to \IT M$, the image of $K_1$ is $T \cF_1$. Moreover, $\sigma_1$ composed with the  isomorphism is the inclusion of $TM$ into $\IT M$, and yields a representative $H_{\sigma_1}$ of the \v{S}evera class of $E$. Meanwhile, an adapted splitting $\sigma_2$ for $K_2$ induces $H_{\sigma_2}$.
The difference is given by   
\begin{align} \label{eqn:differencesplit}
    (\sigma_1 - \sigma_2)(X) = \rho_E^*(\iota_X B)
\end{align}
with $H_{\sigma_1} - H_{\sigma_2} = \de B$. The adapted splittings $\sigma_1$ and $\sigma_2$ are compatible if
\begin{align}
    \llbracket \iota_X B, Y+ \xi \rrbracket_{H_{\sigma_1}} -  {\rm pr}_2(\llbracket X, Y \rrbracket_{H_{\sigma_1}} ) &= -\iota_Y\, \de\, \iota_X B - \iota_Y\, \iota_X H_{\sigma_1} \\[4pt]
    &= -\iota_Y\, \de\, \iota_X B - \iota_Y\, \iota_X\, \de B = -\iota_Y\, \pounds_X B \ ,
\end{align}
for all $Y + \xi \in \sfGamma(\IT M_1),$ where we use that $\iota_{X} H_{\sigma_2} = \iota_X(H_{\sigma_1} - \de B) = 0$ for the second equality. Thus the splittings are compatible if and only if $\pounds_{X}B = 0$ for every $X\in \mathsf{Iso}(W_1)$.
\end{remark}

We are now ready to state our main result about the existence and uniqueness of geometrically T-dual backgrounds in the sense of Definition~\ref{defn:Tdual} through

\begin{theorem}\label{thm:maingeneral}
    Let $E$ be an exact Courant algebroids over $M$. Suppose $E$ is endowed with involutive isotropic subbundles $K_1$ and $K_2$ of the same rank giving T-duality related Courant algebroids $\red E{}_1$ and $\red E{}_2$ with T-duality relation $R$. Assume further that the bisubmersion involutivity condition \eqref{eqn:bisubmersioncd} holds, and that there are adapted splittings $\sigma_1$ and $\sigma_2$ of $E$ which are compatible. Finally, assume that $\red E{}_1$ is endowed with a $D_1$-invariant generalised metric $V_1^+$. Then the following are equivalent:
    \begin{enumerate}[label = (\roman{enumi})]
        \item\label{item:main1} $K_1 \cap K_2^\perp\subseteq K_2 \ . $ 
        \item\label{item:main2} $K_2\cap K_1^\perp \subseteq K_1 \ . $ 
        \item\label{item:main3} There exists a unique generalised metric $V_2^+$ on $\red E{}_2$ such that $R$ is a generalised isometry between $V_1^+$ and $V_2^+$, i.e. $(\red E{}_1, V_1^+)$ and $(\red E{}_2, V_2^+)$ are geometrically T-dual.
    \end{enumerate}
\end{theorem}

\begin{remark}
We split the proof in three parts, beginning with Lemma \ref{lemma:conditionsonPhi} below which establishes the implications
\begin{align}\label{eqn:thmimplications}
    \mathrm{\ref{item:main1}} \Longleftrightarrow \mathrm{\ref{item:main2}} \Longleftarrow \mathrm{\ref{item:main3}} \ .
\end{align}
The final implication requires the construction of the generalised metric $V_2^+$ on $\red E{}_2$. The first step in this construction is to show that one can lift $V_1^+$ to \smash{$\widetilde W_1^+\subset K_1^\perp \cap K_2^\perp$}. Proposition \ref{prop:splittingexists} below establishes this result.
The remainder of the proof is devoted to showing that the induced pre-$K_2$-transverse generalised metric $W_2 \coloneqq K_2 \oplus \widetilde W_1^+$ is a $K_2$-transverse generalised metric, and hence reduces to a generalised metric $V_2^+$.
\end{remark}

\begin{lemma}\label{lemma:conditionsonPhi}
    Suppose that $E$ is an exact Courant algebroid over $M$ such that $\red E{}_1$ and $\red E{}_2$ are T-duality related with T-duality relation $R$. Then the implications \eqref{eqn:thmimplications} hold.    
\end{lemma}

\begin{proof}
To show that \ref{item:main1} implies \ref{item:main2}, we note that 
\begin{align}
 K_2^\perp \cap K_1 \subseteq K_2   \implies K_2^\perp \cap K_1 \subseteq K_1 \cap K_2 \ .
\end{align}
Since $K_2^\perp \supset K_2$, it follows that 
\begin{align}
    K_2^\perp \cap K_1 = K_2 \cap K_1 \ .
\end{align}
By Theorem \ref{thm:topoligicaltd}, $K_2 \cap K_1$ has constant rank. 

Using $\rk(K_1) = \rk(K_2)$ (condition \eqref{eqn:rankK}), we then obtain
\begin{align}
    \rk\big(K_1 \cap K_2^\perp\big) &= \rk\big(K_1\big) + \rk(K_2^\perp) - \rk\big(K_1 + K_2^\perp\big)\\[4pt]
    &= \rk\big( K_2 \big) + \rk(K_2^\perp) - \big(\rk(E_2) - \rk( K_1^\perp \cap K_2)\big) \\[4pt]
    &=\rk(E_2) - \rk(E_2) + \rk\big( K_1^\perp \cap K_2)\big) =
    \rk\big( K_1^\perp \cap K_2)\big) \ .
\end{align}
Hence
\begin{align}
     \rk\big( K_1 \cap K_2 \big) = \rk\big( K_1 \cap K_2^\perp\big) = \rk\big( K_1^\perp \cap K_2\big) \ . 
\end{align}
One always has $K_1^\perp\cap K_2 \supseteq K_1 \cap K_2$. Since their ranks are the same, this becomes an equality. In particular, we obtain condition \ref{item:main2}.

The converse implication \ref{item:main2}$\implies$\ref{item:main1} follows from the same argument.

Finally, suppose that \ref{item:main1} does not hold. Then there is an element $k_1\in K_1 \cap K_2^\perp$ such that $k_1 \notin K_2$. Then by Proposition \ref{prop:GIhavetrivialK}, $R$ cannot be a generalised isometry.
\end{proof}

\begin{proposition}\label{prop:splittingexists}
    Suppose that $E$ is an exact Courant algebroid over $M$ endowed with involutive isotropic subbundles $K_1$ and $K_2$ of the same rank, such that $\red E{}_1$ and $\red E{}_2$ are T-duality related.
    If either of the equivalent conditions \ref{item:main1} or \ref{item:main2} of Theorem \ref{thm:maingeneral} holds,
    then there exists a splitting $s_0 \colon K_1^\perp / K_1 \to K_1^\perp$ of the short exact sequence
    \begin{align}\label{eqn:Kexactsequence}
        0 \longrightarrow K_1 \longrightarrow K_1^\perp \longrightarrow K_1^\perp / K_1 \longrightarrow 0
    \end{align}
    such that $\mathrm{im}(s_0) \subseteq K_2^\perp$, which is unique up to elements of $K_1 \cap K_2$.
\end{proposition}

\begin{proof}
Take an arbitrary splitting $s \colon K_1^\perp / K_1 \to K_1^\perp$ of the short exact sequence \eqref{eqn:Kexactsequence}. If $\mathrm{im}(s) \nsubseteq K_2^\perp$, then there exists $[e]\in K_1^\perp / K_1$ and $k_2\in K_2$ such that $\ip{s([e]),k_2}_{E} \neq 0$. We define the vector bundle map $\beta_s$ over the diffeomorphism $\phi$ by 
\begin{align}
    \beta_s = \flat_{E} \circ s\colon K_1^\perp / K_1 \longrightarrow K_2^* \ ,
\end{align} 
where $\flat_{E} \colon E \to E^*$ is the isomorphism induced by the pairing $\ip{\, \cdot \, , \, \cdot \,}_{E}$ on $E$:
\begin{align}
    \iota_{k_2}\,\beta_s([e_1]) \coloneqq \ip{s([e_1]), k_2}_{E} \ ,
\end{align}
for $[e_1]\in K_1^\perp/K_1$ and $k_2\in K_2$. 

If $k_2\in K_2 \cap K_1$, then
\begin{align}
    \ip{s([e_1]), k_2}_{E} = 0
\end{align}
since $s([e_1]) \in K_1^\perp$. Thus we can further define a map $\bar \beta_0$ by quotienting out this subspace:
\begin{gather}
    \bar \beta_0 \colon K_1^\perp / K_1 \longrightarrow \big (K_2 /(K_2\cap K_1) \big )^* \ ,  \qquad
    \iota_{[k_2]} \bar \beta_0([e_1]) \coloneqq \ip{s([e_1]), k_2}_{E} \ ,
\end{gather}
for $[e_1]\in K_1^\perp/K_1$ and $[k_2]\in K_2/(K_2\cap K_1)$. 

Similarly we define 
\begin{gather}
    \bar \beta \colon  K_1 \longrightarrow \big (K_2 /(K_2\cap K_1) \big )^*\ , 
    \qquad \iota_{[k_2]}\bar \beta(k_1) \coloneqq \ip{k_1, k_2}_{E} \ ,
\end{gather}
for $k_1\in K_1$ and $[k_2]\in K_2/(K_2\cap K_1)$, which is again well-defined.

We show that $\bar \beta$ is surjective: the kernel of $\bar \beta$ is given by
\begin{align}\label{eqn:barbetaker}
    \ker (\bar \beta) = \set{k_1\in K_1 \, | \, \ip{k_1,k_2}_{E} = 0 \ \text{for all} \ k_2 \in K_2} = K_2^\perp \cap K_1 \ .
\end{align}
By assumption, $K_2^\perp \cap K_1 \subseteq K_2$, hence\footnote{See the start of the proof of Lemma \ref{lemma:conditionsonPhi}.}
\begin{align}
    \ker(\bar \beta) = K_2^\perp \cap K_1 = K_2 \cap K_1 \ .
\end{align}
Recalling that $K_2 \cap K_1$ has constant rank and that $\rk(K_1) = \rk(K_2)$, it thus follows that
\begin{align}
    \rk(\bar \beta) &= \rk(K_1) - \rk\big(K_2 \cap K_1\big) \\[4pt]
    &= \rk(K_2) - \rk\big(K_1 \cap K_2\big)
    = \rk \big( (K_2 /(K_2\cap K_1) )^*\big) \ .
\end{align}

Hence $\bar \beta $ is onto, and we can find a map $\alpha \colon K_1^\perp / K_1 \to K_1$ which fits into the diagram
\begin{equation}\label{cd:alphacd}
\begin{tikzcd}[column sep=-1.5em]
    K_1^\perp / K_1 \arrow[dashed]{rr}{\alpha} \arrow[swap]{dr}{\bar \beta_0} & & K_1 \arrow{dl}{\bar \beta} \\[10pt]
    &  \big (K_2 /(K_2\cap K_1) \big )^*
\end{tikzcd}
\end{equation}
We can now define a new splitting $s_0 \colon K_1^\perp / K_1 \to K_1^\perp $ of the short exact sequence \eqref{eqn:Kexactsequence} by
\begin{align}\label{eqn:newsplitting}
    s_0 \coloneqq s - \alpha \ .
\end{align}
It follows that $\text{im}(s_0) \subseteq K_2^\perp$, since for every $k_2 \in K_2$ and $[e_1] \in K_1^\perp / K_1$ we can compute
\begin{align}
    \ip{s_0([e_1]), k_2}_{E} &= \ip{s([e_1]), k_2}_{E} - \ip{\alpha([e_1]), k_2}_{E} \\[4pt]
    &= \iota_{[k_2]}\bar \beta_0([e_1]) - \iota_{[k_2]}\bar \beta (\alpha([e_1])) \\[4pt]
    &= \iota_{[k_2]}\bar \beta_0([e_1]) - \iota_{[k_2]} \bar \beta_0 ([e_1]) = 0 \ .
\end{align}

Hence the new splitting is a map $s_0 \colon K_1^\perp / K_1 \to K_2^\perp\cap K_1^\perp$.
To show uniqueness up to elements of $K_2 \cap K_1$, note that there are two sources of non-uniqueness for $s_0$: one from the choice of the initial splitting $s \colon K_1^\perp / K_1 \to K_1^\perp$, and one from the choice of $\alpha$ fitting the diagram \eqref{cd:alphacd}.

Since the map $\alpha$ is given by the right inverse of $\bar \beta$, it is unique up to elements of the kernel of $\bar \beta$, which by Equation \eqref{eqn:barbetaker} is $K_2 \cap K_1$. 
That is, if $\alpha$ and $\alpha'$ close the diagram \eqref{cd:alphacd}, then there are induced splittings $s_0$ and $s_0'$ respectively defined by Equation \eqref{eqn:newsplitting}, and
\begin{align}
    s_0 - s_0' = \alpha - \alpha ' \colon K_1^\perp / K_1 \longrightarrow K_2 \cap K_1 \ .
\end{align}

Finally, suppose we started with different splittings $s,s'\colon K_1^\perp/ K_1 \to K_1^\perp$. Using the above construction, we obtain splittings $s_0, s_0' \colon K_1^\perp/ K_1 \to K_2^\perp \cap K_1^\perp$. These splittings differ by elements of $K_1$, hence
\begin{align}
    s_0 - s_0' \colon K_1^\perp / K_1 \longrightarrow K_2^\perp \cap K_1 = K_1 \cap K_2 \ ,
\end{align}
giving the required uniqueness.
\end{proof}

\begin{proof}[\textbf{Proof of Theorem \ref{thm:maingeneral}}]
Assume that condition \ref{item:main1} of Theorem \ref{thm:maingeneral} holds. 
By Theorem~\ref{thm:qktgi}, there is a unique $K_1$-transverse generalised metric $W_1$ on $E$ such that $Q(K_1)$ is a regular transverse generalised isometry between $W_1$ and $V_1^+$, using an arbitrary lift \smash{$\widetilde{W}_1^+$} of $W_1^+ = W_1/ K_1$ to $E$. From Proposition~\ref{prop:splittingexists} it follows that we can choose \smash{$\widetilde W_1^+$} such that \smash{$\widetilde W_1^+ \subseteq K_2^\perp \cap K_1^\perp$}. Set \smash{$\widetilde W_2^+ \coloneqq \widetilde W_1^+ \subseteq K_2^\perp$}. By the uniqueness of the splitting $s \colon K_1^\perp / K_1 \to K_1^\perp$ up to elements in $K_1\cap K_2$, this choice of \smash{$\widetilde W_2^+$} is unique up to elements of $K_1 \cap K_2$. 

Thus we may uniquely define the bundle
\begin{align}\label{eqn:transmetric2}
     W_2 = \widetilde W_2^+ \oplus K_2 \ .
\end{align}
The sum is direct: if $k_2\in \widetilde W_2^+\cap K_2$, then $k_2 \in \widetilde W_1^+$. Since $\widetilde W_1^+ \cap K_1 = \set{0}$, it follows that $k_2 \notin K_1$. But this contradicts condition \ref{item:main2} unless $k_2=0$.

It follows that $W_2$ is a pre-$K_2$-transverse generalised metric on $E$. This can be seen by noting firstly that $K_2 \subset W_2 \subset K_2^\perp$, since \smash{$\widetilde W_1^+ \subset K_2^\perp$}. Moreover, for every $w_2\in W_2$ such that $w_2\notin K_2$, 
it follows that \smash{$w_2\in \widetilde W_2^+$}.  
It then follows that $ \ip{w_2,w_2}_{E} > 0$. 

From Theorem~\ref{thm:qktgi} it follows that $ Q(K_1)^\top $ is a transverse generalised isometry between $V_1^+$ and~$W_2$. 

We shall now show that $W_2$ is a $K_2$-transverse generalised metric which descends to the quotient. Since $V_1^+$ is $D_1$-invariant, for every $\red w \in \sfGamma(V_1^+)$ and $\red X \in \mathsf{Iso}(V_1^+)$ it follows by definition that
\begin{align}\label{eqn:Vinvariantagain}
    \llbracket \red{\sigma}_1  (\red X) , \red w_1 \rrbracket_{\red E{}_1} - {\rho}{}_{\red E_1}^*\bigl({\red \sigma}^*_1(\llbracket \red{\sigma}_1  (\red X) , \red{\sigma}_1(\rho{}_{\red E_1}(\red w_1)) \rrbracket_{\red E{}_1}) \bigr) \ \in \ \mathsf{\Gamma}(V_1^+) \ .
\end{align}

Take a subset of sections $\set{\red w{}_j}$ spanning $V_1^+$ pointwise and $\set{\red X_i} \subset \mathsf{Iso}(V^+_1)$ spanning $D_1$ pointwise. The former lifts to a set of basic sections $\set{w_j}$ spanning \smash{$\widetilde W_1^+$} pointwise, the latter lifts to a set of projectable vector fields $\set{X_i} \subset \mathsf{Iso}(W_1)$ spanning $T\cF_2$ pointwise.
Since $\sigma_1$ and $\sigma_2$ are compatible, it follows that
\begin{align}\label{eqn:compatiblesplittingrearrange}
    \llbracket \sigma_2 (X_i) , w_j \rrbracket_{E} = \llbracket \sigma_1 (X_i) , w_j \rrbracket_{E} - \rho_{E}^* \bigl( \sigma^*_1 (\llbracket \sigma_1  (X_i) , \sigma_1(\rho_{E}(w_j)) \rrbracket_{E} )\bigr) \ .
\end{align}

If $X\in \sfGamma(TM_1)$ is projectable, then $\sigma_1(X)$ is a basic section, since $\sigma_1$ is an adapated splitting. 
Thus since the sections $\sigma(X_i)$ and $w_j$ are basic, the right-hand side of Equation~\eqref{eqn:compatiblesplittingrearrange} is basic\footnote{Recall that $\llbracket \sfGamma_{\mathrm{bas}}(K_1^\perp), \sfGamma_{\mathrm{bas}}(K_1^\perp) \rrbracket_{E_1} \subset \sfGamma_{\mathrm{bas}}(K_1^\perp)$.} and descends to the quotient; by Equation~\eqref{eqn:Vinvariantagain}, it moreover descends to a section of $V_1^+$. 
It follows that the right-hand side of Equation~\eqref{eqn:compatiblesplittingrearrange} is a section of $W_1$, hence \smash{$\llbracket \sigma_2 (X_i) , w_j \rrbracket_{E}\in \sfGamma(\widetilde W_1^+) \oplus \sfGamma(K_1)$}.

We next show that in fact $\llbracket \sigma_2 (X_i) , w_j \rrbracket_{E} \in \sfGamma(\widetilde W_1^+ \oplus K_1 \cap K_2)$. Suppose that $\llbracket \sigma_2 (X_i) , w_j \rrbracket_{E} \in \sfGamma(K_1)$. Recall that $X_i\in \mathsf{\Gamma}(\rho_{E}(K_2))$ 
and $\sigma_2$ is the adapted splitting to $K_2$. Therefore $\sigma_2(X_i) \in \mathsf{\Gamma}(K_2)$, and since $w_j\in \mathsf{\Gamma}(K_1^\perp\cap K_2^\perp)$ it follows from Equation \eqref{eqn:basicprop1} that 
\begin{align}
    \llbracket \sigma_{2}(X_i), w_j \rrbracket_{E} \ \in \ \mathsf{\Gamma}\big(K_2^\perp \cap K_1\big) = \mathsf{\Gamma}\big(K_1\cap K_2\big) \ ,
\end{align}
with the last equality following from condition \ref{item:main1}. We must therefore have
\begin{align}\label{eqn:tildeWK1capK2}
    \llbracket \sigma_2 (X_i) , w_j \rrbracket_{E} \ \in \ \sfGamma \big( \widetilde W_1^+ \oplus ( K_1\cap K_2)\big) \ .
\end{align}

The final piece of the puzzle is found by noting that, since $\rho_{E}$ is injective on $K_2$ as shown in Lemma \ref{lemma:injectiverho}, it follows that $\set{\sigma_2(X_i)}$ is a set of sections spanning $K_2$ pointwise. Let us  denote these sections by $k_i \coloneqq\sigma_2(X_i)$. Thus for any $k' \in \mathsf{\Gamma}(K_2)$ and \smash{$w\in \sfGamma(\widetilde W_1^+)$} there are expansions $k' = \sum_i\, f_i\, k_i$ and $w = \sum_j\, h_j\, w_j$ for some functions $f_i,h_j \in C^\infty (M)$ such that the sums are locally finite. 

Using the anchored Leibniz rule \eqref{eqn:anchorLeibniz} and Equation~\eqref{eqn:Leibnizfirst} we may then write
\begin{align}
    \llbracket k', w \rrbracket_{E} &= \sum_{i,j}\,\Big( f_i \, h_j\, \llbracket k_i, w_j \rrbracket_{E} - h_j\, \big( \rho_{E}(w_j)\cdot f_i\big)\, k_i \\
    & \hspace{4cm} + \big(\rho_{E}(f_i\, k_i)\cdot h_j\big)\,w_j + h_j\, \ip{k_i,w_j}_{E}\, \cD_{E} f_i \Big) \ .
\end{align}
The first term lives in $\sfGamma \big( \widetilde W_1^+  \oplus (K_1\cap K_2) \big)$ by Equation~\eqref{eqn:tildeWK1capK2}, the second term in $\mathsf{\Gamma}(K_2)$, the third term in \smash{$\mathsf{\Gamma}(\widetilde W_1^+)$}, and the final term is zero since \smash{$\widetilde W_1^+ \subset K_2^\perp$}.
Hence
\begin{align}
    \llbracket \mathsf{\Gamma}(K_2), \mathsf{\Gamma}(\widetilde W_1^+ )\rrbracket_{E} \ \subset \ \mathsf{\Gamma}(\widetilde W_1^+) + \mathsf{\Gamma} (K_2 ) \ .
\end{align}
Thus
\begin{align}\label{eqn:W2isinvariant}
\begin{split}
  \llbracket \mathsf{\Gamma}(K_2), \mathsf{\Gamma}(W_2) \rrbracket_{E} &= \llbracket \mathsf{\Gamma}(K_2), \mathsf{\Gamma}(\widetilde W_2^+) \rrbracket_{E} + \llbracket \mathsf{\Gamma}(K_2), \mathsf{\Gamma}(K_2) \rrbracket_{E} \\
  & \hspace{5cm} \subset \ \mathsf{\Gamma}(\widetilde W_2^+) + \mathsf{\Gamma}(K_2) = \mathsf{\Gamma}(W_2) \ ,
\end{split}
\end{align}
where we have also used the property that $K_2$ is involutive.

Equation \eqref{eqn:W2isinvariant} shows that $W_2$ is $K_2$-invariant, so by Proposition \ref{prop:reducedKsubbundle} it descends to a subbundle $V_2^+ \coloneqq \natural_{2}(W_2)$ of $\red E{}_2.$
We know from Theorem \ref{thm:qktgi} that $V_2^+$ defines a generalised metric on $\red E{}_2$ such that $Q(K_2)$ is a regular transverse generalised isometry between $W_2$ and $V_2^+$. Finally, since
\begin{align}
    Q(K_2)\diamond Q(K_1)^\top = \set{(\natural_{1}(e),e,e,\natural_{2}(e)) \, | \, e \in K_1^\perp \cap K_2^\perp } \ ,
\end{align}
and since $\widetilde W_1^\pm \subseteq K_1^\perp \cap K_2^\perp$, the rank condition \eqref{eqn:rankcondition} is satisfied. It follows that $Q(K_2) \circ Q(K_1)^\top  = R$ is a generalised isometry between $V_1^+$ and $V_2^+$.
Thus $(\red{E}_1 , V_1^+)$ and $(\red{E}_2, V_2^+)$ are geometrically T-dual.
Uniqueness of $V_2^+$ follows by Corollary~\ref{cor:GMisunique}.
\end{proof}

\begin{remark}[\textbf{Classification of T-duality}]
In Theorem \ref{thm:maingeneral} we are able to construct geometrically T-dual backgrounds because of the bisubmersion condition \eqref{eqn:bisubmersioncd}. This leads us to speculate that geometrically T-dual Courant algebroids can only be considered over Hausdorff Morita equivalent foliated manifolds, i.e. they might represent a subset of this equivalence class. This may lead to a classification of T-dual backgrounds.
\end{remark}

\begin{remark}
So far we have not discussed the dilaton field of the string background.
In order to include the dilaton in our picture, similarly to \cite{Garcia-Fernandez:2016ofz, Severa:2018pag}, we should discuss how divergence operators on Courant algebroids behave under transverse generalised isometries. We defer this to future work.
\end{remark}

\medskip

\subsubsection{Geometric T-duality for Standard Courant Algebroids}~\\[5pt]
\label{sec:geometricTdualitystandardCA}
We shall now discuss Theorem \ref{thm:maingeneral} when choosing an isomorphism of $E$ with $\IT M$. Recall that, in the assumptions of \ref{thm:maingeneral}, we have two compatible adapted splittings $\sigma_1$ and $\sigma_2$. Let $\Psi_1 = \sigma_1 + \tfrac{1}{2} \rho_E^*$ be the isomorphims between $\IT M$ and $E$ induced by $\sigma_1.$ Hence, $\Psi_1^{-1}(K_1) =T\cF_1$. Moreover, by applying $\Psi_1^{-1}$, Equation \eqref{eqn:differencesplit} can be written as $(\sigma_1^* \circ \sigma_2)(X)= -\iota_X B$,  for all $X \in TM$, thus, $\Psi_1^{-1}(K_2) = \e^{-B}(T\cF_2)$ and $\Psi_1^{-1}(K_2^\perp) = \e^{B}(TM) + \ann(T\cF_2)$. Therefore, the condition \ref{item:main1} yields that $T\cF_1 \cap T\cF_2 \subseteq \ker(B^\flat)$. Note that the same condition also follows from \ref{item:main2}.

\begin{proposition}\label{prop:Bisocondition}
Let $(\IT M, H)$ be the twisted standard Courant algebroid isomorphic to $E$ via $\Psi_1$, and suppose that the assumptions of Proposition \ref{prop:splittingexists} hold.
Then $B\in \mathsf{\Omega}^2(M)$ defined in Remark \ref{rmk:compatibilitysplit} takes the form
\begin{align}\label{eqn:Bdecomposition}
    B = B_{\rm ver} + B_{\rm hor} + B_{\rm mix} \ ,
\end{align}
where 
\begin{align}
    B_{\rm mix} \ \in \ \mathsf{\Gamma} \bigl( (T^*\cF_1 \cap \mathrm{Ann} (T\cF_2)) \wedge (\ann ( T\cF_1) \cap T^*\cF_2) \bigr)    
\end{align}
gives an isomorphism between $T\cF_1 \cap (\ann^* (T\cF_2)$ and $\ann ( T\cF_1) \cap T^*\cF_2,$ whereas 
\begin{align}
    B_{\rm ver}\in \mathsf{\Gamma}\left (\midwedge^2 T^*\cF_1 \right) \qquad \text{and} \qquad
    B_{\rm hor}\in \mathsf{\Gamma}\left (\midwedge^2\ann(T\cF_1) \right )
\end{align}
vanish on 
\begin{align}
\mathsf{\Gamma}\bigl((T^*\cF_1 \cap T^*\cF_2) \wedge (T^*\cF_1 \cap \ann (T\cF_2))\bigr)
\end{align}
and 
\begin{align}
\mathsf{\Gamma}\bigl((\ann (T\cF_1) \cap \ann(T\cF_2)) \wedge (\ann (T\cF_1) \cap T^* \cF_2) \bigr) \ , 
\end{align}
respectively.
\end{proposition}

\begin{proof}
We find the conditions $B$ must meet to satisfy the conditions \ref{item:main1} or \ref{item:main2} of Theorem~\ref{thm:maingeneral}.

Write $N^*\cF = \ann(T\cF)$. We can decompose 
\begin{align}
    B = B_{\rm ver} + B_{\rm hor} + B_{\rm mix}
\end{align}
where
\begin{align}
    B_{\rm ver}\in \mathsf{\Gamma}(\midwedge^2 T^*\cF_1) \ , \quad B_{\rm hor} \in \mathsf{\Gamma}( \midwedge^2 N^* \cF_1) \qquad \text{and} \qquad B_{\rm mix} \in \mathsf{\Gamma}( N^*\cF_1 \wedge T^*\cF_1) \ .
\end{align}

Recall that $\Psi_1^{-1}(K_2) = \e^{-B}(T\cF_2)$.
Thus
\begin{align*}
    \Psi_1^{-1}(K_2)\cap \Psi_1^{-1}(K_1^\perp) = \set{v_2 - \iota_{v_2} B \ | \ v_2 \in T\cF_2 \ , \ \iota_{v_2}B \in \mathrm{Ann}(T\cF_1) } \ .
\end{align*}
Similarly
\begin{align*}
    \Psi_1^{-1}(K_2^\perp) \cap \Psi_1^{-1}(K_1) = \set{ v_1 \in T\cF_1 \ | \ \iota_{v_1}B \in \mathrm{Ann}(T\cF_2) } \ .
\end{align*}

Conditions \ref{item:main1} and \ref{item:main2} of Theorem \ref{thm:maingeneral} state that these must be respectively contained in $K_1$ and $K_2$. The conditions given in the statement of the proposition ensure that this happens. For example, if $v_1 \in T\cF_1$ but $v_1 \notin T\cF_2$, then $v_1 \notin \e^B(T\cF_2)$. It follows that $v_1$ cannot be in $\e^B(K_2^\perp)$. The only way this can happen is if $\iota_{v_1}B_{\rm mix} \in T^*\cF_2$ and is non-zero. Since this happens for all $v_1 \in T\cF_1\cap N \cF_2$, $B_{\rm mix}$ gives an isomorphism between $T\cF_1\cap N \cF_2$ and $N^* \cF_1\cap T^*\cF_2$. The other conditions follow from similar considerations. 
\end{proof}

We shall now consider the setting of Example \ref{sec:topogicalTdualitystandardCA}, since it is our main source of explicit examples. 

\begin{theorem}\label{thm:mainsplitcase}
    Let $(\IT M_1, H_1)$ and $(\IT M_2 ,H_2)$ be twisted standard Courant algebroids as in Example \ref{sec:topogicalTdualitystandardCA} and Proposition \ref{prop:splitcaseconditions}, isomorphic via $\Phi = \overline{\phi} \circ \e^{B}\, $, and suppose that $B$ decomposes as in Equation \eqref{eqn:Bdecomposition}. Take a generalised metric $V_1^+$ on $\IT \cQ_1$ defined by $\red g{}_1 \in \mathsf{\Gamma}(\midodot ^2 T^*\cQ_1)$ and $\red b{}_1 \in \mathsf{\Omega}^2 (\cQ_1)$ and define the subalgebras $\mathsf{Iso}(W_1)$ and $\mathsf{Iso}(V_1^+)$ by
    \begin{align}
        \mathsf{Iso}(V_1^+) =& \set{\red X\in \sfGamma(D_1) \colon \pounds_{\red X}(\red g_1 + \red b_1) = 0}\, , \\
        \mathsf{Iso}(W_1) =& (\varpi_1)^{-1}_*(\mathsf{Iso}(V_1^+))\cap \sfGamma(T\cF_2) \, .
    \end{align}
    Suppose that $\mathsf{Iso}(V_1^+)$ spans pointwise $D_1$, and that for each $X \in \mathsf{Iso}(W_1)$
    \begin{align}\label{eqn:Bgbliederiv}
        \pounds_{X} B = 0 \, .  
    \end{align}
    Then there exists a unique generalised metric $V_2^+$ on $\IT\cQ_2$ such that $(\IT \cQ_1, H_1, V_1^+)$ and $(\IT \cQ_2, H_2 , V_2^+)$ are geometrically T-dual, i.e. $R$ is a generalised isometry between $V_1^+$ and $V_2^+$.
\end{theorem}

\begin{proof}
    By Remarks \ref{rmk:D1inv} and \ref{rmk:compatibilitysplit}, the conditions \eqref{eqn:Bgbliederiv} ensure that the generalised metric $(\red g{}_1, \red b{}_1)$ is $D_1$-invariant and the adapted splittings are compatible. The decomposition \eqref{eqn:Bdecomposition} ensures that conditions \ref{item:main1} and \ref{item:main2} of Theorem \ref{thm:maingeneral} are satisfied. Hence the conclusion follows by Theorem \ref{thm:maingeneral}.
\end{proof}
    
\begin{remark}
    The condition $\pounds_{X} B = 0 $ for $X\in \mathsf{Iso}(W_1)$ together with the assumption of reducibility of $E$ with respect to $K_2$ tells us that
    \begin{align}
        \iota_{X} H_1 = \iota_{X}\, \de B = - \de (\iota_{X} B) \ ,
    \end{align}
    similarly to Equation \eqref{eqn:constantSH}, i.e. it gives additional symmetries of the  sigma-model Wess-Zumino functional for $M,$ as discussed in Section \ref{subsect:sigma}.
\end{remark}

\begin{remark}
To see why we named $D_1$ the distribution of T-duality directions, consider a point $(q_1, q_2) \in C \subset \cQ_1 \times \cQ_2$, and take $\red v\in (D_1)_{q_1}$. Let \smash{$v\in \big(T \cF_2\big)_{m}$} be such that $\varpi_{1*}(v) = \red v$, where $\varpi_1(m) = q_1$. Denote $\varpi_2(m) = q_2$. It follows that $v\in (T_{m}\cF_2 \cap \ann^*(T_{m}\cF_1)).$ By Proposition \ref{prop:Bisocondition} we have
\begin{align}
    \iota_v B = \iota_v B_{\rm mix} \ \in \ T_{m}^*\cF_1 \cap  \ann(T_{m}\cF_2) \ .    
\end{align}
Thus the corresponding element in $R_{(q_1,q_2)}$ is
\begin{align}
    \big(\natural_{1}(v), \natural_{2}(v)\big) = \big(\red v, \natural_{2}(\iota_v B)\big) \ .
\end{align}
This is the usual exchange of tangent and cotangent (momentum and winding) directions seen in T-duality.
\end{remark}

\section{T-duality Relations and Doubled Geometry}\label{sec:applications}

In this Section we will show how the construction of Section \ref{sec:T-duality} describes the geometry of various kinds of T-duality such as T-duality with a correspondence space, where the Buscher rules for torus bundles are perfectly reproduced by the generalised isometry of exact Courant algebroids, and generalised T-duality of doubled sigma-models whose target space is endowed with an almost para-Hermitian structure. In particular, we discuss explicitly the Buscher rules for arbitrary rank torus bundles, as well as the examples of T-duality for lens spaces in three dimensions and generalised T-duality for the doubled Heisenberg nilmanifold in six dimensions. 

\medskip

\subsection{T-duality for Correspondence Spaces}\label{ssec:correspondencespace}~\\[5pt]
We briefly recall the definition of T-duality from \cite{cavalcanti2011generalized} for principal torus bundles, and show that this nicely fits into our definition.

\begin{definition} \label{def:CGtduality}
    Let $\cQ_1$ and $\cQ_2$ be principal $\sfT^k$-bundles over a common base manifold $\cB$, and let ${\red H}{}_1\in \mathsf{\Omega}_{\sfT^k}^3(\cQ_1)$ and ${\red H}{}_2\in \mathsf{\Omega}_{\sfT^k}^3(\cQ_2)$ be $\sfT^k$-invariant closed three-forms. Consider the fibred product $M\coloneqq\cQ_1\times_{\cB} \cQ_2$ giving the correspondence space diagram
    \begin{equation} \label{cd:correspondencespace}
        \begin{tikzcd}
             & M \arrow[swap]{dl}{\varpi_1} \arrow{dr}{\varpi_2} & \\
            \cQ_1 \arrow[swap]{dr}{\pi_1} &  & \cQ_2 \arrow{dl}{\pi_2} \\
             & \cB & 
        \end{tikzcd}
    \end{equation}
    with $M$ endowed with the closed three-form $\varpi_1^*\, {\red H}{}_1 - \varpi_2^*\, {\red H}{}_2.$

    Then $\cQ_1$ and $\cQ_2$ are \emph{T-dual} if there is a $\sfT^{2k}$-invariant two-form $B\in \mathsf{\Omega}^2_{\sfT^{2k}}(M)$ such that
    \begin{align}\label{eqn:hfluxdifference}
        \varpi_1^*\, {\red H}{}_1 - \varpi_2^*\, {\red H}{}_2 = \de B \ ,
    \end{align}
    and the smooth skew-symmetric map 
    \begin{align}\label{eqn:Bnondegenerate}
        B \colon \frk_1 \otimes \frk_2 \longrightarrow \IR
    \end{align}
    is non-degenerate, where $\frk_i = \ker (\varpi_{i *})$ for $i=1,2$. 
\end{definition}

\begin{remark}[\textbf{Correspondence Spaces as Bisubmersions}]\label{rmk:correspondencebi}
The commutative diagram \eqref{cd:correspondencespace} represents a particular case of a bisubmersion, as discussed in Remark~\ref{rmk:bisubmersions}. The condition \eqref{eqn:bisubmersioncd} is clearly satisfied, since the generators of the $\frt^k_{1}$-action commute with the generators of the $\frt^k_{2}$-action, where ${\sf Lie}(\sfT^k)= \frt^k$ and the subscript $i$ indicates the manifold $\cQ_i$ on which $\sfT^k$ acts. Here the $\sfT^k$-action whose orbits are the integral manifolds of $\ker(\varpi_{2 *})$ induces a $\sfT^k$-action on $\cQ_1$ which, in order to provide Definition \ref{def:CGtduality}, is supposed to be free and proper. The distribution induced by this $\sfT^k$-action on $\cQ_1$ defines the distribution of T-duality directions $D_1=(\varpi_1)_*\frk_2$, whose integral manifolds are the fibres of $\pi_1 \colon \cQ_1 \to \cB$. 
Similarly, $\cQ_2$ inherits a $\sfT^k$-action, and hence a distribution $D_2$, whose foliation, given by the fibres of $\pi_2 \colon \cQ_2 \to \cB,$ is Hausdorff Morita equivalent to the foliation induced by $D_1$.
\end{remark}

Let us reinterpret this picture in the setting of Section \ref{sec:T-duality}. We start with

\begin{lemma}
  Let  $\cQ_1$ and $\cQ_2$ be T-dual in the sense of Definition \ref{def:CGtduality}. Then there is a T-duality relation in the sense of Definition \ref{def:topologicalTdual}.
\end{lemma}

\begin{proof}
    It is enough to find a Courant algebroid $E$ over $M$ such that Definition \ref{def:topologicalTdual} is satisfied. 
    Take with $E = (\IT M , \, H= \varpi_1^*\,{\red H}{}_1)$ over $M = \cQ_1 \times_\cB \cQ_2$ along with $K_1 = \frk_1$ and $K_2=\e^{-B}(\frk_2)$. 
    Since $\cQ_1$ and $\cQ_2$ are T-dual in the sense of Definition \ref{def:CGtduality}, there exists $B\in \mathsf{\Omega}^2_{\sfT^{2k}}(M)$ satisfying Equation~\eqref{eqn:hfluxdifference} which is non-degenerate on $\frk_1\otimes\frk_2$.  
    Moreover, $ \e^{-B}\,(\frk_2)\cap \frk_1 = \set{0}$ has constant rank. Hence we may form the Courant algebroid relation $R$ fitting the diagram
    \begin{equation} \label{cd:correspondencetdualitysquare}
    \begin{tikzcd}
    & \IT M \arrow[tail]{dr}{Q(K_2)} &  \\
    \IT \cQ_1 \arrow[dashed]{rr}{R}\arrow[dashed]{ur}{Q(K_1)^\top } & & \IT \cQ_2 
\end{tikzcd}  
    \end{equation}
    which is supported on $M=\cQ_1 \times_{\cB} \cQ_2$ by Lemma \ref{lemma:cdsmoothcirclecomp}.
\end{proof}

Note that the converse is not true, since our notion of T-duality related Courant algebroids imposes no condition on $B$ other than the constant rank condition from Proposition \ref{prop:splitcaseconditions}, which is trivially satisfied for any two-form $B\in \mathsf{\Omega}^2(M)$ since $\frk_1 \cap \frk_2 = \set{0}$.

We will see that, upon introducing geometric data, the possible choices for $B$ coincide with the $B$-field choice in Definition \ref{def:CGtduality}. This is exemplified through

\begin{proposition}\label{prop:equivalencewithcorresp}
    If $\cQ_1$ and $\cQ_2$ are T-dual in the sense of Definition \ref{def:CGtduality}, and there is a $\sfT^k_1$-invariant generalised metric $V_1^+$ on $(\IT \cQ_1, \red H{}_1)$, then there is a unique generalised metric $V_2^+$ such that $(\IT \cQ_1, {\red H}{}_1, V_1^+)$ and $(\IT \cQ_2, {\red H}{}_2, V_2^+)$ are geometrically T-dual.
\end{proposition}

\begin{proof}
    If $\cQ_1$ and $\cQ_2$ are T-dual in the sense of Definition \ref{def:CGtduality}, then $B$ decomposes as in Equation~\eqref{eqn:Bdecomposition}. Furthermore, since $V_1^+$ is $\sfT^k$-invariant, we define\footnote{If a Lie group $\sfG$ acts on a manifold $M$, we denote by $\#_M \colon \frg \to \sfGamma(TM)$ the corresponding action of its Lie algebra $\frg$ on $M$.} $\mathsf{Iso}(V_1^+) \coloneqq \#_{\cQ_1} (\frt^k_2) \subset \sfGamma(D_1)$, which spans $D_1$ pointwise. Thus $\mathsf{Iso}(W_1) = \#_{M} (\frt^k_2) \subset \sfGamma(\frk_1)$, and since $B$ is $\sfT^{2k}$-invariant it follows that $\pounds_{X} B = 0,$ for any $X \in \mathsf{Iso}(W_1)$. The conclusion then follows by Theorem \ref{thm:mainsplitcase}.
\end{proof}

In this picture, we can also provide a further characterisation of the components of the three-form $\red H{}_1 \in \mathsf{\Omega}^3_{\rm cl}(\cQ_1)$ through

\begin{lemma} \label{lemma:H1componentscorresp}
Let $(\IT \cQ_1, \red H{}_1)$ and $(\IT \cQ_2, \red H{}_2)$ be twisted standard Courant algebroids with $\red H{}_1\in \mathsf{\Omega}^3_{\sfT^k}(\cQ_1)$ and $\red H{}_2\in \mathsf{\Omega}^3_{\sfT^k}(\cQ_2)$, and let them be endowed with generalised metrics $V_1^+$ and $V_2^+$ such that they are geometrically T-dual. Then $\varpi_1^*\, \red H{}_1$ can be written as the sum of a doubly pulled back three-form $H_\cB \in\mathsf{\Omega}^3(\cB)$ and a component determined by
\begin{align} \label{eqn:iXV2}
 \iota_{X_{\texttt{v}_2}} \varpi_1^*\,\red H{}_1 = \iota_{X_{\texttt{v}_2}}\, \de B \ ,   
\end{align}
for any $X_{\texttt{v}_2} \in \mathsf{\Gamma}(\frt^k_2)$. 
\end{lemma}

\begin{proof}
Definition \ref{def:CGtduality} implies that \smash{$\iota_{Y_{\texttt{v}_1}}\, \iota_{X_{\texttt{v}_1}} \varpi_1^*\, \red H{}_1 = 0,$} 
for all $X_{\texttt{v}_1}, Y_{\texttt{v}_1} \in \mathsf{\Gamma}(\frk_1),$ since $B$ is $\sfT^{2k}$-invariant and \smash{$\iota_{X_{\texttt{v}_1}}\, \varpi^*_2\, \red H{}_2 = 0,$} for all $X_{\texttt{v}_1} \in \mathsf{\Gamma}(\frk_1).$ Hence only two components of $\varpi_1^*\,\red H{}_1$ are non-vanishing and, because of $\sfT^k$-invariance of $\red H{}_1,$ one of them must be the double pullback of a three-form on $\cB.$   The other component is characterised as follows. Since any $\red X \in \mathsf{\Gamma}(D_1)$ lifts to a section 
$X_{\texttt{v}_2} \in \mathsf{\Gamma}(\frk_2)$, and $\varpi^*_2\, \red H{}_2$ is basic with respect to the orbits of $\sfT^k$ whose tangent distribution is $\frk_2,$ that is it comes from the base $\cQ_2$, Equation~\eqref{eqn:iXV2} follows.
\end{proof}

\begin{example}[\textbf{$\boldsymbol B$-fields from Connections}]
  Let $\cQ_1$ and $\cQ_2$ be principal $\sfT^k$-bundles over a common base manifold $\cB$, and let them be endowed with closed three-forms $\red H{}_1 \in \mathsf{\Omega}^3_{\sfT^k}(\cQ_1)$ and $\red H{}_2 \in \mathsf{\Omega}^3_{\sfT^k}(\cQ_2).$ Choose connections $\theta_1 \in \mathsf{\Omega}^1(\cQ_1, \frt^k_1) $ and $\theta_2 \in \mathsf{\Omega}^1(\cQ_2,\frt^k_2),$ respectively, and suppose that $$B = -\varpi^*_1 \theta_1 \wedge \varpi^*_2\theta_2 \ .$$ When Lemma \ref{lemma:H1componentscorresp} holds, we can write
  \begin{align} \label{eqn:H1integralcohom}
   \varpi^*_1\, \red H{}_1  = -\varpi^*_1 \theta_1\wedge\varpi^*_2\big(c_1(\cQ_2)\big) + \varpi^*_1 (\pi^*_1 H_\cB)  \ ,
  \end{align}
  where $c_1(\cQ_2) = \de \theta_2$ represents the Chern class of $\cQ_2.$ 
  This can be shown by letting $X_{\texttt{v}_2} \in \mathsf{\Gamma}(\frk_2)$ be a generator of the $\frt^k_2$-action. Then 
  \begin{align}
  \iota_{X_{\texttt{v}_2}} \varpi_1^*\,\red H{}_1 = \de( \varpi^*_2\theta_2) = \varpi^*_2 \big(c_1(\cQ_2)\big) \ ,  
  \end{align}
  and Equation \eqref{eqn:H1integralcohom} follows.
  A symmetric argument holds for $\varpi_2^*\,\red H{}_2.$ This is the key case considered in \cite{bouwknegt2004tduality,cavalcanti2011generalized} which provides criteria for the existence of T-dual pairs based on the components of the $H$-flux. It shows how T-duality in this case is characterised by an interchange of the Chern classes of the torus bundle with topological data associated to the $H$-flux.
\end{example}   

The analogy with the correspondence space picture extends further when we consider the construction in \cite[Section~3]{cavalcanti2011generalized}, which uses the Fourier-Mukai transform, of the isomorphism\footnote{In \cite{cavalcanti2011generalized} the isomorphism is denoted by $\phi$, but here we use the symbol $\mathscr{R}$ to distinguish it from the already used $\phi$ in the present paper.} of $\sfT^k$-invariant sections \smash{$\mathscr{R} \colon \mathsf{\Gamma}_{\sfT^k}(\IT \cQ_1) \to \mathsf{\Gamma}_{\sfT^k}(\IT \cQ_2) $}. 
In that construction a $\sfT^k$-invariant section $X_1+\xi_1\in \mathsf{\Gamma}_{\sfT^k}(\IT \cQ_1)$ is lifted to a $\sfT^{2k}$-invariant section \smash{${\hat X}_1 + \varpi_1^* \xi_1\in \mathsf{\Gamma}_{\sfT^{2k}}(\IT M)$} whose image under $B$-field transformation is basic:
\begin{align} \label{eqn:hatXlift}
    (\varpi_1^*\xi_1)(Y_2) + B({\hat X}_1, Y_2) = 0  \ , 
\end{align}
for all $Y_2 \in \mathsf{\Gamma}(\frt^k_2)$. The non-degeneracy of $B$ on $\frk_1\otimes\frk_2$ ensures that the lift ${\hat X}_1$ satisfying Equation~\eqref{eqn:hatXlift} is unique.
The isomorphism $\mathscr{R}$ is then constructed as the pushforward by $\varpi_2$ of ${\hat X}_1 + \varpi_1^* \xi_1 + \iota_{\hat X_1} B$:
\begin{align} \label{eqn:fouriermukai}
    \mathscr{R}(X_1 + \xi_1) \coloneqq \varpi_{2*}{\hat X}_1 + \varpi_1^* \xi_1 + \iota_{\hat X_1} B \ .
\end{align}

In our language, since the Courant algebroid relation $R$ is a generalised isometry, it follows by Lemma \ref{lemma:conditionsonPhi} that conditions \ref{item:main1} and \ref{item:main2} of Theorem~\ref{thm:maingeneral} are satisfied. 
Hence by Proposition \ref{prop:splittingexists} there is a splitting $s \colon (\frk_1)^\perp / \frk_1 \to (\frk_1)^\perp$ such that $\mathrm{im}(s) \subset \big(\e^{-B}(\frk_2)\big)^\perp$. 
This splitting is unique up to elements in \smash{$\e^{-B}\,(\frk_2)\cap \frk_1 = \set{0}$}. 
It follows that, at a point $(m,q_1)\in \gr(\varpi_1) \subset M \times \cQ_1$, an element $v_1+ \nu_1 \in \IT_{q_1} \cQ_1 $ lifts uniquely to an element $\hat{v}_1 + \varpi_1^*\nu_1 \in \IT_{m} M $ given by 
\begin{align}
    \hat{v}_1 + \varpi_1^*\nu_1 \coloneqq s_m\big(\mathscr{J}_{q_1,m}(v_1+ \nu_1 )\big) \ ,
\end{align}
such that
\begin{align}
    (\varpi_1^*\nu_1) ( v_2) + B(\hat{v}_1, v_2) = 0 \ , 
\end{align}
for all $ v_2 \in (\frk_2)_m.$

When we restrict this construction to $\sfT^k$-invariant sections, the analogy becomes exact through

\begin{proposition}\label{prop:Fourier-Mukai}
    The T-duality relation $R$ gives rise to an isomorphism $\red\Phi\colon\mathsf{\Gamma}_{\sfT^k}(\IT \cQ_1)\to\mathsf{\Gamma}_{\sfT^k}(\IT \cQ_2)$ of $C^\infty(\cB)$-modules which coincides with the isomorphism $\mathscr{R}$ defined by Equation \eqref{eqn:fouriermukai}.
\end{proposition}

\begin{proof}
    Recall that in this case, a basic section $\hat\psi_1 = X_1 + \xi_1\in \mathsf{\Gamma}_{\mathrm{bas}}\big((\frk_1)^\perp \big)$ satisfies
    \begin{align}
        \llbracket Y_1, X_1 + \xi_1 \rrbracket_{H} = [Y_1, X_1] + \pounds_{Y_1} \xi_1 \ ,
    \end{align}    
    for every $Y_1\in \mathsf{\Gamma}(\frk_1)$, where $X_1 \in \mathsf{\Gamma}_{\mathsf{T}^k}(TM),$ i.e. $X_1$ is projectable with respect to $\frk_1,$ and $\xi_1 \in \mathsf{\Gamma}({\rm Ann}(\frk_1)).$
    Suppose that \smash{$\psi_1 \in \mathsf{\Gamma}_{\sfT^k}(\IT \cQ_1) $}.
    We know that $\mathscr{J}$ extends to an isomorphism of $C^\infty(\cQ_1)$-modules 
 $\mathscr{J} \colon \mathsf{\Gamma}(\IT \cQ_1) \to \mathsf{\Gamma}_{\mathrm{bas}}((\frk_1)^\perp ) / \mathsf{\Gamma}(\frk_1)$.
    By Proposition \ref{prop:splittingexists}, the unique splitting $s \colon (\frk_1)^\perp/ \frk_1 \to (\frk_1)^\perp$ extends to a map between sections as 
    \begin{align}
        s \colon \mathsf{\Gamma}\bigl((\frk_1)^\perp \bigr) \,\big/\, \mathsf{\Gamma}(\frk_1) \longrightarrow \mathsf{\Gamma}\bigl((\frk_1)^\perp \bigr) \ ,
    \end{align}
    since it covers the identity,
    and so we may uniquely define
    \begin{align}
        \hat \psi_1 \coloneqq s\big( \mathscr{J}(\psi_1)\big) \ \in \ \mathsf{\Gamma}_{\sfT^{2k}}\bigl( \bigl(\e^{-B}(\frk_2) \bigr)^\perp  \bigr) \ .
    \end{align}
   and it follows that $\hat \psi_1 \in \mathsf{\Gamma}_{\mathrm{bas}}\bigl( \left(\e^{-B}(\frk_2)\right)^\perp \bigr)$. Define  $\psi_2 \in \mathsf{\Gamma}_{\sfT^k}(\IT \cQ_2)$ by
    \begin{align}
        \psi_2 \coloneqq \natural_{2}\bigl(\hat \psi_1 \bigr) \ .
    \end{align}
   It follows that $\psi_1 \sim_{R} \psi_2$. By construction, for every $\psi_1\in \mathsf{\Gamma}_{\sfT^k}(\IT \cQ_1) $ there is a unique section $\psi_2\in \mathsf{\Gamma}_{\sfT^k}(\IT \cQ_2) $ such that $\psi_1 \sim_{R} \psi_2$. Similarly, for every $\psi_2 \in \mathsf{\Gamma}_{\sfT^k}(\IT \cQ_2) $ there is a unique section $\psi_1 \in \mathsf{\Gamma}_{\sfT^k}(\IT \cQ_1) $ such that $\psi_2 \sim_{R^\top } \psi_1$, where \smash{$R^\top $} is the transpose Courant algebroid relation. 
   
   Thus the map 
    \begin{align}
        \red \Phi \colon \mathsf{\Gamma}_{\sfT^k}(\IT \cQ_1) \longrightarrow \mathsf{\Gamma}_{\sfT^k}(\IT \cQ_2) 
        \end{align}
        defined by
        \begin{align}
        \red\Phi(\psi_1) = \psi_2 \qquad \text{with} \quad \psi_1 \sim_{R} \psi_2
    \end{align}
    is well-defined and injective, hence bijective. It follows that $\red \Phi$ is a $C^\infty(\cB)$-module isomorphism, where the $C^\infty(\cB)$-module structures on $\mathsf{\Gamma}_{\sfT^k}(\IT \cQ_1)$ and $\mathsf{\Gamma}_{\sfT^k}(\IT \cQ_2)$ are given by pulling back smooth functions on $\cB$ by $\pi_1$ and $\pi_2$, respectively. It is clear by construction that $\psi_1 \sim_{R} \psi_2$ if and only if $\mathscr{R}(\psi_1) = \psi_2$, hence $\red \Phi = \mathscr{R}$.
\end{proof}

\begin{remark} 
    In the construction of $\red \Phi$ in Proposition \ref{prop:Fourier-Mukai} there is no mention of a Fourier-Mukai integral transform as in the analogous construction of $\mathscr{R}$ in \cite{cavalcanti2011generalized}. Thus it may be possible to extend a correspondence space type picture for T-duality to cases of non-compact manifolds, where the Fourier-Mukai transform is not defined, such as Drinfel'd doubles. We defer this to future work.
\end{remark}

\begin{remark}[\textbf{Buscher Rules}]\label{rmk:buscherrules}
    We pick connections $\theta_1 \in \mathsf{\Omega}^1(\cQ_1, \frt^k_1) $ and $\theta_2 \in \mathsf{\Omega}^1(\cQ_2,\frt^k_2)$ for the principal $\sfT^k$-bundles $\cQ_1$ and $\cQ_2.$ 
    To fulfil the conditions of Theorem~\ref{thm:mainsplitcase}, we use Remark \ref{rmk:correspondencebi} to consider a $\sfT^k$-invariant generalised metric $(\red g{}_1, \red b{}_1)$ on $\IT \cQ_1$, and set $\mathsf{Iso}(V_1^+) = \#_1 (\frt_2^k)$. 
    
    The connection $\theta_i$ gives a splitting
    \begin{align}
        T\cQ_i \simeq \mathsf{Hor}(\cQ_i) \oplus \mathsf{Ver} (\cQ_i)
        \end{align}
        into horizontal and vertical subbundles, for $i=1,2$, whose sections we denote by $\red X^{\rm h} + \red X^{{\rm v}_i}$. This also gives a decomposition of
    \begin{align}
        \midotimes^2 T^*\cQ_i = \bigl( \midotimes^2 \mathsf{Hor}^* (\cQ_i) \bigr) \ \oplus \ \bigl(\midotimes^2 \mathsf{Ver}^* (\cQ_i) \bigr) \ \oplus \ \bigl(\mathsf{Hor}^* (\cQ_i) \otimes \mathsf{Ver}^* (\cQ_i)\bigr) \ .
    \end{align}
    We denote the corresponding components of sections as $ \red\alpha^{\rm h} + \red\alpha^{{\rm v}_i}+ \red\alpha^{{\rm m}_i}$.
    Thus, denoting $$\red h{}_1 = \red g{}_1 + \red b{}_1 \ , $$ a section $\red w_1\in \mathsf{\Gamma}_{\sfT^k}(V_1^+)$ can be written as
    \begin{align}
        \red w{}_1 = \red X^{{\rm v}_1} + \red X^{\rm h} + \red A_1  \big(\red X^{{\rm v}_1} + \red X^{\rm h}\big) \ , 
        \end{align} 
        where
        \begin{align}
        \red A_1 = \begin{pmatrix}
            \red h{}_1^{{\rm v}_1} & \red h{}_1^{{\rm m}_1} \\[3pt]
            \red{\tilde h}{}_1^{{\rm m}_1} & \red h{}_1^{\rm h}
        \end{pmatrix} \qquad \text{with} \quad \red{\tilde h}{}_1^{{\rm m}_1} = (\red{ h}_1^{{\rm m}_1})^{\rm t} \ .
    \end{align}

    The pullbacks of the connections $\theta_1$ and $\theta_2$ give a splitting 
    \begin{align} 
    TM \simeq \mathsf{Hor} (M) \oplus \mathsf{Ver}_1 (M) \oplus \mathsf{Ver}_2 (M) \ , 
    \end{align}
    and we identify $\mathsf{\Gamma}_{\sfT^k}(\mathsf{Ver}_i(M)) \simeq \mathsf{\Gamma}(\mathsf{Ver}(\cQ_i))$ since they are isomorphic as $C^\infty(\cQ_i)$-modules.\footnote{This follows from the fact that there exists a fibrewise bijective map covering $\varpi_i$ between them, see \cite{Mackenzie}.} Thus $\red w_1 \in \mathsf{\Gamma}_{\sfT^k}(V_1^+)$ lifts to $w_1\in\sfGamma_{\sfT^k}(W_1)$ given by
    \begin{align}
        w_1 = X^{{\rm v}_1} + X^{\rm h} + X^{{\rm v}_2} + A_1 \big(X^{{\rm v}_1} + X^{\rm h} \big) \ ,
    \end{align}
    where by using the non-degenerate map $B \colon \frk_1 \to (\frk_2)^*$ we write
    \begin{align}\label{eqn:buscheruniquelift}
        X^{{\rm v}_2} = -(B^{\rm t})^{-1}\big( \bigl( A_1 (X^{{\rm v}_1}+X_1^{\rm h}) \bigr)^{{\rm v}_1}\big) = -(B^{\rm t})^{-1} \big( h_1^{{\rm v}_1}(X^{{\rm v}_1}) + h_1^{{\rm m}_1}(X^{\rm h})\big) \ ,
    \end{align}
    and $h_1$ is the pullback of $\red h{}_1$ determining the $\frt_{\cQ_1}^k$-transverse generalised metric $W_1$.

    We can now compute
    \begin{align}
        w_2 :=\e^{B}\,(w_1) &= X^{{\rm v}_1} + X^{\rm h} + X^{{\rm v}_2} + \iota_{X^{{\rm v}_1} + X^{\rm h} + X^{{\rm v}_2}}B + A_1 \big(X^{{\rm v}_1}+X^{\rm h}\big)\\[4pt]
        &= X^{{\rm v}_1} + X^{\rm h} + X^{{\rm v}_2} + B(X^{{\rm v}_1})^{{\rm v}_2} + B(X^{{\rm v}_2})^{{\rm v}_1} + A_1 \big(X^{{\rm v}_1}+X^{\rm h}\big)\\[4pt]
        &= X^{{\rm v}_1} + X^{\rm h} + X^{{\rm v}_2} + B(X^{{\rm v}_1})^{{\rm v}_2}  + \bigr( A_1(X^{{\rm v}_1}+X^{\rm h})\bigl)^{\rm h} \ ,
    \end{align}
    using $B(X^{\rm h})=0$, $B(X^{{\rm v}_i})^{{\rm v}_i} = 0$, and Equation \eqref{eqn:buscheruniquelift} for the last equality.
    It follows that $w_2\in \mathsf{\Gamma}_{\sfT^k}(W_2)$ for some $\frt^k_2$-transverse generalised metric $W_2$, with 
    \begin{align}
        A_2 = \begin{pmatrix}
            h_2^{\rm v_2} & h_2^{\rm m_2} \\[3pt]
            \tilde h_2^{\rm m_2} & h_2^{\rm h}
        \end{pmatrix} \ .
    \end{align}
    This means that
    \begin{align}
        B(X^{\rm v_1})^{\rm v_2} + \bigl(A_1\, (X^{\rm v_1}+X^{\rm h})\bigr)^{\rm h} = A_2(X^{\rm v_2}+X^{\rm h}) \ ,
    \end{align}
    or equivalently
    \begin{align}
        B(X^{\rm v_1}) &= h_2^{\rm v_2}(X^{\rm v_2}) + h_2^{\rm m_2}(X^{\rm h}) \ ,
    \end{align}
        and
    \begin{align}   
        \tilde h_1^{\rm m_1}(X^{\rm v_1}) + h_1^{\rm h}(X^{\rm h}) &= \tilde h_2^{\rm m_2}(X^{\rm v_2}) + h_2^{\rm h}(X^{\rm h}) \ .
    \end{align}
    
    Substituting Equation \eqref{eqn:buscheruniquelift}, we get
    \begin{align}
        B(X^{\rm v_1}) & = -h_2^{\rm v_2}\big((B^{\rm t})^{-1}(h_1^{\rm v_1}(X^{\rm v_1}))\big) \ , \\[4pt]
        0 & = -h_2^{\rm v_2}\big((B^{\rm t})^{-1}(h_1^{\rm m_1}(X^{\rm h}))\big)+ h_2^{\rm m_2}(X^{\rm h}) \ ,\\[4pt]
        \tilde h_1^{\rm m_1}(X^{\rm v_1}) & = -\tilde h_2^{\rm m_2}\big((B^{\rm t})^{-1}(h_1^{\rm v_1}(X^{\rm v_1}))\big) \ , \\[4pt]
        h_1^{\rm h}(X^{\rm h}) & = -\tilde h_2^{\rm m_2}\big((B^{\rm t})^{-1}(h_1^{\rm m_1}(X^{\rm h}))\big) + h_2^{\rm h}(X^{\rm h}) \ ,
    \end{align}
    yielding
    \begin{align}
        h_2^{\rm v_2} & = -B\, (h_1^{\rm v_1})^{-1}\, B^{\rm t} \ ,\\[4pt]
        h_2^{\rm m_2} & = B \, (h_1^{\rm v_1})^{-1}\, h_1^{\rm m_1} \ , \\[4pt]
        h_2^{\rm h} & = h_1^{\rm h} - \tilde h_1^{\rm m_1}\,(h_1^{\rm v_1})^{-1}\, h_1^{\rm m_1} \ .
    \end{align}
    By reducing to $\red h{}_2 = \red g{}_2+\red b{}_2$, one can now unravel these formulas to read off the structure tensors \smash{$\red g{}_2\in \sfGamma(\midodot^2 T^*\cQ_2)$} and \smash{$\red b{}_2 \in \sfGamma(\midwedge^2 T^*\cQ_2)$} of the corresponding $\sfT^k$-invariant generalised metric on $\IT\cQ_2$. This results in global Buscher rules for the transformation of the metric and Kalb-Ramond field under T-duality. 
\end{remark}

The case of an $\sfS^1$-fibration discussed in \cite[Section~4]{cavalcanti2011generalized} is a simplified case of the construction of Remark~\ref{rmk:buscherrules}, which here comes from the defining conditions for the generalised isometry~$R.$

\begin{example}[\textbf{Lens Spaces}]
\label{ex:lensspaces}
    We apply the construction above to an explicit example.
    Consider a three-dimensional $\sfS^1$-bundle $\pi_1  \colon \cQ_1 \to \sfS^2$. Choose a chart with coordinates $(x,y,z)$ such that $z$ is adapted to the circle fibres and $(x,y)$ are pullback coordinates from the two-sphere $\sfS^2.$ Pick a connection 
    \begin{align}
    \theta_1 = \de z + m\,  x \, \de y
    \end{align}
    on $\cQ_1$ and an integral volume form 
    \begin{align}
    {\red H}{}_1 = k \, \de x \wedge \de y \wedge \de z \ ,
    \end{align}
    with $m,k\in \IZ$, such that
    \begin{align}
        \int_{\cQ_1}\, {\red H}{}_1 = k \qquad \text{and}  \qquad \int_{\sfS^2}\, c_1(\cQ_1) = m\ ,
    \end{align}
    where $c_1(\cQ_1) = \de \theta_1 = m\,\de x\wedge \de y$ represents the Chern class of $\cQ_1$.
    
    On $\cQ_1$ we take as metric
    \begin{align}
        {\red g}{}_1 = \pi_1^*\,g_{\sfS^2} + \theta_1\otimes\theta_1 \ ,
    \end{align}
     where $g_{\sfS^2}$ is the standard round metric on $\sfS^2$. Then $\frac\partial{\partial z},$ a local expression of the generator of the $\mathfrak{u}(1)$-action, is a Killing vector field for ${\red g}{}_1$.

    Consider another three-dimensional $\sfS^1$-bundle $\pi_2 \colon \cQ_2 \to \sfS^2$ endowed with connection 
    \begin{align}
    \theta_2 = \de \tilde z + n\, x\,  \de y
    \end{align}
    and Chern number $n\in\IZ$, where ${\tilde z}$ is the coordinate adapted to the fibres. 
    Let $M= \cQ_1 \times_{\sfS^2} \cQ_2$ with projections $\varpi_i \colon M \to \cQ_i,$ for $i=1,2,$ and coordinates\footnote{We omit the pullback notation for the coordinates.} $(x,y,z,\tilde z)$ pulled back from $\cQ_1$ and $\cQ_2.$
    In order to satisfy the conditions of Theorem \ref{thm:mainsplitcase}, we consider the two-form 
    \begin{align}
    B = -\varpi_1^*\theta_1 \wedge \varpi_2^*\theta_2 = x\,\de y\wedge(n\,\de z - m\,\de\tilde z) -\de z\wedge\de\tilde z
    \end{align}
    for our $B$-field transformation.
    
    For $(\IT M, H)$ to be reducible with respect to $K_2 = \e^{-B}\frk_2$, we have that 
    \begin{align}
        \varpi_1^*\,{\red H}{}_1 - \de B &= k\, \de x \wedge \de y \wedge \de z + m\, \de x \wedge \de y \wedge \de \tilde z - n\, \de z \wedge \de x \wedge \de y\\[4pt]
        & = (k-n)\, \de x \wedge \de y \wedge \de z + m\, \de x \wedge \de y \wedge \de \tilde z 
    \end{align}
    has to be basic with respect to the fibres of $\varpi_2$.  
    This happens if and only if $n=k$. Thus $\cQ_2$ has three-form ${\red H}{}_2$ and Chern form $c_1(\cQ_2)$ satisfying
    \begin{align}
        \int_{\cQ_2}\, {\red H}{}_2 = m \qquad \text{and} \qquad \int_{\sfS^2}\, c_1(\cQ_2) = k \ ,
    \end{align}
    as well as metric obtained from Remark \ref{rmk:buscherrules} as
    \begin{align}
        {\red g}{}_2 = \pi_2^*\,g_{\sfS^2} + \theta_2\otimes\theta_2 \ ,
    \end{align}
    which is again $\sfS^1$-invariant. 
    
    Hence the lens space $\sfL(m,k)$ is T-dual to the lens space $\sfL(k,m)$, realising a topology change between T-dual spaces, see e.g. \cite{Bouwknegt2003topology}. In particular, the Hopf fibration $\sfL(1,0)\simeq\sfS^3$ without $H$-flux is T-dual to the trivial circle bundle $\sfL(0,1)\simeq\sfS^2\times\sfS^1$ with $H$-flux. Moreover, the Hopf fibration with $H$-flux $\sfL(1,1)$ is self-T-dual.
\end{example}

\medskip

\subsection{Generalised T-duality for Para-Hermitian Manifolds}\label{ssec:parahermitian}~\\[5pt]
We shall now study the conditions under which the relation $R$ becomes a generalised isometry when the manifolds $M_1$ and $M_2$ admit almost para-Hermitian structures, such that the diffeomorphism $\phi$ preserves their split-signature metrics.

\medskip

\subsubsection{Para-Hermitian Manifolds}
\label{sssec:parahermmfld}~\\[5pt]
Let us start by reviewing the main properties of para-Hermitian manifolds, see \cite{Cortes2004, Freidel2017, Freidel2019, Marotta2021born} for more details.

\begin{definition}
An \emph{almost para-complex structure} on an even-dimensional manifold $M$ is an automorphism of the tangent bundle $TM$ such that $\ccK^2=\unit_{TM}$ and whose $\pm1$-eigenbundles $L_\pm = \ker(\unit_{TM}\mp\ccK)$ have the same rank.

An \emph{almost para-Hermitian structure} on $M$ is a pair $(\eta, {\ccK})$ of a split-signature metric $\eta$ and an almost para-complex structure ${\ccK} \in \mathsf{Aut}(TM)$ which is compatible with $\eta$ in the sense that
\begin{align} \label{eqn:conditionparaherm}
    \eta\big({\ccK}(X) , {\ccK} (Y) \big) = - \eta(X,Y) 
\end{align}
for all $X , Y \in \mathsf{\Gamma}(TM).$ The triple $(M, \eta, {\ccK})$ is an \emph{almost para-Hermitian manifold}.

An almost para-Hermitian structure $(\eta, {\ccK})$ is a \emph{para-Hermitian structure} if $L_\pm$ are integrable distributions, and in this case $(M, \eta, {\ccK})$ is a \emph{para-Hermitian manifold}.
\end{definition} 

We will often omit the adjective `almost' for brevity, when no confusion can arise. Equation \eqref{eqn:conditionparaherm} implies that the eigenbundles $L_\pm$ are maximally isotropic with respect to $\eta.$ The splitting $TM = L_+ \oplus L_-$ defines a splitting
\begin{align}
    \midwedge^p T^*M = \bigoplus_{m+n=p}\, \midwedge^{+m,-n}T^*M \ ,
\end{align}
where \smash{$\midwedge^{+m,-n}T^*M = \midwedge^m L_+^* \otimes \midwedge^n L_-^*$}. 

A canonical Courant algebroid may be formed over an almost para-Hermitian manifold $(M,\eta,\ccK)$ in the following way.\footnote{See \cite{Marotta:2022tfe} for details of this construction and \cite{Vaisman2012, Vaisman2013, Freidel2017, Marotta2021born, Marotta:2021sia} for the definition of ${\tt D}$-bracket as well as the proof of this property and examples. See \cite[Section 4]{Svoboda2018} for the general definition of $\tt D$-brackets.} The Levi-Civita connection of the pseudo-Riemannian metric $\eta$ defines a $\tt D$-bracket $[\,\cdot \, , \, \cdot\,]^{\texttt{LC}}$. The para-Hermitian structure $\ccK$ defines a connection $\nabla^{\texttt{can}}$, and hence a canonical $\tt D$-bracket $[\,\cdot \, , \, \cdot\,]^{\texttt{can}}$. 
We can then define a three-form $H^{\texttt{can}}$ on $M$ by
\begin{align}\label{eqn:HcanLC}
    H^{\texttt{can}}(X,Y,Z) = \eta\big([X,Y]^{\texttt{LC}} - [X,Y]^{\texttt{can}} , Z\big) \ ,
\end{align}
for $X,Y,Z\in\sfGamma(TM)$.

For our purposes, the $H$-flux is required to be basic with respect to $L_-$. 
We thus take as our $H$-flux the $(+3,-0)$ component of $H^{\texttt{can}}$. Assuming this is closed, we obtain a canonical Courant algebroid \smash{$\big(\IT M, (H^{\texttt{can}})^{+3,-0}\big)$}. We further have 
\begin{align}\label{eqn:Hcananddomega}
   H:= (H^{\texttt{can}})^{+3,-0} = \de \omega^{+3,-0} \ ,
\end{align}
where the almost symplectic structure $\omega:=\eta\circ\ccK$ is called the \emph{fundamental two-form}.

\begin{remark}[\textbf{Wess-Zumino Functional and Integrability}] \label{rmk:Hcantopological}
The topological term of a sigma-model for a para-Hermitian manifold can be obtained from the lack of integrability of the eigenbundles of $\ccK,$ since 
\begin{align} \label{eq:Hcanintegrab}
  H(X_-,Y_-,Z_-) = \tfrac12\, \bigl( \eta([X_- , Y_-], Z_-) -\eta([X_-,Z_-],Y_-) + \eta(X_-,[Y_-,Z_-]) \bigr)\ ,
\end{align}
for all $X_-, Y_- , Z_- \in \mathsf{\Gamma}(L_-)$. This vanishes if $L_-$ is integrable. With $L_-=T\cF_-$, a foliated almost para-Hermitian manifold $(M, \eta, \ccK, \cF_-)$ gives a topological term for the sigma-model if $H$ is a closed three-form.\footnote{In other words, there exists a flat Ehresmann connection for the fibred manifold $M$ over $\cQ=M/\cF_-,$ see \cite[Section~3.5]{Saunders_1989}.} From now on we assume this is the case. 
\end{remark}

We require a result describing generalised metrics on the tangent bundle $TM$ of a para-Hermitian manifold $(M,\eta,\ccK)$, rather than on the usual double tangent bundle $\IT M$.\footnote{The latter is called the `large Courant algebroid' in~\cite{Jonke2018,Kokenyesi:2018xgj}, where the reductions to T-dual backgrounds are carried out in the language of three-dimensional (membrane) sigma-models.} To distinguish between these two notions of generalised metrics, we rename the former as

\begin{definition}
Let $(M,\eta, \ccK)$ be an almost para-Hermitian manifold. A \emph{generalised para-Hermitian metric} is an automorphism $I\in \mathsf{Aut}(TM)$ covering  $\unit_M$ such that $I^2 = \unit_{TM}$ and which, together with $\eta$, defines a Riemannian metric $\cH$ via
\begin{align*}
    \cH(X, Y) = \eta\big(I(X), Y\big) \ ,
\end{align*}
for $X,Y \in \mathsf{\Gamma}(TM)$.
\end{definition}

This is the counterpart of Definition \ref{defn:generalisedmetric} for tangent bundles of para-Hermitian manifolds. In this case we also have an equivalent formulation in terms of a subbundle $\ccV^+$ of $TM$ which is maximally positive-definite with respect to $\eta.$ This is the $+1$-eigenbundle of $I$ with its orthogonal complement with respect to $\eta$ being the $-1$-eigenbundle.
The similarity extends even further, with the analogous results of Proposition \ref{prop:generalisedmetric} given by~\cite{Marotta2021born}
\begin{proposition}
Let $(M,\eta, \ccK)$ be an almost para-Hermitian manifold, and $L_\pm\subset TM$ the $\pm 1$-eigenbundles of $\ccK$. 
A generalised para-Hermitian metric $I\in \mathsf{Aut}(TM)$ defines a unique pair $(g_+, b_+)$ given by a fibrewise Riemannian metric $g_+ \in \mathsf{\Gamma}(\midodot^2 L_+^*)$ and a two-form $b_+ \in \mathsf{\Gamma}(\midwedge^2 L_+^*)$. 
Conversely, any such pair $(g_+, b_+)$ defines a generalised para-Hermitian metric.
\end{proposition}

\begin{remark}[\textbf{Characterisation of Generalised Para-Hermitian Metrics}]\label{rmk:generalisedpHmetrics}
Denoting by $\eta^\flat \in \mathsf{Hom}(TM,T^*M)$ the musical isomorphism induced by $\eta,$ one can define a metric $g_- \in \mathsf{\Gamma}(\midodot^2 L_-^*)$ by
\begin{align}\label{eqn:gminus}
    g_-(X_-,Y_-) = g_+^{-1}\big( \eta^\flat(X_-), \eta^\flat(Y_-)\big) \ ,
\end{align}
for $X_-, Y_- \in \mathsf{\Gamma}(L_-).$ The two-form $b_+$ defines a map $\gamma_b\colon L_+ \to L_-$ given by
\begin{align*}
    b_+(X_+,Y_+) = \eta\big(\gamma_b(X_+), Y_+\big) \ ,
\end{align*}
for $X_+, Y_+ \in \mathsf{\Gamma}(L_+)$. The resulting Riemannian metric on $M$ in the splitting $TM = L_+ \oplus L_-$ defined by $\ccK$ is then
\begin{align}\label{eqn:generalisedparametric}
    \cH = \left (
    \begin{matrix}
    g_+ + \gamma_b^{\rm t}\,  g_-\, \gamma_b & -\gamma^{\rm t} _b\, g_- \\
    -g_-\, \gamma_b & g_-
    \end{matrix}
    \right ) \ ,
\end{align}
where $\gamma_b^{\rm t} \colon L_-^* \to L_+^*$ is the transpose map. 

It follows that a generalised para-Hermitian  metric is completely determined by the triple $(g_+, b_+, \eta).$ Let us stress that the pair $(g_+,b_+)$ is unique for the given splitting determined by the almost para-complex structure $\ccK.$ A different pair would result in another splitting.
\end{remark}

Equation \eqref{eqn:generalisedparametric} resembles Equation \eqref{eqn:generalisedmetric}, though the background data in the latter are a Riemannian metric $g$ and a two-form $b$ on the tangent bundle $TM$, rather than on the subbundle $L_+$. 

\begin{remark}[\textbf{Bundle-like Generalised Para-Hermitian Metrics}]\label{rmk:pullbackpHmetric}
Assume that $L_-$ is involutive, thereby inducing a foliation $\cF_-$ with smooth leaf space $\cQ = M/\cF_-.$ Then the condition that $g_+\in \mathsf{\Gamma}(\midodot^2 L_+^*)$ and $b_+\in \mathsf{\Gamma}(\midwedge^2 L_+^*)$ are pullbacks of a background metric and Kalb-Ramond field on $\cQ$ requires them to be basic:
\begin{align*}
    \pounds_{X_-}g_+ = 0 \qquad \text{and} \qquad \pounds_{X_-} b_+ = 0 \ ,
\end{align*} 
for every $X_- \in \mathsf{\Gamma}(L_-)$.
Notice that in order for $X_- \in \mathsf{\Gamma}(L_-)$ to be a Killing vector field for $\cH,$ it must be an infinitesimal symmetry of $\eta$ as well:
\begin{align}
    \pounds_{X_-} \eta = 0 \ .
\end{align}

Conversely, starting with a Riemannian metric ${\red g} \in \mathsf{\Gamma} (\midodot^2 T^*\cQ )$ and a two-form ${\red b} \in \mathsf{\Omega}^2(\cQ)$, we can pull these back via the surjective submersion $\varpi\colon M \to \cQ$, so that \smash{$\varpi^*{\red g} \in \mathsf{\Gamma}(\midodot^2 L_+^*)$} and \smash{$\varpi^*{\red b}\in \mathsf{\Gamma}(\midwedge^2 L_+^*)$}, and hence define a generalised para-Hermitian metric $\cH$ via Equation~\eqref{eqn:generalisedparametric}.
\end{remark}

The circumstances under which a diffeomorphism represents an isometry between generalised para-Hermitian metrics is provided by

\begin{proposition}
 Let $(M_1, \eta_1 , \ccK_1)$ and $(M_2, \eta_2 , \ccK_2)$ be almost para-Hermitian manifolds endowed with generalised para-Hermitian metrics $\cH_1$ and $\cH_2,$ respectively. A diffeomorphism $\phi \in \mathsf{Diff}(M_1, M_2)$ is an isometry between $\cH_1$ and $\cH_2$ if and only if $\phi$ is an isometry between $\eta_1$ and $\eta_2$ which intertwines $I_1 \in \mathsf{Aut} (TM_1)$ and $I_2 \in \mathsf{Aut} (TM_2)$:
 \begin{align} \label{eqn:diffgenparametric}
    \phi_* \circ I_1 = I_2 \circ \phi_* \ .  
 \end{align}
\end{proposition}

\begin{proof}
 Assume that $\phi^* \eta_2 = \eta_1$ and Equation \eqref{eqn:diffgenparametric} hold. Then 
 \begin{align}
 \phi^*\cH_2 (X, Y) &= \eta_2 \big(\phi_*(X) , I_2 (\phi_* (Y))\big) = \eta_2\big(\phi_*(X), \phi_*(I_1(Y))\big) 
 =  \cH_1(X,Y)
 \end{align}
 for all $X, Y \in \mathsf{\Gamma}(TM_1),$ where we use Equation \eqref{eqn:diffgenparametric} for the second equality and $\phi^*\eta_2 = \eta_1$ for the third equality.
  Conversely, assuming $\phi^* \cH_2 = \cH_1,$ a similar argument establishes  Equation \eqref{eqn:diffgenparametric} as well as that $\phi$ is an isometry between $\eta_1$ and $\eta_2$.
\end{proof} 

\medskip

\subsubsection{The Reduced Courant Algebroid}~\\[5pt]
Let $(M,\eta, \ccK)$  be an almost para-Hermitian manifold such that the $-1$-eigenbundle $L_-$ is involutive, and the induced foliation $\cF_-$ has smooth leaf space $\cQ = M/ \cF_-$ with smooth surjective submersion $\varpi \colon M \to \cQ$. As discussed in Section~\ref{sssec:parahermmfld}, there is a canonical Courant algebroid $(\IT M, H)$, where $H = (H^{\texttt{can}})^{+3,-0}$ is defined as in Equation~\eqref{eqn:HcanLC} and assumed to be closed. Since $H$ is basic with respect to $\varpi$, the reduction of Theorem \ref{thm:foliationreduction} can be applied to give the Courant algebroid $(\IT \cQ, {\red H})$ over $\cQ$. The double tangent bundle $\IT \cQ=T\cQ\oplus T^*\cQ$ is pointwise isomorphic to $L_+ \oplus L_+^*$.
\medskip

\subsubsection{Generalised T-duality}~\\[5pt]
We now describe \emph{generalised T-duality} for para-Hermitian manifolds~\cite{Marotta2021born} in the language of Section \ref{sec:T-duality}. This is done by first choosing diffeomorphic $2n$-dimensional manifolds $M_1$ and $M_2$, and a diffeomorphism $\phi \in \mathsf{Diff}(M_1,M_2)$. 
Let $(\eta_1, \ccK_1)$  be an almost para-Hermitian structure on $M_1$ such that the $-1$-eigenbundle $L_{1-}$ is involutive, and the induced foliation $\cF_{1-}$ has smooth leaf space $\cQ_1 = M_1/ \cF_{1-}$ with smooth surjective submersion $\varpi_1 \colon M_1 \to \cQ_1$. Let $(M_2, \eta_2)$ be a smooth manifold endowed with a split-signature metric $\eta_2$ such that $\phi^*\eta_2 =\eta_1.$ Choosing a T-duality direction in the quotient tangent space $T\cQ_1$, an almost para-Hermitian structure $(\eta_2, \ccK_2)$ can be constructed on $M_2$ which, provided its $-1$-eigenbundle $L_{2-}$ is involutive and the induced foliation $\cF'_{2-}$ has smooth leaf space, gives the T-dual space (up to diffeomorphisms) $\cQ'_2 = M_2 / \cF'_{2-}$ to $\cQ_1$.

We discuss how this construction is obtained, starting with a local characterisation of the almost para-Hermitian structure $(\eta_1, \ccK_1)$ on $M_1.$
Let $U_1\subset M_1$ be an open subset associated with a coordinate chart, and choose a local frame $\set{Z_I}_{I=1,\dots,2n} = \set{Z_i, \tilde Z^i}_{i= 1,\dots,n}$ which diagonalises $\ccK_1$, i.e. such that $Z_i\in \mathsf{\Gamma}_{U_1}(L_{1+})$ are $+1$-eigenvectors and \smash{$\tilde Z^i \in \sfGamma_{U_1}(L_{1-})$} are $-1$-eigenvectors of $\ccK_1$ at every point in $U_1.$ Denote the dual coframe by \smash{$\set{\Theta^I}_{I=1,\dots,2n} =$} \smash{$ \set{\Theta^i,\tilde \Theta_i}_{i=1,\dots,n}$}. In this frame, the tensors $\ccK_1|_{U_1}$, $\eta|_{U_1}$ and $\omega_1|_{U_1}$ can be written as\footnote{We use the Einstein convention for summation over repeated upper and lower indices throughout.}
\begin{align}
    \ccK_1|_{U_1} = \Theta^i \otimes Z_i-\tilde{\Theta}_i \otimes \tilde{Z}^i \quad , 
    \quad \eta_1|_{U_1}=\Theta^i \otimes \tilde{\Theta}_i+\tilde{\Theta}_i \otimes \Theta^i 
    \quad , \quad \omega_1|_{U_1}=\Theta^i \wedge \tilde{\Theta}_i \ .
\end{align}

In the chart $U_1$, the frame fields 
$\set{Z_I}_{I=1,\dots,2n}$ close a Lie algebra $$[Z_I,Z_J] = C_{IJ}{}^K\,Z_K \ . $$ In the splitting $TM_1 = L_{1+} \oplus L_{1-}$ this 
becomes the Lie algebra
\begin{align}\label{eqn:strucfunc}
    \begin{aligned}
& {[Z_i, Z_j]=f_{ij}{ }^k\, Z_k+H_{ijk}\, \tilde{Z}^k} \ , \\[4pt]
& {[Z_i, \tilde{Z}^j]=f_{k i}{ }^j\, \tilde{Z}^k+Q_i{ }^{jk}\, Z_k} \ , \\[4pt]
& {[\tilde{Z}^i, \tilde{Z}^j]=Q_k{ }^{ij}\, \tilde{Z}^k+R^{ijk}\, Z_k} \ .
\end{aligned}
\end{align}
That $L_{1-}$ is integrable is equivalent to $R^{ij k}=0$ for all $i,j,k = 1,\dots,n$. Lemma \ref{lemma:structureconsts} below tells us that $Z_i$ are $\varpi_1$-projectable if and only if $Q_i{}^{j k}=0$.

\begin{remark}
    In string theory the local structure functions in Equation~\ref{eqn:strucfunc} are called \emph{generalised fluxes}. In particular, $f_{ij}{}^k$ and $H_{ijk}$ are known as \emph{geometric fluxes}, while $Q_i{}^{jk}$ and $R^{ijk}$ are \emph{non-geometric fluxes}. Only when $R^{ijk}$ vanishes is a reduction to a smooth quotient space $\cQ_1=M_1/\cF_{1-}$ possible.

    The condition that $Q_i{}^{j k}$ vanishes tells us that $L_{1-}$ is locally abelian. Note that the chart diagonalising $\ccK_1$ need not be a chart adapted to the foliation $\cF_{1-}$, in which we could write \smash{$\tilde{Z}^i = \tfrac{\del}{\del \tilde{x}_i}$} where $\tilde{x}_i$ are coordinates adapted to the leaves. Conversely, if the diagonalising chart is adapted to the foliation $\cF_{1-}$, then \smash{$Q_i{}^{jk}$} vanish. That is, even though $L_{1-}$ is always locally abelian (by looking in an adapted chart), we require that the (possibly not adapted) chart diagonalising $\ccK_1$ is also locally abelian.
    
    Note also that for the procedure to work, we do not require that the local diagonalising frame fields $Z_i$ and dual one-forms $\Theta^i$ are basic sections, since the only requirement for reduction of an $H_1$-twisted standard Courant algebroid is that $H_1$ is basic, see Example \ref{eg:twistedCAreduction}. Thus the condition that $Q_i{}^{jk}$ vanish could be relaxed, though in that case the resulting relations $Q(L_{i-})$ and $R$ would have a more complex and less insightful description. Thus, for sake of simplicity, we keep this assumption.
\end{remark}

Assuming $Q_i{}^{j k}=0$, the vector fields $\red Z{}_i \coloneqq \varpi_{1*}(Z_i)$ give local coordinates for the chart $\red U{}_1 \coloneqq \varpi_1(U_1)\subset \cQ_1$. Fixing $d\in \set{1,\dots,n}$, define $D_1|_{\red U{}_1} = {\rm Span}\set{\red Z{}_1,\dots,\red Z{}_d}$. This can be done in a neighbourhood of every point in $M_1$, and so gives a subbundle $D_1\subset T\cQ_1$ which will become the T-duality directions, as in Definition \ref{def:tdualitydirections}.
On $M_1$ this gives the local frame for $U_1$ as\footnote{From now on, upper case Latin indices run from $1$ to $2n$, lower case Latin indices run from $1$ to $n$, lower case Greek letters run from $d+1$ to $n$, and mirrored lower case Greek letters run from $1$ to $d$.}
\begin{align}
    \set{Z_{\um},Z_\mu, \tilde Z^{\un}, \tilde Z^\nu} \ , \quad \um,\un \in \set{1,\dots,d} \ , \ \mu,\nu \in \set{d+1,\dots, n} \ .
\end{align}

To obtain an almost para-Hermitian structure on $M_2$, notice that $(\phi^{-1})^* \eta_1$ defines a split-signature metric on $M_2$. In the coordinate chart $U_2 \coloneqq \phi(U_1)$, define the vectors
\begin{align}\label{eqn:howZsswap}
    Z_\um' \coloneqq \phi_*( \tilde Z^\um) \quad , \quad  \tilde Z'^\un \coloneqq \phi_*(Z_\un) \quad , \quad 
    Z_\mu' \coloneqq \phi_*( Z_\mu ) \quad , \quad  \tilde Z'^\nu \coloneqq \phi_*(\tilde Z^\nu) \ ,
\end{align}
which give the respective dual one-forms $\Theta'^\um$, $\Theta'^\mu$, $\tilde \Theta'_\un$ and $\tilde \Theta'_\nu$ whose pullbacks satisfy similar equations to those of Equation \eqref{eqn:howZsswap}. The almost para-Hermitian structure on $U_2$ is then given by
\begin{align}
    \ccK_2|_{U_2} = \Theta'^i \otimes Z'_i-\tilde{\Theta}'_i \otimes \tilde{Z}'^i \quad , 
    \quad \eta_2|_{U_2}=\Theta'^i \otimes \tilde{\Theta}'_i+\tilde{\Theta}'_i \otimes \Theta'^i 
    \quad , \quad \omega_2|_{U_2}=\Theta'^i \wedge \tilde{\Theta}'_i \ .
\end{align}
Since we can make this construction in the neighbourhood of every point, we get an almost para-Hermitian structure $(\eta_2, \ccK_2)$ on $M_2$ for which \smash{$\set{Z'_I}_{I=1,\dots,2n} = \set{Z'_i,\tilde Z'^i}_{i=1,\dots,n}$} is a diagonalising frame for $\ccK_2$ in the coordinate chart $U_2$. 

We thus see that $\phi^*\eta_2 = \eta_1$, while $\phi^*\ccK_2 \neq \ccK_1$. The pullback $\phi^* \ccK_2$ defines an almost para-complex structure $\ccK_1'$ on $M_1$ compatible with $\eta_1$, and hence gives a new splitting $TM_1 = L_{1+}' \oplus L_{1-}'$. On the coordinate chart $U_1$, define
\begin{align}
    L_1^{++}|_{U_1} = L_{1+}|_{U_1} \cap L_{1+}'|_{U_1} \quad , \quad L_1^{+-}|_{U_1} = L_{1+}|_{U_1} \cap L_{1-}'|_{U_1} \ ,\\[4pt]
    L_1^{-+}|_{U_1} = L_{1-}|_{U_1} \cap L_{1+}'|_{U_1} \quad , \quad L_1^{--}|_{U_1} = L_{1-}|_{U_1} \cap L_{1-}'|_{U_1} \ .
\end{align}
Notice that \smash{$L_1^{++}|_{U_1} = {\rm Span}\set{Z_\mu}_{\mu=d+1,\dots,n}$}, and hence it patches into a subbundle $L_1^{++}$. Similarly $L_1^{+-}$, $L_1^{-+}$ and $L_1^{--}$ are subbundles of $TM_1$. In particular, $L_1^{--}$ has constant rank.

Lemma \ref{lemma:structureconsts} below gives the conditions on the local structure functions for the chosen frame such that $L_{2-}$, the $-1$-eigenbundle of $\ccK_2$, is integrable and hence defines a regular foliation $\cF'_{2-}$. Assuming this holds, and assuming the quotient $\cQ_2 \coloneqq M_1 / \cF_{2-}$, where $\cF_2= \phi^{-1}(\cF_2')$, is smooth with smooth surjective submersion $\varpi_2 \colon M_1 \to \cQ_2$, we obtain the candidate T-dual manifold $\cQ_2$. We set $C \coloneqq \set{(\varpi_1(m_1), \varpi_2(m_1)) \, | \, m_1 \in M_1}$. As discussed in Section \ref{ssec:topoligcalTduality}, this will be the support of our T-duality relation. Theorem \ref{thm:topoligicaltd} requires that $C$ is smooth, which we assume, and that $L_1^{--}$ has constant rank, which we have already established.

To obtain a T-duality relation $R \colon \IT\cQ_1 \dashrightarrow \IT \cQ_2$ supported on $C$, in light of Example \ref{sec:topogicalTdualitystandardCA} we construct a Courant algebroid isomorphism $\Phi \coloneqq \overline\phi  \circ \e^B\colon \IT M_1 \to \IT M_2$ covering $\phi$. Thus we need an appropriate $B$-field. In para-Hermitian geometry, a $B$-field transformation from $L_{1+}$ to $L_{1-}$ is induced by a two-form $B_+\in \sfGamma(\midwedge^2 L_{1+}^*)$. The fundamental two-form $\omega_1$ and the canonical three-form $H_1^{\texttt{can}}$ map to
\begin{align}
    \omega_{B_+} = \omega_1 - 2\,B_+ \qquad \text{and} \qquad H_{B_+}^{\texttt{can}} = H_1^{\texttt{can}} - \de B_+
\end{align}
by Equation \eqref{eqn:HcanLC} and \cite[Proposition 5.9]{Svoboda2018}, while the split-signature metric $\eta_1$ is preserved.

Following this, we use the $\eta_1$-compatible para-complex structures $\ccK_1$ and $\ccK_1' \coloneqq \phi^* \ccK_2$ on $M_1$, with respective fundamental two-forms $\omega_1$ and $\omega_1' \coloneqq \phi^* \omega_2$, to define
\begin{align}
    B = \tfrac 12\,(\omega_1- \omega_1') \ .
\end{align}
Locally this is given by
\begin{align}
    B|_{U_1} = \Theta^\um \wedge {\tilde \Theta}_\um \ .
\end{align}

The construction of T-duality relation $R$ realised through
    \begin{proposition}\label{prop:paraTdualrelation}
    Suppose that the structure functions for the local Lie algebra \eqref{eqn:strucfunc} satisfy
    \begin{align}\label{eqn:SFconditions}
        R^{ijk}=0 \ , \ Q_i{}^{jk}=0 \ , \ f_{\um \un}{}^i=0 \ , \ f_{\um\mu}{}^i=0 \ , \ H_{\um\un\ahpla}=0 \ , \ H_{i \um \mu} = H_{i \mu \um}
    \end{align}
    for every $i,j,k=1,\dots,n$, $\um, \un, \ahpla = 1,\dots,d$ and $ \mu = d+1,\dots,n$.
Suppose moreover that $H_1$ is closed.
    Then the T-duality relation
    \begin{equation}
    \begin{tikzcd}
        R \colon (\IT \cQ_1, \red H{}_1) \arrow[dashed]{r} & (\IT \cQ_2 , \red H{}_2)
    \end{tikzcd}
    \end{equation}
    exists, where, in the setting of Example \ref{sec:topogicalTdualitystandardCA}, $\Phi \coloneqq \overline{\phi} \circ \e^{B}\,$ with $B = \frac 12\,(\omega_1 - \phi^* \omega_2)$ and $\overline{\phi} = \phi_* + (\phi^{-1})^*$, while $H_2$ is locally given by
    \begin{align}
        H_2|_{\phi(U_1)} = \tfrac 12\, f_{\mu\nu}{}^\um\, \delta_{\um \un}\, \Theta'^\mu \wedge \Theta'^\nu \wedge \Theta'^\un +
        \tfrac 12\, H_{\mu\nu \alpha}\, \Theta'^\mu \wedge \Theta'^\nu \wedge \Theta'^\alpha \ .
    \end{align}
\end{proposition}

Before constructing the Courant algebroid relation $R$, we describe the conditions on the structure functions through

\begin{lemma}\label{lemma:structureconsts}
    Let $U_1 \subset M_1$ be any coordinate chart.
    \begin{enumerate}[label=\arabic{enumi})]
        \item The $-1$-eigenbundle of $\ccK_1$ is integrable if and only if the structure functions of Equation~\eqref{eqn:strucfunc} satisfy
        \begin{align}\label{eqn:SFRiszero}
            R^{ijk} = 0 \ , 
        \end{align}   
        for $i,j,k=1,\dots,n$.
        \item\label{item:Ziprojectable} If Equation~\eqref{eqn:SFRiszero} holds, the diagonalising local sections $Z_i$ for $\ccK_1$ with $i=1,\dots,n$ are projectable with respect to $\varpi_1$ if and only if 
        \begin{align}\label{eqn:SFQiszero}
            Q_i{}^{jk} = 0 \ , 
        \end{align}
        for $i,j,k=1,\dots,n$.
        This also ensures that $\Theta^i$ is basic with respect to $\varpi_1$ for $i=1,\dots,n$, and that the bisubmersion condition \eqref{eqn:bisubmersioncd} is satisfied.
        \item\label{item:L2-integrable} If Equations~\eqref{eqn:SFRiszero} and \eqref{eqn:SFQiszero} hold, the $-1$-eigenbundle of $\ccK_2$ is integrable if and only if
        \begin{align}\label{eqn:SFfHarezero}
            f_{\um\un}{}^\mu = 0 \qquad \text{and} \qquad H_{\um \un \ahpla} = 0 \ , 
        \end{align}
        for all $\um, \un, \ahpla=1,\dots,d$ and $\mu,\nu = d+1,\dots,n$.
        \item\label{item:strucH2basic} If Equations~\eqref{eqn:SFRiszero}--\eqref{eqn:SFfHarezero} hold and the three-form $H_1$ is closed, the three-form $H_2 \coloneqq (\phi^{-1})^*(H_1 - \de B)$ is basic if and only if
        \begin{align}
            \begin{gathered}\label{eqn:SFmorezero}
            f_{\um\mu}{}^\nu = 0 \quad , \quad f_{\um \un}{}^\ahpla = 0 \quad , \quad
            H_{\mu \um\un}=H_{\mu \un\um} \quad , \quad H_{\mu \um \nu} = H_{\mu \nu \um} \ ,
        \end{gathered}
        \end{align}
        for all $\um, \un, \ahpla=1,\dots,d$ and $\mu,\nu = d+1,\dots,n$.
    \end{enumerate}
\end{lemma}

\begin{proof}
    We have already mentioned the first item. The projectability of the local sections $Z_i$ and the bisubmersion condition from item~\ref{item:Ziprojectable}, as well as item~\ref{item:L2-integrable}, can be read off immediately from Equation \eqref{eqn:strucfunc}. For example, in a local coordinate chart $U_1$, the vector fields $Z_i$ are projectable with respect to $\varpi_1$ if and only if $[X,Z_i] \in \sfGamma_{U_1}(L_{1-})$ for every $X\in \sfGamma_{U_1}(L_{1-})$, which happens if and only if $[{\tilde Z}^j, Z_i] \in \sfGamma_{U_1}(L_{1-})$ for every $i,j =1,\dots,n$, hence $Q_i{}^{jk} =0$.

    To show that $\Theta^i$ are basic, we consider the Maurer-Cartan equations for the coframe $\set{\Theta^I}_{I=1,\dots,2n}$:
    \begin{align}
        \de \Theta^I = -\tfrac{1}{2}\, C_{JL}{}^{I}\, \Theta^J \wedge \Theta^L \ .
    \end{align}
    In the splitting $TM_1 = L_{1+} \oplus L_{1-}$ this reads
    \begin{align}\label{eqn:maurercartan}
        \begin{aligned}
            \de \Theta^i&=-\tfrac{1}{2}\,\big(f_{j k }{ }^i\, \Theta^j \wedge \Theta^k+R^{ijk}\, \tilde{\Theta}_j \wedge \tilde{\Theta}_k\big)-Q_k{ }^{j i}\, \Theta^k \wedge \tilde{\Theta}_j \ , \\[4pt]
            \de \tilde{\Theta}_i & = f_{ij}{ }^k\, \tilde{\Theta}_k \wedge \Theta^j-\tfrac{1}{2}\,\big(Q_i{ }^{jk}\, \tilde{\Theta}_j \wedge \tilde{\Theta}_k+H_{ijk}\, \Theta^j \wedge \Theta^k\big) \ .
        \end{aligned}
    \end{align}
    The one-form $\Theta^i$ is basic with respect to $\varpi_1$, for $i=1,\dots,n$, if and only if $\iota_{{\tilde Z}^j} \Theta^i = 0$ and $\pounds_{{\tilde Z}^j}\Theta^i = 0$ for each $j=1,\dots,n$. The first condition holds by definition, thus assuming Equations~\eqref{eqn:SFRiszero} and \eqref{eqn:SFQiszero} we get
    \begin{align}
        \pounds_{{\tilde Z}^j}\Theta^i = \iota_{{\tilde Z}^j}\,\de \Theta^i +\de\,\iota_{{\tilde Z}^j}\Theta^i =-\tfrac12\,f_{k l}{}^i\,\iota_{{\tilde Z}^j}\, \big(\Theta^k \wedge \Theta^l\big) = 0 \ .
    \end{align}

    To show item \ref{item:strucH2basic}, recall that $B = \Theta^\um \wedge {\tilde \Theta}_\um$, and since \smash{$H_1 = \de\omega_1^{+3,-0}$} one must find the constraints such that $(\phi^{-1})^*(H_1 - \de B)$ is basic, i.e. \smash{$(\phi^{-1})^*(H_1 - \de B) \in {\sfGamma}(\midwedge^{+3,-0} T^*M_2).$}
    After a lengthy calculation, one arrives at the conditions \eqref{eqn:SFmorezero}.
\end{proof}

\begin{proof}[\textbf{Proof of Proposition \ref{prop:paraTdualrelation}}]
Let $K_1 = L_{1-} \oplus \set{0}$ and $K_2 = \Phi^{-1}(L_{2-} \oplus \set{0})$. Since the three-form \smash{$H_1 = \de \omega_1^{+3,-0}$} is closed, the assumption $\iota_X H_1 =0,$ for all $X \in \sfGamma(L_1)$ of Proposition \ref{prop:splitcaseconditions} is satisfied. 

The condition $\iota_{X}H_2 = 0,$ for all $X \in \sfGamma(L_2)$ of Proposition \ref{prop:splitcaseconditions} is satisfied under the assumption of Equation~\eqref{eqn:SFconditions}, and since $H_1$ is closed it follows that $H_2\coloneqq (\phi^{-1})^*(H_1 - \de B)$ is closed. Hence the Courant algebroids $(\IT M_1, H_1)$ and $(\IT M_2, H_2)$ are isomorphic via $\Phi$.

Finally, $\ker(B|_{T\cF_{1-} \cap (T\cF_{2-})}) = T\cF_{1-} \cap T\cF_{2-} = L_1^{--}$ has constant rank $n-d$, and hence all the assumptions of Proposition \ref{prop:splitcaseconditions} are satisfied.

Applying Proposition \ref{prop:splitcaseconditions} we thus obtain a Courant algebroid relation supported on $C = \set{(q_1,q_2)\in\cQ_1\times\cQ_2 \, | \, q_1 = \varpi_1(m_1) \ , \ q_2 = \varpi_2(m_1) \ , \ m_1 \in M_1}$ which, at a point $(q_1,q_2) \in C$, is given by
\begin{align}
    R_{(q_1, q_2)} = \text{Span}\set{({\red Z}{}_\um, {\red \Theta}'^\um) \,,\, ({\red Z}{}_\mu, {\red Z}{}'_\mu) \,,\, ({\red \Theta}^\un, {\red Z}'_\un) \,,\, ({\red \Theta}^\nu, {\red \Theta}'^\nu)} \ \subset \ \IT_{q_1} \cQ_1 \times \overline{\IT_{q_2} \cQ_2} \ ,
\end{align}
as required.
\end{proof}

From now on, we assume that the conditions of Proposition~\ref{prop:paraTdualrelation} are satisfied. Having constructed the T-duality relation $R$, we show that it gives a generalised isometry when appropriate background fields are taken on $\IT \cQ_1$ and $\IT \cQ_2$. The first step is to provide conditions under which the $B$-field has the required symmetry, which we establish through

\begin{lemma}\label{lemma:paraBinvariance}
    $\pounds_X B =0$ for each $X \in \set{Z_\um, {\tilde Z}^\mu}^{\mu = d+1,\dots, n}_{\um = 1,\dots,d}$ if and only if $H_{\um \un \mu} = H_{\um \mu \un}$ for each $\um, \un = 1,\dots,d$ and $\mu = d+1,\dots, n$.
\end{lemma}

\begin{proof}
    The proof again follows from the Maurer-Cartan equations \eqref{eqn:maurercartan} and Cartan calculus.
\end{proof}

Assuming the conditions of Lemma~\ref{lemma:paraBinvariance} from now on, we then find

\begin{corollary}
    In a coordinate chart $U_1$, the three-form $H_1$ is given by
    \begin{align} \label{eqn:H1U1}
        H_1|_{U_1} = \tfrac 12\, H_{i \mu \nu }\, \Theta^{i} \wedge \Theta^{\mu} \wedge \Theta^{\nu} \ .
    \end{align}
\end{corollary}

Equation~\eqref{eqn:H1U1} is the analog of~\cite[Equation~(2.5)]{cavalcanti2011generalized}.

Take a generalised metric $V_1^+$ on $\IT \cQ_1$ specified by $({\red g}{}_1, {\red b}{}_1)$. Locally, we suppose that 
\begin{align}\label{eqn:paraD1invmetric}
    \red g_1 = \red g_{i j} \red \Theta^i \otimes \red \Theta^j \, , \quad \red b_1 = \red b_{i j} \red \Theta^i \wedge \red \Theta^j \, \, \text{ such that } \quad \pounds_{\red Z_\um} \red g_{ij} = \pounds_{\red Z_\um} \red b_{ij} = 0.
\end{align}
Hence we define $\mathsf{Iso}(V_1^+) \coloneqq \mathrm{Span}_{\IR}\set{Z_1,\dots,Z_d}$, so that $V_1^+$ is $D_1$-invariant. The two-form $B$ decomposes as in Equation~\eqref{eqn:Bdecomposition} (with $B_{\rm hor}=B_{\rm ver}=0$), and by Lemma \ref{lemma:paraBinvariance}, the conditions of Theorem \ref{thm:mainsplitcase} are satisfied. Thus 
\begin{corollary}
    With $V_1^+$ given by \eqref{eqn:paraD1invmetric}, there is a unique background $V_2^+$ on $\IT \cQ_2$ given by $({\red g}{}_2, {\red b}{}_2)$ such that $R$ is a generalised isometry between $V_1^+$ and $V_2^+$.    
\end{corollary}
While there is always a way to construct the unique generalised metric $V_2^+$, with the extra structure a para-Hermitian manifold carries there is cleaner way to find it. 

Remarks \ref{rmk:generalisedpHmetrics} and \ref{rmk:pullbackpHmetric} show that we get a generalised para-Hermitian metric $\cH_1$ on $TM_1$. Similarly, the T-dual background $({\red g}{}_2, {\red b}{}_2)$ on $\IT \cQ_2$ gives a generalised para-Hermitian metric $\cH_2$ on $TM_2$. Since $\phi$ is an isometry between the split-signature metrics $\eta_1$ and $\eta_2,$ we can show that it is also an isometry between $\cH_1$ and $\cH_2$ through

\begin{proposition}\label{lemma:tdualgenparametric}
    Let $R \colon (\IT \cQ_1, \red H{}_1, ({\red g}{}_1, {\red b}{}_1)) \dashrightarrow (\IT \cQ_2, \red H{}_2, ({\red g}{}_2, {\red b}{}_2))$ be the generalised isometry coming from Theorem \ref{thm:maingeneral} applied to the $\eta$-isometric almost para-Hermitian manifolds $(M_1,\eta_1,\ccK_1)$ and $(M_2,\eta_2,\ccK_2)$. Then  $\phi^*\cH_2 = \cH_1$.
\end{proposition}

\begin{proof}
    Assume, by absorbing the two-form ${\red b}{}_1$ into ${\red H{}_1}$ if necessary, that the background on $\IT \cQ_1$ is given solely by the Riemannian metric ${\red g}{}_1$. Using Equation~\eqref{eqn:generalisedparametric}, the generalised para-Hermitian metric $\cH_1$ thus takes the form
    \begin{align}\label{eqn:paraHmetricdiagonal}
    \cH_1 =
    \begin{pmatrix}
    g_{1+}  & 0 \\
    0 & g_{1-}
    \end{pmatrix} \ .
\end{align}
The four-fold splitting $$TM_1 = L_1^{+-} \oplus L_1^{++} \oplus L_1^{--} \oplus L_1^{-+}$$ gives a decomposition of $g_{1+}$ into four parts: 
\begin{align}
    g_1^{\pmpm}\in \sfGamma\big(\midotimes^2 (L_1^{+-})^*\big) \quad & , \quad g_1^{\pmpp}\in \sfGamma\big((L_1^{+-})^*\otimes(L_1^{++})^*\big) \ , \\[4pt]
    g_1^{\pppm} \in \sfGamma\big((L_1^{++})^*\otimes (L_1^{+-})^*\big) \quad & , \quad g_1^{\pppp}\in \sfGamma\big(\midotimes^2 (L_1^{++})^*\big)   \ ,
\end{align}
 where $g_1^{\pppm} = \big(g_1^{\pmpp}\big)^{\rm t} = g_1^{\pmpp}$.
 Locally this decomposition is given explicitly by
\begin{align}
    g_{1+}|_{U_1} = (g_{1+})_{\um\un}\,\Theta^\um \otimes \Theta^\un + (g_{1+})_{\um\mu}\, \Theta^\um \otimes \Theta^\mu + (g_{1+})_{\mu \nu}\,  \Theta^\mu \otimes \Theta^\nu \ ,
\end{align}
with $(g_{1+})_{\um\un} = \big(g_1^{\pmpm}\big)_{\um\un}$, $(g_{1+})_{\um\mu} = \big(g_1^{\pppm}\big)_{\um\mu} $ and $(g_{1+})_{\mu\nu} = \big(g_1^{\pppp}\big)_{\mu\nu}$. 

The T-dual generalised metric \smash{$V_2^+$} is determined by the reduction of the basic tensors $g_{2+}\in \sfGamma\big(\midotimes^2 (L_{2+})^*\big)$, with components \smash{$g_2^{\pmpm}$}, \smash{$g_2^{\pmpp}$} and \smash{$g_2^{\pppp}$}, and $b_{2+}\in \sfGamma\big(\midwedge^2 (L_{2+})^*\big)$, with components \smash{$b_2^{\pmpm}$}, \smash{$b_2^{\pmpp}$} and \smash{$b_2^{\pppp}$} where \smash{$b_2^{\pppm} = \big(b_2^{\pmpp}\big)^{\rm t} = -b_2^{\pmpp}.$} Since $B = \Theta^\um \wedge \tilde\Theta_\um$, by a calculation similar to that of Remark \ref{rmk:buscherrules} this can be written locally as
\begin{align}
    g_{2+}|_{U_1} &= (g_{2+})_{\um\un} \, \Theta'^\um \otimes \Theta'^\un + (g_{2+})_{\um\mu} \, \Theta'^\um \otimes \Theta'^\mu + (g_{2+})_{\mu \nu} \, \Theta'^\mu \otimes \Theta'^\nu \ , \\[4pt]
    b_{2+}|_{U_1} &= (b_{2+})_{\um\un}\,\Theta'^\um \wedge \Theta'^\un + (b_{2+})_{\um\mu} \,\Theta'^\um \wedge \Theta'^\mu + (b_{2+})_{\mu \nu}\, \Theta'^\mu \wedge \Theta'^\nu \ ,
\end{align}
where 
\begin{align}\label{eqn:parabuscherrules}
\begin{split}
    (g_{2+})_{\um\un} = \phi_*\big(g_1^{\pmpm}\big)^{-1}_{\um\un} & \quad , \quad (g_{2+})_{\um\mu} = 0 \ , \\[4pt] 
    (g_{2+})_{\mu\nu} = \phi_*(g_{1+})_{\mu\nu} - \phi_*\big( & (g_{1+})_{\ahpla\mu} \,(g_{1+})^{\ahpla\ateb}\,(g_{1+})_{\nu\ateb}\big) \ , \\[4pt]
    (b_{2+})_{\um\mu} = 0 \quad , \quad (b_{2+})_{\um\mu} = \, \phi_*\big( & (g_{1+})_{\mu \flipm{\alpha}} \,(g_{1+})^{\flipm{\alpha} \um}\big) \quad , \quad (b_{2+})_{\mu \nu} = 0 \ ,
\end{split}
\end{align}
and $(g_{1+})^{\ahpla\ateb} = \big(g_1^{\pmpm}\big)^{-1}_{\ahpla \ateb}$.

Conversely, the generalised para-Hermitian metric \eqref{eqn:paraHmetricdiagonal} is given by
\begin{align}\label{eqn:riemmetric4x4}
    \cH_1 =
    \begin{pmatrix}
    g_1^{\pmpm} & g_1^{\pppm} & 0 & 0 \\[3pt]
    g_1^{\pmpp} & g_1^{\pppp} & 0 & 0 \\[3pt]
    0 & 0 & g_1^{\mpmp} & g_1^{\mmmp} \\[3pt]
    0 & 0 & g_1^{\mpmm} & g_1^{\mmmm}
    \end{pmatrix} \ ,    
\end{align}
where
\begin{align}
    g_1^{\mmmm} &= \bigl(g_1^{\pppp} - g_1^{\pmpp}\, \bigl(g_1^{\pmpm}\bigr)^{-1} \,g_1^{\pppm}\bigr)^{-1} \ , \\[4pt]
    g_1^{\mmmp} &= -\big(g_1^{\pmpm}\big)^{-1}\, g_1^{\pppm}\, g_1^{\mmmm} \ ,\\[4pt]
    g_1^{\mpmp} &= \big(g_1^{\pmpm}\big)^{-1} + \big(g_1^{\pmpm}\big)^{-1}\, g_1^{\pppm}\, g_1^{\mmmm}\, g_1^{\pmpp}\, \big(g_1^{\pmpm}\big)^{-1} \ ,
\end{align}
and $g_1^{\mpmm} = \big(g_1^{\mmmp}\big)^{\rm t}$.
This comes from the formula for the inverse of a $2\times 2$ matrix.

Applying $(\phi^{-1})^*$ to Equation~\eqref{eqn:riemmetric4x4} yields
\begin{align}\label{eqn:riemmetric4x42}
    (\phi^{-1})^*\cH_1 =
    \begin{pmatrix}
    (\phi^{-1})^*g_1^{\mpmp} & 0 & 0 & (\phi^{-1})^*g_1^{\mmmp} \\[3pt]
    0 & (\phi^{-1})^*g_1^{\pppp} & (\phi^{-1})^*g_1^{\pmpp} & 0 \\[3pt]
    0 & (\phi^{-1})^*g_1^{\pppm} & (\phi^{-1})^*g_1^{\pmpm} & 0 \\[3pt]
    (\phi^{-1})^*g_1^{\mpmm} & 0 & 0 & (\phi^{-1})^*g_1^{\mmmm}
    \end{pmatrix} \ ,
\end{align}
This is in the form of Equation~\eqref{eqn:generalisedparametric}, and hence  
we can read off the metric $g_{2+}$ and two-form $b_{2+}$ on $L_{2+}$ as
\begin{align}
    g_{2+} &=     
    \begin{pmatrix}
        (\phi^{-1})^*\big(g_1^{\pmpm}\big)^{-1} & 0 \\[3pt]
        0 & (\phi^{-1})^*\big(g_1^{\mmmm}\big)^{-1}
    \end{pmatrix} \ , \\[4pt]
    b_{2+} &= 
    \begin{pmatrix}
        0 & - (\phi^{-1})^*\big(g_1^{\pmpm}\big)^{-1}\, g_1^{\pppm} \\[3pt]
        (\phi^{-1})^*g_1^{\pppm}\, \big(g_1^{\pmpm}\big)^{-1} & 0
    \end{pmatrix} \ .
\end{align}
Looking in a coordinate chart $\phi(U_1)=U_2\subset M_2$, we find that this agrees with Equation~\eqref{eqn:parabuscherrules}.
\end{proof}

Equation~\eqref{eqn:parabuscherrules} gives the component form of the Buscher rules for generalised T-duality in para-Hermitian geometry.

\medskip

\subsection{Doubled Nilmanifolds}\label{ssec:doubledtwistedtorus}~\\[5pt]
In this final section we consider the class of string theory compactifications provided by {nilmanifolds}, which are quotients of nilpotent Lie groups by a discrete cocompact subgroup.
We focus on the particular example of the three-dimensional Heisenberg nilmanifold $\mathsf{N}(m,k)$, which is a quotient of the Heisenberg group $\mathsf{H}$ of upper triangular $3\times3$ matrices whose diagonal entries are all equal to $1$; geometrically it defines a circle bundle of degree $m\in\IZ$ over a two-torus $\sfT^2$ with $H$-flux representing the \v{S}evera class labelled by $k\in\IZ.$ This is analogous to the lens space $\mathsf{L}(m,k)$ from Example~\ref{ex:lensspaces}, which is a quotient of $\sfS^3\simeq\mathsf{SU}(2)$ by a cyclic subgroup $\IZ_m$ of the isometry group $\mathsf{SU}(2)$ of the round three-sphere. In particular, it can be similarly treated in the correspondence space picture of Section~\ref{ssec:correspondencespace}, giving a topology changing T-duality between $\mathsf{N}(m,k)$ and $\mathsf{N}(k,m)$.

Here we consider this T-duality relation instead in the framework of Section~\ref{ssec:parahermitian} by compactifying the Drinfel'd double $T^*\mathsf{H} = \sfH \ltimes \IR^3$, the cotangent bundle of the Heisenberg group $\sfH$, by the action of a discrete cocompact subgroup which defines a doubled nilmanifold~\cite{Hull2009,ReidEdwards2009, Chaemjumrus:2019fsw}. We illustrate this explicitly for the simplest case with $k=0$, where $\mathsf{N}(m,0)$ is the nilmanifold $\sfT^\sfH$ of degree $m$ without $H$-flux and $\mathsf{N}(0,m)$ is the three-torus $\mathsf{T}^3$ with \v{S}evera class $m$. We shall study two foliations of the doubled Heisenberg nilmanifold, and obtain a T-duality between the respective quotients in sense of Section \ref{ssec:parahermitian}.

\medskip

\subsubsection{The Drinfel'd Double of the Heisenberg Group}~\\[5pt]
The doubled Heisenberg nilmanifold is constructed by considering a quotient of the Drinfel'd double $\sfD^\sfH$ of the three-dimensional Heisenberg group $\sfH$  given by $\sfD^\mathsf{H} = T^* \mathsf{H} = \mathsf{H}\ltimes \IR^3,$ with Lie algebra $\frd = \frh 
 \ltimes \IR^3$, by a discrete cocompact subgroup $\sfD^\mathsf{H}(\IZ)$. 
Here the Heisenberg algebra $\mathfrak{h}={\sf Lie}(\sfH)$ and the abelian Lie algebra $\mathbb{R}^3$ together with $\mathfrak{d}=\mathfrak{h}\ltimes \mathbb{R}^3$ form a Manin triple. 
Despite the fact that $\sfH$ is not semi-simple, we can still give a matrix representation for the Lie algebra of the Drinfel'd double $T^*\sfH$. This is useful for explicitly writing down the coordinate identifications defining the global structure of the doubled nilmanifold. 
 
 In local coordinates $(x,y,z, \tilde x, \tilde y, \tilde z)\in\sfH\times\IR^3$, an element $\gamma$ in $\sfD^\mathsf{H}$ may be written as
\begin{align*}
    \gamma=\left(\begin{array}{cccccc}
1 & m\, x & y & 0 & 0 & \tilde{z} \\
0 & 1 & z & 0 & 0 & -\tilde{y} \\
0 & 0 & 1 & 0 & 0 & 0 \\
0 & -m\, \tilde{y} & \tilde{x}-m\, z\, \tilde{y} & 1 & m\, x & y+\frac{1}{2}\, m\, \tilde{y}^2 \\
0 & 0 & 0 & 0 & 1 & z \\
0 & 0 & 0 & 0 & 0 & 1
\end{array}\right) \ .
\end{align*}
Then the left-invariant one-forms on $\sfD^\sfH$ are the Lie algebra components of
the corresponding Maurer-Cartan one-form $\Theta = \gamma^{-1}\, \de \gamma$ given by
\begin{align}
\begin{gathered}\label{eqn:nilfold1forms}
\Theta^{x}=\mathrm{d} x \quad , \quad 
\Theta^{y}=\mathrm{d} y-m\, x\, \mathrm{d} z \quad , \quad
\Theta^{z}=\mathrm{d} z \ ,\\[4pt]
\tilde{\Theta}_{x}=\mathrm{d} \tilde{x}-m\, z\, \mathrm{d} \tilde{y} \quad , \quad 
\tilde{\Theta}_{y}=\mathrm{d} \tilde{y} \quad , \quad 
\tilde{\Theta}_{z}=\mathrm{d} \tilde{z}+m\, x\, \mathrm{d} \tilde{y} \ .
\end{gathered}
\end{align}

The dual left-invariant vector fields are therefore
\begin{align}
\begin{gathered}\label{eqn:nilfoldvectors}
Z_x = \frac{\partial}{\partial x} \quad , \quad 
Z_y = \frac{\partial}{\partial y} \quad , \quad 
Z_z = \frac{\partial}{\partial z} + m\,x\,\frac{\partial}{\partial y} \ ,\\[4pt]
\tilde Z^x = \frac{\partial}{\partial \tilde x} \quad , \quad 
\tilde Z^y = \frac{\partial}{\partial \tilde y} + m\,z\, \frac{\partial}{\partial \tilde x} - m\,x\, \frac{\partial}{\partial \tilde z} \quad , \quad 
\tilde Z^z = \frac{\partial}{\partial \tilde z} \ .
\end{gathered}
\end{align}
The nilpotent Lie algebra $\frd$ of $T^*\sfH=\sfH\ltimes\IR^3$ thus has non-vanishing brackets
\begin{align}
[Z_x, Z_z]= m\, Z_y \ , \quad [Z_x, \tilde{Z}^y]=- m\, \tilde Z^z \qquad
\mbox{and} \qquad [Z_z, \tilde{Z}^y]= m\, \tilde{Z}^x \ ,  \label{eq:lieth}
\end{align}
and in particular the only non-vanishing structure function is $f_{xz}{}^y=-f_{zx}{}^y=m.$

The Lie algebra $\frd$ is naturally endowed with a para-Hermitian structure which induces a left-invariant para-Hermitian structure $(\eta^\frd_1,\ccK^\frd_1)$ on $\sfD^\sfH$ given by
\begin{align}\label{eqn:parahermitiannilfold}
    \ccK^\frd_1 = \Theta^i \otimes Z_i - {\tilde \Theta}_i \otimes {\tilde Z}^i \quad , \quad 
    \eta^\frd_1 = \Theta ^i \otimes \tilde \Theta_i + \tilde \Theta_i \otimes \Theta^i \quad , \quad 
    \omega^\frd_1 = \Theta^i \wedge {\tilde \Theta}_i \ .
\end{align}
This makes $(\sfD^{\sfH}, \eta^\frd_1, \ccK^\frd_1)$ a para-Hermitian manifold. 
The  eigenbundles of $\ccK^\frd_1$ are then the integrable distributions given by $L^\frd_+$ spanned by the left-invariant vector fields $\set{Z_x,Z_y,Z_z}$ induced by the Lie subalgebra $\frh$ whose leaves are all given by the Heisenberg group $\mathsf{H}$, and $L^\frd_-$ spanned by $\set{\tilde Z^x,\tilde Z^y,\tilde Z^z}$ coming from the generators of $\IR^3$ whose leaves are just $\IR^3$.
The global frame \smash{$\set{Z_i,{\tilde Z}_j}$ for $i,j\in\set{x,y,z}$} defines a diagonalising frame for the para-Hermitian structure~$(\eta_1^\frd , \ccK^\frd_1)$. 

\medskip

\subsubsection{The Doubled Heisenberg Nilmanifold}~\\[5pt]
To compactify $\sfD^\mathsf{H}$, we consider the left action by a discrete cocompact subgroup $\sfD^\mathsf{H}(\IZ)$ whose generic element $\xi$ is given by
\begin{align*}
    \xi=\left(\begin{array}{cccccc}
1 & m\, \alpha & \beta & 0 & 0 & \tilde{\delta} \\
0 & 1 & \delta & 0 & 0 & -\tilde{\beta} \\
0 & 0 & 1 & 0 & 0 & 0 \\
0 & -m\, \tilde{\beta} & \tilde{\alpha}-m\, \delta\, \tilde{\beta} & 1 & m \alpha & \beta+\frac{1}{2}\, m\, \tilde{\beta}^2 \\
0 & 0 & 0 & 0 & 1 & \delta \\
0 & 0 & 0 & 0 & 0 & 1
\end{array}\right) \ ,
\end{align*}
where $\alpha, \beta, \delta , \tilde{\alpha}, \tilde \beta, \tilde \delta \in\IZ$. Under the identification $\gamma \sim \xi \, \gamma$, we get the simultaneous identifications of coordinates as
\begin{align}\label{eqn:nilfoldidentification}
\begin{gathered}
x \sim x+\alpha \quad , \quad y \sim y+m\, \alpha\, z+\beta \quad , \quad z \sim z+\delta \ , \\[4pt]
\tilde{x} \sim \tilde{x}+m\, \delta\, \tilde{y}+\tilde{\alpha} \quad , \quad \tilde{y} \sim \tilde{y}+\tilde{\beta} \quad , \quad \tilde{z} \sim \tilde{z}-m\, \alpha\, \tilde{y}+\tilde{\delta} \ ,
\end{gathered}
\end{align}
which defines the doubled Heisenberg nilmanifold $M^\mathsf{H} = \sfD^\mathsf{H}(\IZ)\setminus\sfD^\mathsf{H}$.

The left-invariant one-forms \eqref{eqn:nilfold1forms} and left-invariant vector fields \eqref{eqn:nilfoldvectors} are invariant under the identifications \eqref{eqn:nilfoldidentification}, and hence descend globally through the quotient. Thus the distributions $L^\frd_+$ and $L^\frd_-$ descend respectively to integrable distributions $L_{1\pm} \subset T M^\sfH$. Their leaves are, respectively, the Heisenberg nilmanifold $\sfT^\mathsf{H} = \mathsf{H}(\IZ)\setminus\mathsf{H}$ and the three-torus~$\sfT^3= \IR^3/ \IZ^3$. 

The para-Hermitian structure \eqref{eqn:parahermitiannilfold} is left-invariant, hence invariant under the left discrete action. Thus it descends to a para-Hermitian structure $( \eta_1,\ccK_1)$ on $M^\sfH$, with integrable 
$\pm 1$-eigenbundles $L_{1\pm}$ and globally defined diagonalising frame  $\set{Z_i,{\tilde Z}^j}$ for $i,j\in\set{x,y,z}$.
Since $L_{1\pm}$ are involutive, Remark \ref{rmk:Hcantopological} tells us that $H_1 = 0$. 

\medskip

\subsubsection{Generalised T-duality}~\\[5pt]
Consider the diffeomorphism of $M_1:=M^\sfH$ given in coordinates by
\begin{align}\label{eqn:tdualdiffeo}
    \phi(x, y, z, \tilde x, \tilde y, \tilde z) = (x, \tilde y, z, \tilde x - m\,z\,\tilde y, y, \tilde z) \eqqcolon (x',y',z',\tilde x', \tilde y', \tilde z') \ ,
\end{align}
with the coordinate identifications
\begin{align}
\begin{gathered}\label{eqn:t3identification}
x' \sim x'+\alpha' \quad ,\quad y' \sim y' + \tilde \beta' \quad ,  \quad  z' \sim z'+\delta' \ , \\[4pt]
\tilde{x}' \sim \tilde{x}'+m\, \delta'\, \tilde y' +m\,z'\,(m\,\alpha'\, z' + \beta') + \tilde \alpha' \quad , \quad \tilde{y}' \sim \tilde{y}'+m\,\alpha'\, z' + \beta' \ ,\\[4pt]
\tilde{z}' \sim \tilde{z}'-m\, \alpha'\, y'+\tilde{\delta}' \ ,
\end{gathered}
\end{align}
where $\alpha', \beta', \delta' , \tilde{\alpha}', \tilde \beta', \tilde \delta' \in\IZ$.

Since the only non-vanishing structure function is $f_{xz}{}^y$, by Proposition \ref{prop:paraTdualrelation} it follows that the only T-duality direction we can consider is in the $y$-direction, that is, $D_1 = \text{Span}\set{\red Z{}_y}$.
The global frame fields on $M_2:=\phi(M^\sfH)$, where the T-duality direction is swapped as in Equation~\eqref{eqn:howZsswap}, are given by
\begin{align}
\begin{gathered}\label{eqn:t3vectors}
Z_x' = \frac{\partial}{\partial x'} \quad , \quad 
Z_y' = \frac{\partial}{\partial y'} - m\,x'\,\frac{\partial}{\partial \tilde z'} \quad , \quad 
Z_z' = \frac{\partial}{\partial z'} -m\,y'\, \frac{\partial}{\partial \tilde x'} + m\,x'\, \frac{\partial}{\partial \tilde y'} \ ,\\[4pt]
\tilde Z'^x = \frac{\partial}{\partial \tilde x'} \quad , \quad 
\tilde Z'^y = \frac{\partial}{\partial \tilde y'} \quad , \quad 
\tilde Z'^z = \frac{\partial}{\partial \tilde z'} \ ,
\end{gathered}
\end{align}
with the non-vanishing Lie brackets
\begin{align}
    [Z_x',Z_z']=m\,\tilde Z'^y \ , \quad [Z_x',Z_y']=-m\,\tilde Z'^z \qquad \text{and} \qquad [Z_z',Z_y'] = m\,\tilde Z'^x \ ,
\end{align}
and dual one-forms
\begin{align}\label{eqn:t31forms}
\begin{gathered}
\Theta^{\prime x} =\de x' \quad , \quad 
\Theta^{\prime y}=\de y' \quad , \quad 
\Theta^{\prime z}=\de z' \ , \\[4pt]
\tilde{\Theta}_{x}^{\prime}=\mathrm{d} \tilde{x}'+m\, y'\, \mathrm{d} z' \quad , \quad 
\tilde \Theta'_y = \de\tilde y' - m\,x'\, \mathrm{d}z' \quad , \quad 
\tilde{\Theta}_{z}^{\prime}=\mathrm{d} \tilde{z}'+m\, x'\, \mathrm{d} y' \ .
\end{gathered}
\end{align}

This is the global diagonalising frame for the almost para-Hermitian structure $(\eta_2, \ccK_2)$ on the doubled nilmanifold $M^\sfH$, with $\phi^* \eta_2 = \eta_1$ as well as
\begin{align}\label{eqn:parahermitiant3}
    \ccK_2 = Z_i' \otimes \Theta'^i - \tilde Z'^i \otimes \tilde \Theta'_i \qquad \text{and} \qquad \omega_2 = \Theta'^i \wedge \tilde \Theta'_i \ .
\end{align}
Note that the $+1$-eigenbundle $L_{2+}$ of $\ccK_2$ is not an integrable distribution.

Taking the quotient by the foliation $\cF_{1-}$ with $L_{1-}=T\cF_{1-}$ leaves the coordinates $(x,y,z)$, which by Equation~\eqref{eqn:nilfoldidentification} parametrise the Heisenberg nilmanifold $\sfT^\sfH$, while the quotient by the foliation $\cF_{2-}$ with $L_{2-}=T\cF_{2-}$ leaves the coordinates $(x',y', z')$, which by Equation~\eqref{eqn:t3identification} parametrise the three-torus $\sfT^3$. 
Thus $M^\sfH$ is the doubled space for $\sfT^\sfH$ and $\sfT^3$ obtained in \cite[Section 7.3]{Marotta2021born}, and hence the respective quotients $M^\sfH/\cF_{1-}=\sfT^\sfH$ and $M^\sfH / \cF_{2-} = \sfT^3$ of $M^\sfH$ by the foliations $\cF_{1-}$ and $\cF_{2-}$ integrating the $-1$-eigenbundles $L_{1-}$ of $\ccK_1$ and $L_{2-}$ of $\ccK_2$ are smooth manifolds. Since $\sfT^{\sfH}$ and $\sfT^3$ can be viewed as circle bundles over $\sfT^2$, Lemma~\ref{lemma:cdsmoothcirclecomp} implies that 
\begin{align}
 C \coloneqq \set{(\varpi_1(m), \varpi_2(\phi(m))) \, | \, m \in M^{\sfH} }=\sfT^{\sfH} \times_{\sfT^{2}} \sfT^{3}
 \end{align}
 is hence smooth.

We are now ready to apply the results of Section \ref{ssec:parahermitian}.
We know that $H_1 = 0$ is closed, and hence we may apply Proposition \ref{prop:paraTdualrelation} to obtain the T-duality relation $R \colon (\IT \sfT^\sfH, 0) \rel (\IT\sfT^3, \red H{}_{\sfT^3})$, where $\Phi=\overline{\phi}\circ \e^{B}\,$ and
\begin{align}
    \red H{}_{\sfT^3} = m\, \de x' \wedge \de y' \wedge \de z' \ .
\end{align}
Since Lemma \ref{lemma:paraBinvariance} says that the two-form
\begin{align}
    B = \Theta^y\wedge\tilde\Theta_y = (\de y-m\,x\,\de z)\wedge\de\tilde y
\end{align}
has the appropriate invariance,\footnote{Note that the $H$-flux $H_2 = (\varphi^{-1})^*\de B = \varpi_2^*\,\red H{}_{\sfT^3}$ is exact on $M^\sfH$. However, this $B$-field is not basic and so does not descend to a two-form on \smash{$\sfT^3$}. On the other hand, $H_2$ is basic and descends to a cohomologically non-trivial three-form $\red H{}_{\sfT^3}$ on $\sfT^3$. That $B$ is not basic is the origin of `topology change' in T-duality.} the relation $$R=\mathrm{Span}\set{(Z_x,Z_x') \,, \, (Z_y,\Theta'^y)\,,\,(Z_z,Z_z')\,,\, (\Theta^x,\Theta'^x)\,,\, (\Theta^y,Z_y')\,,\,(\Theta^z,\Theta'^z)}$$ is a generalised isometry in the following way. 

The standard left-invariant Riemannian metric ${\red g}{}_{\sfT^\sfH}$ on the Heisenberg nilmanifold $\sfT^\sfH$ is given by
\begin{align}
    {\red g}{}_{\sfT^\sfH} = \delta_{ij}\,{\red \Theta}^i \otimes {\red \Theta}^j = \de x \otimes \de x + (\de y - m\,x\, \de z) \otimes (\de y - m\,x\, \de z) + \de z \otimes \de z \ .
\end{align}
With $\mathsf{Iso}(V_1^+) = \mathsf{Lie}(\mathsf{T^H})$, the notion of $D_1$ invariance becomes that of $\mathsf{T^H}$-invariance, and hence ${\red g}{}_{\sfT^\sfH}$ is a $D_1$-invariant metric.

The pullback of ${\red g}{}_{\sfT^\sfH}$ to $M_1$, \smash{$ g_{1+} =\varpi_1^*\,{\red g}{}_{\sfT^\sfH} = \delta_{ij}\,{\Theta}^i \otimes {\Theta}^j$}, defines a generalised para-Hermitian metric $\cH_1 = g_{1+} + g_{1-}$ on \smash{$M^\sfH$} given by
\begin{align}
    \cH_1 = \delta_{ij}\,{\Theta}^i \otimes {\Theta}^j + \delta^{ij}\,\tilde{\Theta}_i \otimes \tilde{\Theta}_j \ .
\end{align}
The pullback by $\phi^{-1}$ gives $\cH_2 = \delta_{ij}\,{\Theta}'^i \otimes {\Theta}'^j + \delta^{ij}\,\tilde{\Theta}'_i \otimes \tilde{\Theta}'_j$, and hence we get a basic tensor $g_{2+} = \delta_{ij}\,{\Theta}'^i \otimes {\Theta}'^j$ which is the pullback by $\varpi_2$ of the standard flat Riemannian metric 
\begin{align}
{\red g}{}_{\sfT^3} = \delta_{ij}\, {\red \Theta}'^i \otimes {\red \Theta}'^j = \de x'\otimes\de x' + \de y'\otimes\de y' + \de z'\otimes\de z'
\end{align}
on the three-torus $\sfT^3$. 

Thus the Heisenberg nilmanifold $\sfT^\sfH$ with zero $H$-flux and background metric \smash{${\red g}{}_{\sfT^\sfH}$} is T-dual to the three-torus $\sfT^3$ with $H$-flux $\red H{}_{\sfT^3}$ and background metric \smash{${\red g}{}_{\sfT^3}$} via the generalised isometry
\begin{equation}
\begin{tikzcd}
    R \colon \big(\IT \sfT^\sfH\,,\, 0\,,\, ({\red g}{}_{\sfT^\sfH},0)\big) \arrow[dashed]{r} & \big(\IT \sfT^3\,,\, \red H{}_{\sfT^3}\,,\, ({\red g}{}_{\sfT^3},0)\big) \ .
    \end{tikzcd}
\end{equation}

\appendix

\section{Change of Splitting for Compositions} \label{app:changesplitting}

The composition of transverse generalised isometries $R_1 \colon (E_1,W_1) \rel (E_2,W_2)$ and $R_2 \colon (E_2,W_2) \rel (E_3,W_3)$ may fail due to the fact that the lifts of $W_2^\pm$ for the decomposition \eqref{eqn:transdecomp} may differ for $R_1$ and $R_2$. 
To see the solution, let us first investigate the case of relations $R_i$ which are graphs of classical Courant algebroid morphisms given by isomorphisms of exact Courant algebroids. Suppose that $\Phi_1\colon E_1\to E_2$ and $\Phi_2\colon E_2 \to E_3$ are Courant algebroid isomorphisms, where $E_i$ are exact Courant algebroids over $M_i$. Let $K_i$ be isotropic involutive subbundles of $E_i$, let $W_i$ be pre-$K_i$-transverse generalised metrics, and assume that $\Phi_i$ is a regular transverse generalised isometry between $W_i$ and $W_{i+1}$. 

Then there exist the bundles $\widetilde W_1^\pm$, $\widetilde W_2^\pm$, $\widetilde W_2'^\pm$ and $\widetilde W_3'^\pm$ such that 
\begin{align}
    \Phi_1\big(\widetilde W_1^\pm\big) = \widetilde W_2^\pm \qquad \text{and} \qquad \Phi_2\big(\widetilde W_2'^\pm\big) = \widetilde W_3'^\pm \ .
\end{align}
We can define the subbundles $\widetilde W_1'^\pm \coloneqq \Phi_1^{-1}\big(\widetilde W_2'^\pm\big)$. If $\widetilde W_1'^\pm$ are also lifts of $W_1^\pm$ to $W_1$, then 
\begin{align}
    \Phi_2 \circ \Phi_1 \big(\widetilde W_1'^\pm\big) = \widetilde W_3'^\pm \ ,
\end{align}
and $\Phi_2 \circ \Phi_1$ is a regular transverse generalised isometry between $W_1$ and $W_3$. If $s^\pm_2 \colon W_2^\pm \to W_2 $ and $s_2'^\pm \colon W_2^\pm \to W_2$ are the splittings with images \smash{$\widetilde{W}_2^\pm$} and \smash{$\widetilde{W}_2'^\pm$} respectively, the requirement that \smash{$\widetilde W_1'^\pm$} are also lifts of $W_1^\pm$ to $W_1$ is that the differences $s^\pm_2 - s_2'^\pm$ map to~$\Phi_1(K_1)$. 
Similarly one could define \smash{$\widetilde{W}_3^\pm \coloneqq \Phi_2\big(\widetilde W_2^\pm\big)$} which, provided that $s^\pm_2 - s_2'^\pm$ map to $\Phi_2^{-1}(K_3)$, define new splittings such that \smash{$\Phi_2 \circ \Phi_1 \big(\widetilde{W}_1^\pm\big) = \widetilde{W}_3^\pm$}.

When $R_i$ are general Courant algebroid relations, the condition on the splittings $s_2^\pm$ and $s_2'^\pm$ may be written as 
\begin{align}
    {\rm pr}_1\big (R_1\cap (E_1 \times \mathrm{im}(s^\pm_2 - s_2'^\pm)\big ) \subset K_1 \qquad \text{or} \qquad {\rm pr}_2\big (R_2\cap ( \mathrm{im}(s_2^\pm - s_2'^\pm) \times E_3)\big ) \subset K_3 \ ,
\end{align}
which should be viewed as pointwise inclusions.
These are similar to the conditions given in Theorem \ref{thm:maingeneral}, though here the splittings $s^\pm_2$ and $s_2'^\pm$ are given \emph{a priori} whereas in the proof of Theorem \ref{thm:maingeneral} we had to construct these splittings. That is, a portion of the proof is showing that $\Phi$ is a transverse generalised isometry. The pointwise nature of the present condition is also in contrast with the construction in Section \ref{ssec:geometricTduality}, which provides a smooth splitting and hence shows that $\Phi$ is a regular transverse generalised isometry.

\bibliographystyle{ourstyle}
\bibliography{bibprova3}

\end{document}